\documentclass[12pt]{amsart}

\usepackage{amsfonts, amstext, amsmath, amsthm, amscd, amssymb, enumitem, url}
\usepackage[table]{xcolor}
\usepackage{tikz-cd}
\usepackage{svg}
\usepackage{pdfpages}
\usepackage{pdflscape}

\usepackage{geometry}
\usepackage[parfill]{parskip}
\usepackage{microtype}

\usepackage[colorlinks,citecolor=cyan,linkcolor=magenta]{hyperref}
\usepackage[nameinlink]{cleveref}
\usepackage{multirow}
\usepackage{subfig}

\usepackage{tabularx}

\newcolumntype{x}{>{\centering\arraybackslash}X}
\usepackage{makecell}
\usepackage{longtable}
\usepackage{hhline}

\usetikzlibrary{arrows.meta}

\newtheorem{thm}{Theorem}[section]
\newtheorem{cor}[thm]{Corollary}
\newtheorem{lemma}[thm]{Lemma}
\newtheorem{prop}[thm]{Proposition}

\newtheorem{conj}[thm]{Conjecture}

\theoremstyle{definition}
\newtheorem{defn}[thm]{Definition}

\newtheorem{constr}[thm]{Construction}
\newtheorem{eg}[thm]{Example}
\newtheorem{quest}[thm]{Question}

\numberwithin{equation}{section}

\usepackage{todonotes}
\renewcommand{\epsilon}{\varepsilon}

\newcommand{\rad}{\mathrm{rad}}

\newcommand{\real}{\mathrm{real}}
\newcommand{\infs}{\mathrm{inf}}

\newcommand{\indPH}{\mathrm{ind}_\mathrm{PH}}
\newcommand{\indL}{\mathrm{ind}_\mathrm{Lef}}

\DeclareMathOperator{\tr}{\mathrm{tr}}
\DeclareMathOperator{\len}{\mathrm{length}}

\newcommand{\graphscale}{0.9}

\allowdisplaybreaks

\begin{document}

\title{Minimum dilatations of pseudo-Anosov braids}

\author{Chi Cheuk Tsang}
\address{Département de mathématiques \\
Université du Québec à Montréal \\
201 President Kennedy Avenue \\
Montréal, QC, Canada H2X 3Y7}
\email{tsang.chi\_cheuk@uqam.ca}

\author{Xiangzhuo Zeng}
\address{Faculté des Sciences d’Orsay, Université Paris-Saclay \\
F-91405 Orsay Cedex}
\email{xiangzhuo.zeng@etu-upsaclay.fr}

\begin{abstract}
We determine the minimum dilatation $\delta_n$ among pseudo-Anosov braids with $n$ strands, for large enough values of $n$. These are the dilatations attained by the examples of Hironaka-Kin and Venzke, and they satisfy $\lim_{n \to \infty} \delta_n^n = (2+\sqrt{3})^2 \approx 13.928$. Together with previous work, this result confirms conjectures by Kin-Takasawa and Venzke, and solves the minimum dilatation problem on the $n$-punctured sphere, for all but $6$ values of $n$.
\end{abstract}

\maketitle

% \setcounter{tocdepth}{3}
% \tableofcontents

\section{Introduction} \label{sec:intro}

Let $S$ be an orientable finite-type surface.
An orientation-preserving homeomorphism $f:S \to S$ is a \textbf{pseudo-Anosov map} if there exists a pair of transverse measured singular foliations $\ell^s$ and $\ell^u$ on $S$ such that $f$ contracts the leaves of $\ell^s$ and expands the leaves of $\ell^u$ by a factor of $\lambda(f) > 1$.
The number $\lambda(f)$ is called the \textbf{dilatation} of $f$, and gives a measure of the dynamical complexity of $f$.

A theorem of Ivanov \cite{Iva88} states that, with the surface $S$ fixed, there are only finitely many conjugacy classes of pseudo-Anosov maps with dilatation bounded above by a given number. 
This allows one to ask the \textbf{minimum dilatation problem}: What is the minimum value of the dilatation among pseudo-Anosov maps on $S$?
In this paper, we address this question in the case when $S$ has genus zero.

Let $D$ be a disc with $n$ marked points. A \textbf{braid} (with $n$ strands) is a union of $n$ embedded arcs in $D \times [0,1]$ that meet $D \times \{0,1\}$ in the marked points and which project to monotone paths in $[0,1]$. A braid determines a homeomorphism on $D$ by sweeping from $D \times \{0\}$ to $D \times \{1\}$ and letting the braid dictate the movement of the marked points.
By puncturing out the $n$ marked points and forgetting about the boundary component, we obtain a homeomorphism on the sphere with $n+1$ punctures. 

Conversely, given a homeomorphism $f$ on a punctured sphere, one can deduce from the Lefschetz fixed point theorem that $f$ must have a fixed puncture or a fixed interior point. The latter case can be reduced to the former case by puncturing out a fixed point. 
We can then obtain a braid by completing a fixed puncture into a boundary component and suspending $f$.

These constructions determine a correspondence between braids and homeomorphisms of punctured spheres. We refer to \cite{Bir74} for more details.
We say that a braid is \textbf{pseudo-Anosov} if its corresponding homeomorphism is a pseudo-Anosov map.
The main theorem of this paper is the following.

\begin{thm} \label{thm:braiddillowerbound}
Let $f$ be a pseudo-Anosov braid with $n$ strands. Then the dilatation of $f$ satisfies
$$\lambda(f) \geq \min \left\{ 14.5^{\frac{1}{n}},
\begin{cases}
|x^{2k+1}-2x^{k+1}-2x^{k}+1| & \text{if $n=2k+1$} \\
|x^{4k}-2x^{2k+1}-2x^{2k-1}+1| & \text{if $n=4k$} \\
|x^{4k+2}-2x^{2k+3}-2x^{2k-1}+1| & \text{if $n=4k+2$}
\end{cases} \right\}.$$
\end{thm}

Here we denote by $|P(x)|$ the largest positive root of a polynomial $P(x)$.

For the rest of this paper, we denote the second item in the minimum in \Cref{thm:braiddillowerbound} by $\underline{\delta}_n$.
We also denote 
\begin{align*}
\underline{N} &= \{n \mid \underline{\delta}_n^n < 14.5 \} \\
&= \{n \mid n=2k+1 \geq 9\} \cup \{n \mid n=4k \geq 16\} \cup \{n \mid n=4k+2 \geq 30\}.
\end{align*}
This is the set of $n$ for which the right hand side in \Cref{thm:braiddillowerbound} equals $\underline{\delta}_n$.

In \cite{HK06}, Hironaka and Kin constructed a sequence of braids $\sigma_{k-1,k+1}$, where each $\sigma_{k-1,k+1}$ has $n=2k+1$ strands and dilatation $\underline{\delta}_n$. In \cite{Ven08}, Venzke extended their work by constructing sequences of braids $\psi_n$ with $n$ strands and dilatation $\underline{\delta}_n$ for $n=4k$ and $n=4k+2$ as well. Together with \Cref{thm:braiddillowerbound}, this implies the following corollary.

\begin{cor} \label{cor:braidmindil}
For $n \in \underline{N}$, the minimum dilatation $\delta_n$ among pseudo-Anosov braids with $n$ strands is $\underline{\delta}_n$.
In particular, $\lim_{n \to \infty} \delta_n^n = (2+\sqrt{3})^2 \approx 13.928$.
\end{cor}

\Cref{cor:braidmindil} answers \cite[Question 4.2]{HK06} in the negative, confirms \cite[Conjectures 5.3 and 5.4]{Ven08} for all $n \in \underline{N}$, confirms \cite[Conjecture 5.10]{Ven08}, confirms \cite[Conjecture 4.1(1)]{KT11}, and answers \cite[Question 1.16(3)]{KT13} in the positive.
On the other hand, note that \cite[Conjecture 5.4]{Ven08} is not true for $n=10$, see \Cref{subsec:remainminbraiddil}.

Under the correspondence between braids and homeomorphisms of punctured spheres, \Cref{thm:braiddillowerbound} and the examples of Hironaka-Kin and Venzke also solve the minimum dilatation problem on $n$-punctured spheres for $n \in \underline{N}$.

\begin{cor} \label{cor:spheremindil}
For $n \in \underline{N}$, the minimum dilatation $\lambda_{0,n}$ among pseudo-Anosov maps on the $n$-punctured sphere is $\underline{\delta}_n$.
In particular, $\lim_{n \to \infty} \lambda_{0,n}^n = (2+\sqrt{3})^2 \approx 13.928$.
\end{cor}

We explain the proof of \Cref{cor:spheremindil} from \Cref{thm:braiddillowerbound} in more detail in \Cref{subsec:pabraid}.

A pseudo-Anosov map on a closed orientable surface is \textbf{hyperelliptic} if it commutes with a hyperelliptic involution. Via taking the quotient under the involution, hyperelliptic pseudo-Anosov maps on the closed genus $g$ surface are in correspondence with pseudo-Anosov maps on the $2g+2$ punctured sphere, see \cite[Chapter 9.4]{FM12} for more details. Consequently, we also have the following corollary.

\begin{cor} \label{cor:hypellipmindil}
For $g \in \{g \mid g=2k+1 \geq 7\} \cup \{g \mid g=2k \geq 14\}$, the minimum dilatation $\eta_g$ among hyperelliptic pseudo-Anosov maps on the closed orientable genus $g$ surface is $\underline{\delta}_{2g+2}$.
In particular, $\lim_{g \to \infty} \eta_g^g = 2+\sqrt{3} \approx 3.732$.
\end{cor}

For the rest of this introduction, we review some previous work, discuss the ideas used in proving \Cref{thm:braiddillowerbound}, and state an addendum to \Cref{thm:braiddillowerbound} which gives information about the braids that attain the minimum dilatation.

\subsection{Previous work}

Recall that we denote by $\delta_n$ the minimum dilatation among pseudo-Anosov braids with $n$ strands and by $\underline{\delta}_n$ the second item in the minimum in \Cref{thm:braiddillowerbound}.

It is a classical fact that $\delta_3 = \underline{\delta}_3 = \frac{3+\sqrt{5}}{2}$.
Ko, Los, and Song showed in \cite{KLS02} that $\delta_4 = \underline{\delta}_4$. Ham and Song showed in \cite{HS07} that $\delta_5 = \underline{\delta}_5$. 
These results use the tool of train track automata. 
A train track map is a combinatorial way of encoding the dynamics of a pseudo-Anosov map. We explain this in more detail in \Cref{subsec:introtraintrack}.
The crucial fact here is that every pseudo-Anosov map is encoded by a closed path in some train track automaton, with maps of smaller dilatation giving shorter paths. 

In both \cite{KLS02} and \cite{HS07}, the authors first construct the suitable train track automata. Then by searching through all paths in the automata that have length bounded by some a priori bound, they recover all the pseudo-Anosov maps with small dilatation and in particular are able to locate the minimum dilatation.

Lanneau and Thiffeault showed in \cite{LT11b} that $\delta_6 = \underline{\delta}_5$, $\delta_7 = \underline{\delta}_7$, and $\delta_8 = \underline{\delta}_8$. (Note that the value for $\delta_6$ is \emph{not} a typo.)
They do this by considering the orientable double cover of a pseudo-Anosov braid, for which the dilatation can be read off from the action on the first homology. The characteristic polynomial of this action must be reciprocal, and the Lefschetz fixed point theorem imposes additional conditions on its coefficients. Using these restrictions, they were able to list out the possible characteristic polynomials and locate the minimum dilatation.

Note that these previous results concern the values of $\delta_n$ for small $n$, while \Cref{cor:braidmindil} gives the values of $\delta_n$ for large $n$.
We compile all these known values of $\delta_n$ in \Cref{tab:braidmindil}. The shaded cells are results from previous work. The unshaded cells that contain values are \Cref{cor:braidmindil}. The remaining cells are labelled `???'.

\begin{table}
    \centering

\renewcommand{\arraystretch}{1.2}
\begin{tabularx}{\textwidth}{|c||x|x|x|x|}
    \hline
    $k$ & $n=4k-1$ & $n=4k$ & $n=4k+1$ & $n=4k+2$ \\
    \hhline{|=#=|=|=|=|}
    1 & \cellcolor{black!10} $\underline{\delta}_3 \approx 2.61803$ & \cellcolor{black!10} $\underline{\delta}_4 \approx 2.29663$ & \cellcolor{black!10} $\underline{\delta}_5 \approx 1.72208$ & \cellcolor{black!10} $\underline{\delta}_5 \approx 1.72208$ \\
    \hline
    2 & \cellcolor{black!10} $\underline{\delta}_7 \approx 1.46557$ & \cellcolor{black!10} $\underline{\delta}_8 \approx 1.41345$ & $\underline{\delta}_9 \approx 1.34372$ & ??? \\
    \hline
    3 & $\underline{\delta}_{11} \approx 1.27248$ & ??? & $\underline{\delta}_{13} \approx 1.22572$ & ??? \\
    \hline
    4 & $\underline{\delta}_{15} \approx 1.19267$ & $\underline{\delta}_{16} \approx 1.18129$ & \multirow{6}{*}{\vdots} & ??? \\
    \cline{1-3} \cline{5-5}
    5 & \multirow{5}{*}{\vdots} & $\underline{\delta}_{20} \approx 1.14192$ & & ??? \\
    \cline{1-1} \cline{3-3} \cline{5-5}
    6 & & \multirow{4}{*}{\vdots} & & ??? \\
    \cline{1-1} \cline{5-5}
    7 & & & & $\underline{\delta}_{30} \approx 1.09309$ \\
    \cline{1-1} \cline{5-5}
    8 & & & & $\underline{\delta}_{34} \approx 1.08144$ \\
    \cline{1-1} \cline{5-5}
    \vdots & & & & \multirow{1}{*}{\vdots} 
\end{tabularx}

    \caption{The minimum dilatation among pseudo-Anosov braids with $n$ strands. The shaded cells are results from previous work \cite{KLS02}, \cite{HS07}, \cite{LT11b}. The entry for $n=6$ is \emph{not} a typo.}
    \label{tab:braidmindil}
\end{table}

A similar table can be compiled for the minimum dilatation $\lambda_{0,n}$ among pseudo-Anosov maps on the $n$-punctured sphere: Reasoning as in the proof of \Cref{cor:spheremindil}, one can deduce from the aforementioned results of \cite{KLS02}, \cite{HS07}, \cite{LT11b} that $\lambda_{0,4}=\underline{\delta}_3$, $\lambda_{0,5}=\lambda_{0,6}=\underline{\delta}_5$, $\lambda_{0,7}=\underline{\delta}_7$, and $\lambda_{0,8}=\underline{\delta}_8$. With \Cref{cor:spheremindil}, the remaining unknown values of $\lambda_{0.n}$ are for $n=10, 12, 14, 18, 22, 26$. 

See \Cref{subsec:remainminbraiddil} and \Cref{subsec:othermindilproblem} for more discussion.

\subsection{Train tracks, curve complexes, and clique polynomials} \label{subsec:introtraintrack}

A \textbf{train track} $\tau$ on a finite-type surface $S$ is an embedded graph where the half-edges incident to each vertex are tangent to some tangent line.
See \Cref{fig:tienbd} top.
A \textbf{train track map} is a map $g:\tau \to \tau'$ that sends vertices to vertices and smooth edge paths to smooth edge paths. 
See \Cref{fig:ttmap} bottom.
The \textbf{transition matrix} of a train track map is the matrix 
$$g_* \in \mathrm{Hom}(\mathbb{R}^{E(\tau)},\mathbb{R}^{E(\tau')})$$
whose $(e',e)$-entry is the number of times $g(e)$ passes through $e'$.

Bestvina and Handel \cite{BH95} showed that every pseudo-Anosov map $f:S \to S$ can be homotoped into some train track map $g:\tau \to \tau$ on a train track $\tau$ on $S$. 
The train track map $g$ encodes the dynamics of $f$ in the following ways:
\begin{itemize}
    \item The spectral radius of the transition matrix $g_*$ equals the dilatation of $f$. 
    \item There is a correspondence between the periodic orbits of $g$ and the periodic orbits of $f$.
\end{itemize}
In this context, we say that $g$ \textbf{carries} $f$.

In \cite{HT22}, Hironaka and the first author showed that if $f$ has at least two singularity orbits, then one can choose $g:\tau \to \tau$ to be a \textbf{standardly embedded train track map}. In this case, by separating the edges of $\tau$ into the \textbf{infinitesimal} and \textbf{real} edges, the transition matrix $g_*$ admits a block decomposition 
$$g_* = \begin{bmatrix}
P & * \\
0 & g_*^\real
\end{bmatrix}$$
where $P$ is a permutation matrix and $g_*^\real$ is a Perron-Frobenius matrix.
In particular, the spectral radius of $g_*^\real$ equals the dilatation of $f$.
Morally, the \textbf{real transition matrix} $g_*^\real$ packages the dynamics of $f$ more efficiently than the entire transition matrix.

The fact that $g_*^\real$ is Perron-Frobenius also means that we can compute its spectral radius as the growth rate of a curve complex.
This theory is developed in \cite{McM15}. We will recall the details in \Cref{sec:growthrate}. For now, here are the key points.

Let $\Gamma$ be the directed graph whose adjacency matrix equals $g_*^\real$.
The fact that $g_*^\real$ is Perron-Frobenius implies that $\Gamma$ is strongly connected.

A \textbf{curve} of $\Gamma$ is an embedded directed edge cycle. 
The \textbf{curve complex} of $\Gamma$ is the graph $G$ whose vertices are the curves of $\Gamma$, and where two vertices are connected by an edge if and only if the corresponding curves of $\Gamma$ are disjoint. For a vertex $v \in V(G)$, we define its \textbf{weight} $w(v)$ to be the length of $v$ as a curve.

The \textbf{right-angled Artin semi-group} associated to $G$ is defined to be
$$A_+(G) = \langle v \in V(G) \mid [v_1,v_2] = 1 \text{ if there is an edge between $v_1$ and $v_2$}\rangle.$$
The function $w:V(G) \to \mathbb{R}_+$ extends to a semi-group homomorphism $w:A_+(G) \to \mathbb{R}_+$.
The \textbf{growth rate} of $(G,w)$ is defined to be 
$$\lambda(G,w) = \lim_{T \to \infty} (\text{\# elements $g \in A_+(G)$ with $w(g) \leq T$})^{\frac{1}{T}}.$$

\cite[Theorem 1.4]{McM15} implies that the spectral radius of $g_*^\real$ equals the growth rate of $(G,w)$.
In turn, the growth rate of $(G,w)$ can be computed by its clique polynomial, which we will define in \Cref{subsec:cliquepoly}.

The advantage of this perspective is that one can estimate the spectral radius of $g_*^\real$ with partial information.
For example, if we know that $f$ has a periodic point with small period, then $g$ will also have a periodic point with small period, which means that $\Gamma$ has a short curve.
In general, the more intersecting short curves there are in $\Gamma$, the more non-commuting generators with small weights there are in $A_+(G,w)$, thus the larger the growth rate will be.
This gives us a way of finding lower bounds for the dilatation of $f$.

\subsection{Floral train tracks} \label{subsec:introfloraltt}

The main innovation in this paper is the observation that, when $f$ is a pseudo-Anosov braid, up to possibly puncturing an extra point, one can further arrange the train track $\tau$ to be \textbf{floral}, i.e. consisting of
\begin{itemize}
    \item one central infinitesimal polygon, which we call the \textbf{pistil},
    \item several 1-pronged infinitesimal polygons, which we call the \textbf{anthers}, 
    \item some real edges that each connect one anther to the pistil, which we call the \textbf{filaments}, and
    \item some real edges each of whose endpoints lie on the same vertex of the pistil, which we call the \textbf{petals}.
\end{itemize}
See \Cref{fig:floraltteg} for an example.

\begin{figure}
    \centering
    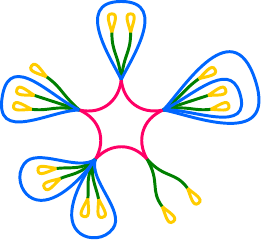
    \caption{An example of a floral train track. Here the pistil is in red, the anthers are in yellow, the filaments are in green, and the petals are in blue.}
    \label{fig:floraltteg}
\end{figure}

Once this is the case, the directed graph $\Gamma$ associated to the real transition matrix $g_*^\real$ satisfies some nice properties. For example:
\begin{itemize}
    \item There is a collection of \textbf{filament curves} such that each vertex of $\Gamma$ that is a filament in $\tau$ passes through one and only one filament curve (\Cref{prop:jointlessttprop}(2)).
    \item Edges that enter filament curves occur in pairs (\Cref{prop:jointlessttprop}(3)).
    \item There is a collection of \textbf{petal curves} such that each vertex of $\Gamma$ that is a petal in $\tau$ passes through one and only one petal curve (\Cref{prop:jointlessttprop}(2)).
    \item Edges that exit petal curves occur in pairs (\Cref{prop:floralttpetalexit}).
\end{itemize}
Intuitively, all these properties say that there are more edges in $\Gamma$ than one would expect in a general setting. Along with the existence of certain short curves which come from periodic points that are guaranteed by the Lefschetz fixed point theorem, this gives us enough control on the curve complex in order to show \Cref{thm:braiddillowerbound}.

However, in reality, the analysis is actually quite tedious, since we have to divide into many cases depending on how the short curves, the filament curves, and the petal curves are positioned with respect to each other.
For example, if the short curves intersect many filament/petal curves, then this already gives a large growth rate. 
If the short curves intersect few filament/petal curves, then since $\Gamma$ is strongly connected, they must be connected to the remaining filament/petal curves via other curves. These connective curves may be long, but by the aforementioned properties, they are plentiful, hence we get large growth rate in those cases as well.

We have chosen to include the details of our computations in only one of the cases in the main paper, namely the case where the values $\underline{\delta}_n$ arise, in \Cref{sec:caseI}. This should give ample demonstration of the ideas involved.
The computations in the other cases are contained in an auxiliary file \texttt{pAbraid\_rootcheck.ipynb}. 
In \Cref{sec:computation}, we explain the content of this file.
It suffices to say here that in each of these cases, we will show that $\lambda^n \geq 14.5$.

\subsection{Braids that attain the minimum dilatation}

After knowing the values of the minimum dilatations $\delta_n$, a natural problem is to find the set of braids that attain the minimum dilatation. We do not answer this problem in this paper, but we make some progress towards it by showing that the braids in question must satisfy certain topological and dynamical properties.

\begin{thm} \label{thm:braiddilequality}
Let $f$ be a pseudo-Anosov braid with $n$ strands, where $n \in \underline{N}$. If $\lambda(f) = \underline{\delta}_n$, then:
\begin{itemize}
    \item The action of $f$ on the strands has a single orbit, i.e. the closure of the braid is a knot.
    \item Every strand is a punctured 1-pronged singular point of $f$.
    \item Aside from the strands, there are exactly two other singular points of $f$, which are 
    $$\begin{cases}
    \text{$(k+1)$- and $k$-pronged} & \text{if $n=2k+1$,} \\
    \text{$(2k+1)$- and $(2k-1)$-pronged} & \text{if $n=4k$,} \\
    \text{$(2k+3)$- and $(2k-1)$-pronged} & \text{if $n=4k+2$.}
    \end{cases}$$
    \item The action of $f$ on the collection of half-leaves at each singular point has a single orbit.
\end{itemize}
\end{thm}

\Cref{thm:braiddilequality} will drop out of our proof: In our computations we will see that the value $\underline{\delta}_n$ is attained when the directed graph $\Gamma$ is of a very specific form, where in particular the lengths and quantities of the filament curves and petal curves mentioned in \Cref{subsec:introfloraltt} are determined. 
From this information, one can deduce the items in \Cref{thm:braiddilequality}.

In fact, we expect that with more work, the knowledge of the directed graph $\Gamma$ will allow one to exactly recover the list of pseudo-Anosov braids that attain equality. See \Cref{subsec:mindilequality} for more discussion on this.

\subsection*{Organization} 
In \Cref{sec:pamap} we recall some background material on pseudo-Anosov maps. 
In \Cref{sec:growthrate} we recall the theory of clique polynomials and curve complexes. 
In \Cref{sec:sett}, we recall the theory of standardly embedded train tracks, expositing in parallel an upgrade which plays a role in defining floral train tracks. 

In \Cref{sec:floraltt}, we develop the theory of floral train tracks. We show that every pseudo-Anosov braid, up to possibly puncturing an extra point, is carried by a floral train track map, and we record some properties of floral train track maps.

In \Cref{sec:mainthmproof}, we lay out the casework for the computations in showing the main theorems, and explain the details of the computation in one of the cases. 

In \Cref{sec:questions}, we discuss some future directions coming out of this paper.

In \Cref{sec:computation}, we explain the code which we wrote for performing computations. We also walk through the computations for the remaining cases outlined in \Cref{sec:mainthmproof}. The code and the actual computations are contained in the auxiliary file \texttt{pAbraid\_rootcheck.ipynb}.

\subsection*{Acknowledgements}

The first author thanks Eriko Hironaka for introducing him to the subject of pseudo-Anosov maps with small dilatations. 
We thank Ian Agol for pointing out \Cref{cor:hypellipmindil}. We thank Eiko Kin for suggesting the addition of \Cref{subsec:examples}. We also thank Michael Landry and Curtis McMullen for comments on an earlier version of the paper.

This project started under the CRM-ISM undergraduate summer research program, and while the first author is a CRM-ISM postdoctoral fellow based at CIRGET. We thank the center and the institute for their support.
The second author was also supported by a Science Undergraduate Research Award at McGill University.

\section{Pseudo-Anosov maps} \label{sec:pamap}

In this section, we recall some facts about pseudo-Anosov maps.
We refer to \cite[Exposés 9-13]{FLP79} for more details.

\subsection{General definitions} \label{subsec:pamapdefn}

In this paper, a \textbf{finite-type surface} is an oriented closed surface with finitely many points, which we call the \textbf{punctures}, removed. 

A orientation-preserving homeomorphism $f$ on a finite-type surface $S$ is said to be a \textbf{pseudo-Anosov map} if there exists a pair of singular measured foliations $(\ell^s,\mu^s)$ and $(\ell^u,\mu^u)$ such that:
\begin{enumerate}
    \item Away from a finite set of \textbf{singular points}, which includes the punctures, $\ell^s$ and $\ell^u$ are locally conjugate to the foliations of $\mathbb{R}^2$ by vertical and horizontal lines respectively.
    \item Near a singular point $x$, $\ell^s$ and $\ell^u$ are locally conjugate to either
    \begin{itemize}
        \item the pull-back of the foliations of $\mathbb{R}^2$ by vertical and horizontal lines, respectively, by the map $z \mapsto z^{\frac{n}{2}}$, for some $n \geq 3$, or
        \item the pull-back of the foliations of $\mathbb{R}^2 \backslash \{0\}$ by vertical and horizontal lines, respectively, by the map $z \mapsto z^{\frac{n}{2}}$, for some $n \geq 1$.
    \end{itemize}
    In either case, we say that $x$ is \textbf{$n$-pronged}.
    See \Cref{fig:pasing} for examples.
    \item $f_* (\ell^s,\mu^s) = (\ell^s,\lambda^{-1} \mu^s)$ and $f_* (\ell^u,\mu^u) = (\ell^u,\lambda \mu^u)$ for some $\lambda=\lambda(f)>1$.
\end{enumerate} 
We refer to $(\ell^s,\mu^s)$ and $(\ell^u,\mu^u)$ as the \textbf{stable} and \textbf{unstable} measured foliations respectively. 
We refer to $\lambda(f)$ as the \textbf{dilatation} of $f$.

\begin{figure}
    \centering
    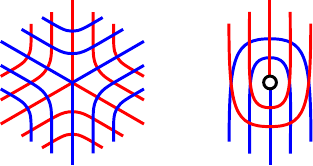
    \caption{Left: An unpunctured 3-pronged singular point. Right: A punctured 1-pronged singular point.}
    \label{fig:pasing}
\end{figure}

The \textbf{Poincaré-Hopf index} of a singular point $x$ is defined to be
\begin{equation} \label{eq:indPH}
\indPH(x) = \begin{cases}
1-\frac{n}{2} & \text{if $x$ is $n$-pronged and is unpunctured,} \\
-\frac{n}{2} & \text{if $x$ is $n$-pronged and is punctured.}
\end{cases}
\end{equation}
Here we add the letters PH to avoid confusion with the Lefschetz fixed point index, which we define later in this section.

\begin{thm}[Poincaré-Hopf] \label{thm:poincarehopf}
Let $f:S \to S$ be a pseudo-Anosov map. Then 
$$\chi(S) = \sum_x \indPH(x)$$
where the sum is taken over all singular points.
\end{thm}

Having certain singular points lie in $S$ while others being punctures makes it awkward to phrase some definitions. To uniformize the notation, we define the \textbf{completion} of $f$ to be the homeomorphism $\overline{f}:\overline{S} \to \overline{S}$ obtained by filling in every puncture of $S$.

However, note that $\overline{f}$ is in general not a pseudo-Anosov map, since we might have filled in some 1-pronged punctures. 
Nevertheless, the stable and unstable foliations on $S$ can be extended into singular measured foliations $(\overline{\ell}^s, \overline{\mu}^s)$ and $(\overline{\ell}^u,\overline{\mu}^u)$ on $\overline{S}$ respectively, and we still have $\overline{f}_* (\overline{\ell}^s,\overline{\mu}^s) = (\overline{\ell}^s,\lambda^{-1} \overline{\mu}^s)$ and $\overline{f}_* (\overline{\ell}^u,\overline{\mu}^u) = (\overline{\ell}^u,\lambda \overline{\mu}^u)$ where $\lambda$ is the dilatation of $f$.

Putting this definition into use, we define a \textbf{stable/unstable half-leaf} to be the image of a proper embedding of $[0,\infty)$ into a leaf of $\overline{\ell}^{s/u}$, respectively.

\begin{prop}[{\cite[Proposition 9.6]{FLP79}}] \label{prop:halfleafdense}
Every stable/unstable half-leaf is dense.
\end{prop}

We define a point $x \in \overline{S}$ to be \textbf{periodic} if it is fixed by $\overline{f}^p$ for some $p \geq 1$. The smallest possible value for $p$ is the \textbf{period} of $x$.
The set $\{\overline{f}^k(x) \mid k \in \mathbb{Z}/p\}$ is the \textbf{periodic orbit} containing $x$.
The \textbf{period} of a periodic orbit is the period of any of its elements.

For example, since $\overline{f}$ acts as a permutation on the finite set of singular points, each singular point is periodic.
We define a \textbf{singularity orbit} to be a periodic orbit that consists of singular points.

Suppose the stable leaf containing $x$ is the union of $n$ half-leaves $\overline{\ell}^s_1, ..., \overline{\ell}^s_n$ meeting at $x$. In this case, we say that $x$ is \textbf{$n$-pronged}. This agrees with our previous terminology when $x$ is singular; if $x$ is nonsingular, it is $2$-pronged.
The smallest possible value of $P$ such that $\overline{f}^P$ fixes $x$ and each of the half-leaves $\overline{\ell}^s_i$ is the \textbf{rotationless period} of $x$.
We always have $p \mid P$ and $P \mid pn$.

Finally, we say that a periodic orbit is \textbf{$n$-pronged} if every element in it is $n$-pronged.
The \textbf{rotationless period} of a periodic orbit is the rotationless period of any of its elements.

\begin{prop} \label{prop:periodicpointsuniqueleaves}
A leaf of $\overline{\ell}^{s/u}$ contains at most one periodic point.
\end{prop}
\begin{proof}
Suppose $\overline{\ell}$ is a leaf of $\overline{\ell}^s$ that contains two periodic points $x_1$ and $x_2$ of period $p_1$ and $p_2$ respectively. Take a path $\alpha$ within $\overline{\ell}$ between $x_1$ and $x_2$. Then for every $n \geq 0$, $\overline{f}^{np_1p_2}$ fixes $x_1$ and $x_2$, thus $\overline{f}^{np_1p_2}(\alpha)$ is a path between $x_1$ and $x_2$ contained in a leaf of $\overline{\ell}^s$.
However, $\overline{\mu}^u(\overline{f}^{np_1p_2}(\alpha)) = \lambda^{-np_1p_2} \overline{\mu}^u(\alpha) \to 0$ as $n \to \infty$, so the $\overline{\mu}^u$-length of $\overline{f}^{np_1p_2}(\alpha)$ converges to $0$, and $x_1 = x_2$.
\end{proof}

\subsection{Lefschetz fixed point theorem} \label{subsec:lefschetz}

Let $\overline{f}$ be a homeomorphism on a closed surface. Let $x$ be an isolated fixed point of $\overline{f}$. Let $\nu$ be a small closed neighborhood of $x$ that contains no other fixed points.
Choose an orientation-preserving homeomorphism between $\nu$ and a neighborhood in $\mathbb{R}^2$, so that it makes sense to talk about the difference between two points in $\nu$ as an element in $\mathbb{R}^2$. 
The \textbf{Lefschetz fixed point index} of $\overline{f}$ at $x$, which we denote by $\indL(\overline{f},x)$, is defined to be the degree of the map $\partial \nu \to S^1$ that takes $y$ to $\frac{f(y)-y}{||f(y)-y||}$.

\begin{thm}[Lefschetz] \label{thm:lefschetz}
Let $\overline{f}$ be a homeomorphism on a closed surface $\overline{S}$ where every fixed point is isolated. Then 
$$\sum_{k=0}^2 (-1)^k \tr (\overline{f}_*:H_k(\overline{S}) \to H_k(\overline{S})) = \sum_x \indL(\overline{f},x)$$
where the sum on the right hand side is taken over all fixed points of $\overline{f}$.
\end{thm}

We specialize to the setting where $\overline{f}$ is the completion of a pseudo-Anosov map $f: S \to S$. 
In this case, every fixed point of $\overline{f}$ is isolated, and the Lefschetz fixed point index can be computed from the action of $\overline{f}$ on the stable and unstable foliations as follows:

Suppose $x$ is a $n$-pronged fixed point. Then the stable leaf containing $x$ is the union of $n$ half-leaves $\overline{\ell}^s_1, ..., \overline{\ell}^s_n$ meeting at $x$. We say that $x$ is \textbf{unrotated} if $\overline{f}$ preserves each $\overline{\ell}^s_i$, otherwise we say that $x$ is \textbf{rotated}. In particular if $n=1$ then $x$ must be unrotated. 

The Lefschetz fixed point index of $\overline{f}$ at $x$ is given by
\begin{equation} \label{eq:indL}
\indL(\overline{f},x) = 
\begin{cases}
1-n & \text{if $x$ is $n$-pronged and is unrotated,} \\
1 & \text{if $x$ is $n$-pronged and is rotated.}
\end{cases}
\end{equation}

\subsection{Pseudo-Anosov braids} \label{subsec:pabraid}

We define a \textbf{pseudo-Anosov braid} with $n$ strands to be a pseudo-Anosov map $f:S \to S$ where $S$ is a sphere with $n+1$ punctures and where $f$ fixes one of the punctures $b$. 
Under the usual way of representing braids, one thinks of $S$ as the interior of a disc with $n$ punctures, with the boundary of the disc corresponding to the fixed puncture $b$.

\Cref{thm:lefschetz} simplifies for pseudo-Anosov braids.

\begin{prop} \label{prop:lefschetzbraid}
Let $\overline{f}$ be the completion of a pseudo-Anosov braid. Then 
$$\sum_x \indL(\overline{f},x) = 2$$
where the sum is taken over all fixed points of $\overline{f}$.
\end{prop}
\begin{proof}
Since $\overline{f}$ is an orientation-preserving homeomorphism on the sphere $S^2$, we have
$$\sum_{k=0}^2 (-1)^k \tr (\overline{f}_*:H_k(\overline{S}) \to H_k(\overline{S})) = 1 - 0 + 1 = 2$$
on the left hand side of \Cref{thm:lefschetz}.
\end{proof}

We illustrate one application of \Cref{prop:lefschetzbraid} by showing how \Cref{cor:spheremindil} follows from \Cref{thm:braiddillowerbound} and the examples of Hironaka-Kin and Venzke, as promised in the introduction.

\begin{proof}[Proof of \Cref{cor:spheremindil} assuming \Cref{thm:braiddillowerbound}]
Let $f$ be a pseudo-Anosov map on the $n$-punctured sphere $S$, and let $\overline{f}$ be the completion of $f$. 

From \Cref{prop:lefschetzbraid} we know that $\sum_x \indL(\overline{f},x) = 2$.
From \Cref{eq:indL}, we observe that the contribution of every fixed point is at most $1$, thus there are at least two fixed points. Let $b$ be one of them.

If $b$ is a puncture of $S$, then $f$ is a pseudo-Anosov braid with $n-1$ strands, thus $\lambda(f) \geq \delta_{n-1}$.
If $b$ is not a puncture of $S$, we let $f^\circ$ be the map obtained from $f$ by puncturing out $b$. Then $f^\circ$ is a pseudo-Anosov braid with $n$ strands, thus $\lambda(f) = \lambda(f^\circ) \geq \delta_n$.

Combining these two cases, we have $\lambda(f) \geq \min \{\delta_{n-1},\delta_n\} \geq \min \{14.5^{\frac{1}{n}}, \underline{\delta}_{n-1}, \underline{\delta}_n \}$ by \Cref{thm:braiddillowerbound}. 
A straightforward computation shows that $\underline{\delta}_{n-1} > \underline{\delta}_n$ for $n \neq 6,10$. Thus we conclude that $\lambda(f) \geq \underline{\delta}_n$ for $n \in N$.

Meanwhile, the braids $\sigma_{k-1,k+1}$ with $n=2k+1$ strands in Hironaka-Kin \cite{HK06} that have dilatation $\underline{\delta}_n$ have a $k$-pronged singular point at the fixed puncture. For $n \in N$, we have $k \geq 4$, thus we can fill in the puncture to get a pseudo-Anosov map on the $n$-punctured sphere that also has dilatation $\underline{\delta}_n$.

Similarly, the braids $\psi_n$ with $n=4k$ or $n=4k+2$ strands in Venzke \cite{Ven08} that have dilatation $\underline{\delta}_n$ have a $2k-1$-pronged singular point at the fixed puncture. For $n \in N$, we have $k \geq 4$, thus we can fill in the puncture to get a pseudo-Anosov map on the $n$-punctured sphere that also has dilatation $\underline{\delta}_n$.
\end{proof}

\section{Clique polynomials and curve complexes} \label{sec:growthrate}

In this section, we recall the theory of clique polynomials and curve complexes as described in \cite{McM15}. 

\subsection{Clique polynomials} \label{subsec:cliquepoly}

A \textbf{weighted graph} is a (simple, undirected) graph $G$ together with a function $w:V(G) \to \mathbb{R}_+$, which we refer to as the \textbf{weight} of the vertices.

Given a weighted graph $(G,w)$, the \textbf{right-angled Artin semi-group} associated to $G$ is defined to be
$$A_+(G) = \langle v \in V(G) \mid [v_1,v_2] = 1 \text{ if there is an edge between $v_1$ and $v_2$}\rangle.$$
The function $w:V(G) \to \mathbb{R}_+$ extends to a semi-group homomorphism $w:A_+(G) \to \mathbb{R}_+$.
The \textbf{growth rate} of $(G,w)$ is defined to be 
$$\lambda(G,w) = \lim_{T \to \infty} (\text{\# elements $g \in A_+(G)$ with $w(g) \leq T$})^{\frac{1}{T}}.$$

A \textbf{clique} of $G$ is a subset $K \subset V(G)$ where every pair of vertices in $K$ is connected by an edge. In particular the empty set is always a clique. The \textbf{size} of $K$, which we denote by $|K|$, is the number of vertices in $K$. The \textbf{weight} of $K$, which we denote by $w(K)$, is the sum of the weights of the vertices in $K$.

The \textbf{clique polynomial} of $(G,w)$ is defined to be 
$$Q_{(G,w)}(t) = \sum_K (-1)^{|K|} t^{w(K)}$$
where the sum is taken over all cliques.

\begin{thm}[{\cite[Theorem 4.2]{McM15}}] \label{thm:cliquepolycomputesgrowthrate}
The smallest positive zero of $Q_{(G,w)}(t)$ equals $\frac{1}{\lambda(G,w)}$.
\end{thm}

We record some properties of the growth rate $\lambda(G,w)$ as a function of the graph $G$ and the weight function $w$.

Recall that $H$ is an \textbf{induced subgraph} of $G$ if $V(H) \subset V(G)$, and there is an edge between $v_1$ and $v_2$ in $H$ if and only if there is an edge between $v_1$ and $v_2$ in $G$. This is sometimes referred to as a full subgraph in the literature.

We say that $H$ is a \textbf{wide subgraph} of $G$ if $V(H)=V(G)$ and there is an edge between $v_1$ and $v_2$ in $H$ only if there is an edge between $v_1$ and $v_2$ in $G$.

\begin{prop} \label{prop:growthrateprop}
\leavevmode
\begin{enumerate}
    \item (Homogeneity) For any $\alpha \in \mathbb{R}_+$, we have $\lambda(G,\alpha w) = \lambda(G,w)^{\frac{1}{\alpha}}$.
    \item (Monotonicity) If $w \leq w'$, then $\lambda(G,w) \geq \lambda(G,w')$.
    \item If $H$ is an induced subgraph of $G$, then $\lambda(G,w) \geq \lambda(H,w|_H)$.
    \item If $H$ is a wide subgraph of $G$, then $\lambda(H,w) \geq \lambda(G,w)$.
    \item (Convexity) For any $w_1, w_2: V(G) \to \mathbb{R}_+$ and $t \in [0,1]$, we have 
    $$\lambda(G,tw_1+(1-t)w_2) \leq \lambda(G,w_1)^t\lambda(G,w_2)^{1-t}.$$
\end{enumerate}
\end{prop}
\begin{proof} 
(1)-(3) can be found in \cite[P.196 `Basic properties and examples' 1 and 2]{McM15}.

To see (4), note that there is a surjective semi-group homomorphism $A_+(H) \twoheadrightarrow A_+(G)$. The weight function $w:A_+(H) \to \mathbb{R}_+$ factors through this homomorphism, so we have an induced surjection
$$\{g \in A_+(H) \mid w(g) \leq T\} \twoheadrightarrow \{g \in A_+(G) \mid w(g) \leq T\}$$
for each $T$, which implies that $\lambda(H,w) \geq \lambda(G,w)$.

(5) follows from \cite[Theorems 3.3 and 4.2]{McM15}.
\end{proof}

We introduce an operation on weighted graphs that will be handy in our analysis.

\begin{defn} \label{defn:weightedgraphvertexsum}
Let $(G,w)$ be a weighted graph. Let $v_1$ and $v_2$ be two vertices of $G$. We define a weighted graph $(H,z)$ as follows:
\begin{itemize}
    \item The vertices of $H$ are the same as those of $G$ but with $v_1$ and $v_2$ identified into a single vertex $v_3$.
    \item $z(u)=w(u)$ for any vertex $u \neq v_3$ in $H$.
    \item $z(v_3)=w(v_1)+w(v_2)$.
    \item The edges of $H$ between vertices other than $v_3$ are the same as those of $G$.
    \item There is an edge between $v_3$ and another vertex $u$ in $H$ if and only if there are edges between $v_1$ and $u$ and between $v_2$ and $u$ in $G$.
\end{itemize}
We say that $(H,z)$ is obtained from $(G,w)$ by \textbf{summing $v_1$ and $v_2$}.
See \Cref{fig:weightedgraphvertexsum} for an example.
\end{defn}

\begin{figure}
    \centering
    \resizebox{!}{3.5cm}{
    \begin{tikzpicture}
    
    \draw[draw=black, fill=black, thin, solid] (1,0) circle (0.1);
    \node[black, anchor=north] at (1,-0.2) {\large $w(v_1)$};
    \draw[draw=black, fill=black, thin, solid] (3,0) circle (0.1);
    \node[black, anchor=north] at (3,-0.2) {\large $w(v_2)$};
    \draw[draw=black, fill=black, thin, solid] (0,2) circle (0.1);
    \draw[draw=black, fill=black, thin, solid] (2,3) circle (0.1);
    \draw[draw=black, fill=black, thin, solid] (4,2) circle (0.1);
    \draw[draw=black, thin, solid] (1,0) -- (4,2);
    \draw[draw=black, thin, solid] (3,0) -- (2,3);
    \draw[draw=black, thin, solid] (3,0) -- (4,2);
    \node[black, anchor=north] at (2,-1.2) {\Large $(G,w)$};
    
    \draw[draw=black, -Latex, thick, solid] (5,1) -- (9,1);
    
    \draw[draw=black, fill=black, thin, solid] (12,0) circle (0.1);
    \node[black, anchor=north] at (12,-0.2) {\large $z(v_3)=w(v_1)+w(v_2)$};
    \draw[draw=black, fill=black, thin, solid] (10,2) circle (0.1);
    \draw[draw=black, fill=black, thin, solid] (12,3) circle (0.1);
    \draw[draw=black, fill=black, thin, solid] (14,2) circle (0.1);
    \draw[draw=black, thin, solid] (12,0) -- (14,2);
    \node[black, anchor=north] at (12,-1.2) {\Large $(H,z)$};

    \end{tikzpicture}}
    \caption{Summing the vertices $v_1$ and $v_2$ in a weighted graph $(G,w)$ to get a weighted graph $(H,z)$.}
    \label{fig:weightedgraphvertexsum}
\end{figure}

\begin{prop} \label{prop:vertexsumgrowthrate}
Suppose $(H,z)$ is obtained from $(G,w)$ by summing two vertices. Then $\lambda(H,z) \leq \lambda(G,w)$.
\end{prop}
\begin{proof}
We use the tool of minimal words as described in \cite[Section 4]{McM15}, see also \cite{GS00}: Given a total order on $V(G)$, a word in $A_+(G)$ is \textbf{minimal} if it does not contain a subword of the form $c_1...c_k d$ where $d$ commutes with $c_1,...,c_k$ and $d<c_1$. Each element of $A_+(G)$ is represented by a unique minimal word. 

We choose a total order on $V(G)$ with $v_1<v_2$ being the two lowest ordered elements. This induces a total order on $V(H)$ by letting $v_3$ be the lowest ordered element and retaining the ordering on the rest of the elements. Under these orders, given a minimal word in $A_+(H)$, we can replace every maximal string of the form $v_3^k$ by $v_1^kv_2^k$ to obtain a minimal word in $A_+(G)$ with the same weight.

This defines an injection 
$$\{g \in A_+(H) \mid z(g) \leq T\} \hookrightarrow \{g \in A_+(G) \mid w(g) \leq T\}$$
for each $T$, which implies that $\lambda(H,z) \leq \lambda(G,w)$.
\end{proof}

More generally, we can sum together a collection of vertices. The growth rate is non-increasing under this operation.

\subsection{Curve complexes} \label{subsec:curvecomplex}

Let $A \in M_{n \times n}(\mathbb{Z}_{\geq 0})$ be a square matrix with non-negative integer entries. The \textbf{directed graph associated to $A$}, which we denote by $\Gamma_A$, is the graph whose vertex set is $\{1,...,n\}$ and with $A_{ji}$ edges from $i$ to $j$. Equivalently, $\Gamma_A$ is the directed graph whose adjacency matrix is $A$.
More generally, for $k \geq 1$, the $(j,i)$-entry of $A^k$ is the number of directed edge paths from $i$ to $j$.

Suppose that $A$ is Perron-Frobenius.
This implies that $\Gamma_A$ is \textbf{strongly connected}, i.e. for any pair of vertices $(v_1,v_2)$ there is a directed edge path from $v_1$ to $v_2$.

We use the following terminology for directed graphs:
A \textbf{cycle} of $\Gamma_A$ is a cyclic sequence of directed edges $(e_i)_{i \in \mathbb{Z}/p}$ where the terminal point of $e_i$ equals the initial point of $e_{i+1}$ for every $i$. In particular, a cycle is allowed to have repeated vertices. The \textbf{length} of a cycle is the number of edges that it consists of. 
A cycle is \textbf{primitive} if it is not a multiple of another cycle.
A \textbf{curve} of $\Gamma_A$ is an embedded cycle, i.e. a cycle that does not have repeated vertices.
Every cycle $c$ is the concatenation of a (possibly non-unique) collection of curves $c_1,...,c_n$. In this case, we abuse terminology slightly and say that each $c_i$ is \textbf{contained} in $c$.

The \textbf{curve complex} of $\Gamma_A$, which we denote by $(G_A,w_A)$, is the weighted graph whose vertices are the curves of $\Gamma_A$, the weight of a vertex being the length of the curve, and where two vertices are connected by an edge if and only if the corresponding curves of $\Gamma_A$ are disjoint, i.e. do not share any vertex.

The following theorem provides the connection between Perron-Frobenius matrices and the clique polynomial.

\begin{thm}[{\cite[Theorem 1.4]{McM15}}] \label{thm:cliquepoly=charpoly}
The characteristic polynomial of $A$ equals the reciprocal of the clique polynomial of $(G_A,w_A)$. In particular the spectral radius of $A$ equals the growth rate of $(G_A,w_A)$.
\end{thm}

\section{Standardly embedded train tracks} \label{sec:sett}

In this section, we set up the machinery of standardly embedded train tracks. 
The main new result is \Cref{thm:basedsettexists}.
Most ideas in this section can be found in \cite{HT22}, and we refer to that paper for some of the proofs.

\subsection{Train track terminology} \label{subsec:ttdefn}

A \textbf{train track} $\tau$ on a finite-type surface $S$ is an embedded graph with a choice of a tangent line at each vertex, such that the half-edges incident to each vertex are tangent to the tangent line and there is at least one half-edge tangent to each side of the tangent line.
See \Cref{fig:tienbd} top for a local picture of a train track.
We refer to the vertices of a train track as its \textbf{switches}.

A \textbf{tie neighborhood} of $\tau$ is a closed subset $N$ of $S$ that can be written as a union of compact intervals, called the \textbf{ties}, such that the map $S \to S$ defined by quotienting each tie sends $N$ to $\tau$. 
The ties that map to the switches of $\tau$ are the \textbf{switch ties}. See \Cref{fig:tienbd} bottom.

\begin{figure}
    \centering
    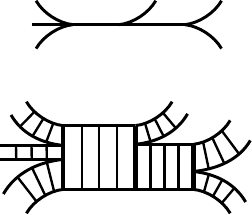
    \caption{Top: A train track. Bottom: The tie neighborhood of a train track. The switch ties are bolded.}
    \label{fig:tienbd}
\end{figure}

A train track $\tau$ is said to \textbf{carry} a lamination $\ell$ on $S$ if there is a tie neighborhood $N$ of $\tau$ such that the leaves of $l$ are contained in $N$ and are transverse to the ties.
More generally, $\tau$ is said to \textbf{carry} a singular foliation $\ell$ if $\tau$ carries the lamination obtained by blowing air into the singular leaves of $\ell$.

In this paper, all train tracks will carry the unstable foliation of some pseudo-Anosov map. The laminations obtained by blowing air into the singular leaves of such a foliation have the property that every complementary region is a disc $D^2$ or a half-open annulus $S^1 \times [0,1)$ with $n \geq 1$ points on its boundary removed, where $n$ is the number of prongs of the singular point that the leaf contained. Correspondingly, the complementary regions of our train tracks will always be discs or half-open annuli with $n \geq 1$ cusps on its boundary. More precisely, by a \textbf{cusp}, we mean a non-smooth point of the boundary curve of the complementary region.

We define a \textbf{boundary component} of a train track $\tau$ to be a boundary component of a complementary region. 
We say that a boundary component is \textbf{$n$-pronged} if it contains $n$ cusps.
We denote the (abstract) union of boundary components of $\tau$ by $\partial \tau$.

Let $\tau$ and $\tau'$ be train tracks on a surface $S$. Let $N$ and $N'$ be tie neighborhoods of $\tau$ and $\tau'$ respectively. A \textbf{train track map} from $\tau$ to $\tau'$ is a homeomorphism $g:S \to S$ that sends $N$ into $N'$, mapping the ties of $N$ into ties of $N'$ and mapping the switch ties of $N$ into the switch ties of $N'$.
See \Cref{fig:ttmap} top.

\begin{figure}
    \centering
    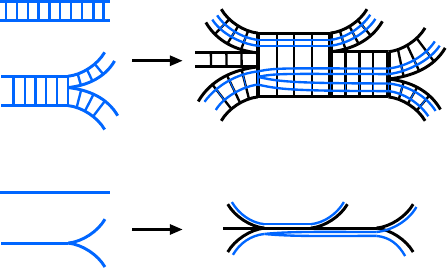
    \caption{Top: A homeomorphism that sends the tie neighborhood $N$ into the tie neighborhood $N'$. Bottom: The induced map on the train tracks $\tau \to \tau'$. For clarity, we have drawn the image of $\tau$ lying slightly off $\tau'$.}
    \label{fig:ttmap}
\end{figure}

By collapsing along the ties, we get an induced map $\tau \to \tau'$ that sends the switches of $\tau$ to switches of $\tau'$ and smooth edge paths of $\tau$ to smooth edge paths of $\tau'$. 
See \Cref{fig:ttmap} bottom.
In this paper, we take the customary approach of referring to this map as the train track map instead, while implicitly remembering the data of the tie neighborhoods.

The \textbf{transition matrix} of a train track map $g:\tau \to \tau'$ is the matrix 
$$g_* \in \mathrm{Hom}(\mathbb{R}^{E(\tau)},\mathbb{R}^{E(\tau')})$$
whose $(e',e)$-entry is the number of times $g(e)$ passes through $e'$.

A \textbf{train track isomorphism} is a train track map that is a graph isomorphism. The transition matrix of a train track isomorphism is always a permutation matrix.

\subsection{(Based) standardly embedded train tracks}

\begin{defn}[Standardly embedded train tracks] \label{defn:sett}
Let $\partial \tau = \partial_I \tau \sqcup \partial_O \tau$ be a partition of the boundary components of $\tau$ into a nonempty set of \textbf{inner boundary components} and a nonempty set of \textbf{outer boundary components} respectively.

A train track $\tau$ is said to be \textbf{standardly embedded} with respect to $(\partial_I \tau, \partial_O \tau)$ if its set of edges $E(\tau)$ can be partitioned into a set of \textbf{infinitesimal edges} $E_\infs(\tau)$ and a set of \textbf{real edges} $E_\real(\tau)$, such that:
\begin{itemize}
    \item At each switch, the infinitesimal edges are tangent to one side of the tangent line while the real edges are tangent to the other side. 
    \item The union of infinitesimal edges is a disjoint union of cycles, which we call the \textbf{infinitesimal polygons}.
    \item The infinitesimal polygons are exactly the inner boundary components of $\tau$.
\end{itemize}
\end{defn}

We show one example of a standardly embedded train track in \Cref{fig:standardlyemb}.

\begin{figure}
    \centering
    \resizebox{!}{4cm}{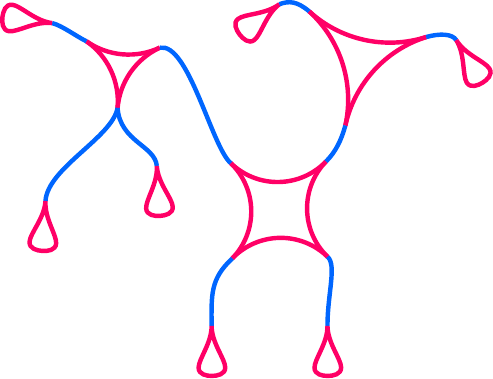}
    \caption{An example of a standardly embedded train track. The infinitesimal edges are in red and the real edges are in blue.}
    \label{fig:standardlyemb}
\end{figure}

A \textbf{standardly embedded train track map} is a train track map $g:\tau \to \tau'$ between standardly embedded train tracks that send infinitesimal edges to infinitesimal edges.
By listing the infinitesimal edges in front of the real edges, the transition matrix of a standardly embedded train track map admits a block decomposition 
$$g_* = \begin{bmatrix}
P & * \\
0 & g_*^\real
\end{bmatrix}$$
where $P$ is a permutation matrix. We call $g_*^\real$ the \textbf{real transition matrix} of $g$.

The dimension of the real transition matrix is well-controlled:

\begin{prop}[{\cite[Proposition 3.2]{HT22}}] \label{prop:realedgenumber}
Let $\tau$ be a standardly embedded train track. The number of real edges of $\tau$ is equal to $-\chi(\tau)$.
\end{prop}

The standardly embedded train track maps that we will deal with in this paper can be decomposed into elementary moves, which we now describe.

\begin{constr} \label{constr:elementaryfold}
Let $\tau$ be a standardly embedded train track. Let $c$ be an outer boundary component, and let $s$ be a cusp of $c$. Let $e_0$ and $e_1$ be the two edges of $\tau$ adjacent to $s$ along $c$. These must be real edges. Suppose $e_0$ is not a complete side of $c$, i.e. suppose there is an infinitesimal edge $e'_0$ adjacent to $e_0$ along $c$ such that the concatenation $e_0 * e'_0$ is a smooth edge path lying along $c$.

We say that the map $g$ that folds a proper subpath of $e_1$ onto $e_0 * e'_0$ is an \textbf{elementary fold} of $e_1$ onto $e_0$ at $s$. See \Cref{fig:elementaryfold}.

\begin{figure}
    \centering
    \selectfont \fontsize{8pt}{8pt}
    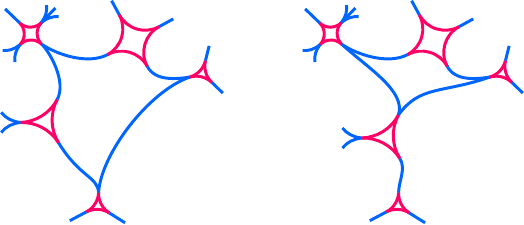
    \caption{An elementary fold of $e_1$ onto $e_0$ at $s$.}
    \label{fig:elementaryfold}
\end{figure}

Note that $g$ is a standardly embedded train track map from $\tau$ to another standardly embedded train track $\tau'$ that has the same infinitesimal polygons and the same number of real edges as $\tau$. 
In fact, there is a canonical identification between the edges of $\tau$ and that of $\tau'$: We denote the image of $e_0$ as $e_0$, the image of $e'_0$ as $e'_0$, and the subpath of $e_1$ that was not folded as $e_1$, and retain the label on the remaining edges.
Under this labelling, the real transition matrix $g^\real_*$ is the matrix $I+N$ where $I$ is the identity matrix and $N$ has a $1$ in the $(e_0,e_1)$-entry and $0$ elsewhere.
\end{constr}

In \cite{HT22}, Hironaka and the first author showed that the dynamics of any pseudo-Anosov map with at least two singularity orbits can be encoded by a standardly embedded train track map, in the following sense.

\begin{thm}[{\cite[Theorem 1.15]{HT22}}] \label{thm:settexists}
Let $f$ be a pseudo-Anosov map. Given a partition of the singular points of $S$ into two nonempty $f$-invariant subsets $\mathcal{X} = \mathcal{X}_I \sqcup \mathcal{X}_O$, there exists a standardly embedded train track $\tau$ and a standardly embedded train track map $g:\tau \to \tau$ such that:
\begin{enumerate}
    \item $g$ is induced by a map that sends a tie neighborhood of $\tau$ into itself and that is homotopic to $f$.
    \item $\tau$ carries the unstable foliation of $f$. Each boundary component in $\partial_{I/O} \tau$ corresponds to one element of $\mathcal{X}_{I/O}$ respectively.
    \item $g$ can be factored into elementary folds followed by a train track isomorphism.
    \item The real transition matrix $g^\real_*$ is Perron-Frobenius and reciprocal.
    \item The spectral radius of $g^\real_*$ is the dilatation of $f$.
\end{enumerate}
\end{thm}

Technically, (3) is not recorded in \cite[Theorem 1.15]{HT22}, but it follows from the proof of \cite[Lemma 3.20]{HT22}. In any case, \Cref{thm:settexists} will be subsumed by \Cref{thm:basedsettexists}, which is a mild upgrade saying that the cusps of the outer boundary components can be chosen to lie on prescribed inner boundary components.

To work towards \Cref{thm:basedsettexists}, we introduce some specializations of our definitions so far:

\begin{defn}[Based standardly embedded train tracks] \label{defn:basedsett}
Let $\partial \tau = \partial_P \tau \sqcup \partial_A \tau \sqcup \partial_O \tau$ be a partition of the boundary components of $\tau$ into a nonempty set of \textbf{pistil boundary components}, a (possibly empty) set of \textbf{anther boundary components}, and a nonempty set of \textbf{outer boundary components} respectively.

A train track that is standardly embedded with respect to $(\partial_P \tau \sqcup \partial_A \tau, \partial_O \tau)$ is said to be \textbf{based at $\partial_P \tau$} if each cusp of an outer boundary component lies on a pistil boundary component. When $\partial_P \tau$ is understood, we simply say that $\tau$ is a \textbf{based standardly embedded train track}.
\end{defn}

Our terminology is motivated by the special case of floral train tracks, which we mentioned in \Cref{subsec:introfloraltt} (see \Cref{fig:floraltteg}) and will define precisely in \Cref{subsec:floraltt}.

Performing an elementary folding move on a based standardly embedded train track $\tau$ returns a standardly embedded train track $\tau'$, but $\tau'$ may no longer be based at $\partial_P \tau$. Indeed, in the notation of \Cref{constr:elementaryfold}, the cusp $s$ now lies on the inner boundary component that contains $e'_0$, which may or may not have been an element of $\partial_P \tau$.

To make sure we stay within the category of based standardly embedded train tracks, we define a \textbf{based fold} to be a composition of elementary folds at the same cusp $s$, $\tau= \tau_0 \overset{g_1}{\to} ...\overset{g_n}{\to} \tau_n$ where the intermediate train tracks $\tau_1,...,\tau_{n-1}$ are not based at $\partial_P \tau$, but the final train track $\tau_n$ is based at $\partial_P \tau$.
We compose the edge identifications before and after each elementary fold to identify the edges before and after a based fold.

\subsection{Train track partitions}

The goal of this subsection is to retrace the constructions in \cite[Section 3]{HT22} in order to state and prove \Cref{thm:basedsettexists}.

Let $\overline{f}:\overline{S} \to \overline{S}$ be the completion of a pseudo-Anosov map $f:S \to S$.

A \textbf{rectangle} is the image of a map $R:[0,1] \times [0,1] \to \overline{S}$ where
\begin{itemize}
    \item $R$ is an embedding in the interior $(0,1) \times (0,1)$,
    \item the image of $\{u_0\} \times [0,1]$ lies within a stable leaf for every $u_0$, and
    \item the image of $[0,1] \times \{s_0\}$ lies within an unstable leaf for every $s_0$.
\end{itemize}
We refer to $R(\{0\} \times [0,1])$ and $R(\{1\} \times [0,1])$ as the \textbf{stable sides} and $R([0,1] \times \{0\})$ and $R([0,1] \times \{1\})$ as the \textbf{unstable sides} of the rectangle.

Let $x$ be a singular point. Suppose $x$ is $n$-pronged. Then the stable/unstable leaf containing $x$ is the union of $n$ half-leaves $\overline{\ell}^{s/u}_1,...,\overline{\ell}^{s/u}_n$ meeting at $x$. A \textbf{stable/unstable prong} at $x$ is a compact segment $P_i$ of some $\overline{\ell}^{s/u}_i$ that contains $x$. A \textbf{stable/unstable star} at $x$ is a union of prongs $\bigcup_{i=1}^n P_i$, where each $P_i$ is contained in $\overline{\ell}^{s/u}_i$. A \textbf{side} of a stable star is the union of two adjacent prongs.

\begin{defn}[(Based) train track partitions] \label{defn:ttpartition}
Let $\mathcal{X}=\mathcal{X}_P \sqcup \mathcal{X}_A \sqcup \mathcal{X}_O$ be a partition of the set of singular points $\mathcal{X}$. We write $\mathcal{X}_I = \mathcal{X}_P \sqcup \mathcal{X}_A$.

A \textbf{train track partition} with respect to $(\mathcal{X}_I,\mathcal{X}_O)$ consists of
\begin{itemize}
    \item a stable star $\sigma^s_x$ at each $x \in \mathcal{X}_I$, and
    \item an unstable star $\sigma^u_x$ at each $x \in \mathcal{X}_O$
\end{itemize}
such that
\begin{itemize}
    \item the complementary regions of $\bigcup_{x \in \mathcal{X}_I} \sigma^s_x \cup \bigcup_{x \in \mathcal{X}_O} \sigma^u_x$ are rectangles, and
    \item the interiors of the stable and unstable stars are disjoint from each other.
\end{itemize}

A train track partition $((\sigma^s_x)_{x \in \mathcal{X}_I}, (\sigma^u_x)_{x \in \mathcal{X}_O})$ is \textbf{Markov} if 
\begin{itemize}
    \item $\overline{f}(\bigcup_{x \in \mathcal{X}_I} \sigma^s_x) \subset \bigcup_{x \in \mathcal{X}_I} \sigma^s_x$, and
    \item $\overline{f}(\bigcup_{x \in \mathcal{X}_O} \sigma^u_x) \supset \bigcup_{x \in \mathcal{X}_O} \sigma^u_x$.
\end{itemize}

A train track partition $((\sigma^s_x)_{x \in \mathcal{X}_I}, (\sigma^u_x)_{x \in \mathcal{X}_O})$ is \textbf{based at $\mathcal{X}_P$} if each endpoint of an unstable star $\sigma^u_x$ lies on a stable star $\sigma^s_x$ where $x \in \mathcal{X}_P$.
\end{defn}

Given a $f$-invariant partition of $\mathcal{X}$, we demonstrate a way of constructing based Markov train track partitions:

\begin{constr} \label{constr:buildttpartition}
Let $\mathcal{X}=\mathcal{X}_P \sqcup \mathcal{X}_A \sqcup \mathcal{X}_O$ be a partition of $\mathcal{X}$ into $f$-invariant subsets, where $\mathcal{X}_P$ and $\mathcal{X}_O$ are nonempty. We write $\mathcal{X}_I = \mathcal{X}_P \sqcup \mathcal{X}_A$.

We construct a stable star $\sigma^s_x$ for every $x \in \mathcal{X}_I$ and an unstable star $\sigma^u_x$ for every $x \in \mathcal{X}_O$ as follows:
\begin{enumerate}
    \item For each $x \in \mathcal{X}_P$, let $\sigma^s_x$ be the stable star at $x$ for which each of its prongs has $\mu^u$-length 1.
    \item For each $x \in \mathcal{X}_O$, let $\sigma^u_x$ be the unstable star at $x$ that is maximal with respect to the property that its interior is disjoint from $\bigcup_{y \in \mathcal{X}_P} \sigma^s_y$. That is, each $\sigma^u_x$ is defined by extending the unstable prongs at $x$ until it bumps into some $\sigma^s_y$.
    \item For each $x \in \mathcal{X}_I$, let $\sigma^s_x$ be the stable star at $x$ that is maximal with respect to the property that its interior is disjoint from $\bigcup_{y \in \mathcal{X}_O} \sigma^u_y$. That is, each $\sigma^s_x$ is defined by extending the stable prongs at $x$ until it bumps into some $\sigma^u_y$.
\end{enumerate}

Here we have implicitly applied \Cref{prop:halfleafdense} in steps (2) and (3) to make sure that we can extend $\sigma^u_x$ and $\sigma^s_x$ until they bump into the required stars.

The same argument as in \cite[Construction 3.6]{HT22} shows that $((\sigma^s_x)_{x \in \mathcal{X}_I}, (\sigma^u_x)_{x \in \mathcal{X}_O})$ is a train track partition with respect to $(\mathcal{X}_I,\mathcal{X}_O)$:
By construction the interior of the stable and unstable stars is disjoint from each other.
The complementary regions are foliated by the restriction of the stable leaves, and their corners are all convex, so they must be rectangles.

The same argument as in \cite[Proposition 3.7]{HT22} shows that $((\sigma^s_x)_{x \in \mathcal{X}_I}, (\sigma^u_x)_{x \in \mathcal{X}_O})$ is Markov: 
One checks that $\overline{f}(\bigcup_{x \in \mathcal{X}_I} \sigma^s_x) \subset \bigcup_{x \in \mathcal{X}_I} \sigma^s_x$ and $\overline{f}(\bigcup_{x \in \mathcal{X}_O} \sigma^u_x) \supset \bigcup_{x \in \mathcal{X}_O} \sigma^u_x$ are true at each stage of the construction (taking the yet-to-be-defined stars to be empty).
At stage (1), we have $\overline{f}(\bigcup_{x \in \mathcal{X}_P} \sigma^s_x) \subset \bigcup_{x \in \mathcal{X}_P} \sigma^s_x$ since $\overline{f}$ contracts the stable stars by a factor of $\lambda(f)$. This implies that $\overline{f}(\bigcup_{x \in \mathcal{X}_O} \sigma^u_x) \supset \bigcup_{x \in \mathcal{X}_O} \sigma^u_x$ in stage (2). This in turn implies that $\overline{f}(\bigcup_{x \in \mathcal{X}_I} \sigma^s_x) \subset \bigcup_{x \in \mathcal{X}_I} \sigma^s_x$ in stage (3).

Finally, $((\sigma^s_x)_{x \in \mathcal{X}_I}, (\sigma^u_x)_{x \in \mathcal{X}_O})$ is based at $\mathcal{X}_P$ by construction.
\end{constr}

\Cref{constr:buildttpartition} is in a sense the most natural way to build a based train track partition, given the argument in the following proposition.

\begin{prop} \label{prop:pistilstardeterminettpartition}
A based train track partition $((\sigma^s_x)_{x \in \mathcal{X}_I}, (\sigma^u_x)_{x \in \mathcal{X}_O})$ is uniquely determined by its stable stars $\sigma^s_x$ for $x \in \mathcal{X}_P$.
\end{prop}
\begin{proof}
If two based train track partitions have the same collection of stable stars $\sigma^s_x$ for $x \in \mathcal{X}_P$, then their unstable stars $\sigma^u_x$ for $x \in \mathcal{X}_O$ must coincide since these must be the maximal unstable stars whose interiors are disjoint from $\bigcup_{x \in \mathcal{X}_P} \sigma^s_x$. In turn, their stable stars $\sigma^s_x$ for $x \in \mathcal{X}_A$ must coincide since these must be the maximal stable stars whose interiors are disjoint from $\bigcup_{x \in \mathcal{X}_O} \sigma^u_x$.
\end{proof}

We also record the following construction, which recovers the basedness of a Markov train track partition after moving some elements from $\mathcal{X}_A$ to $\mathcal{X}_O$.

\begin{constr} \label{constr:convertaottpartition}
Let $\mathcal{X}=\mathcal{X}_P \sqcup \widehat{\mathcal{X}}_A \sqcup \widehat{\mathcal{X}}_O$ be a partition of $\mathcal{X}$ into $f$-invariant subsets, where $\mathcal{X}_P$ and $\widehat{\mathcal{X}}_O$ are nonempty. We write $\widehat{\mathcal{X}}_I = \mathcal{X}_P \sqcup \widehat{\mathcal{X}}_A$. Suppose $((\widehat{\sigma}^s_x)_{x \in \widehat{\mathcal{X}}_I}, (\widehat{\sigma}^u_x)_{x \in \widehat{\mathcal{X}}_O})$ is a based Markov train track partition. 

Let $\mathcal{X}=\mathcal{X}_P \sqcup \mathcal{X}_A \sqcup \mathcal{X}_O$ be another partition of $\mathcal{X}$ into $f$-invariant subsets, where $\mathcal{X}_A \subset \widehat{\mathcal{X}}_A$ and $\mathcal{X}_O \supset \widehat{\mathcal{X}}_O$. We write $\mathcal{X}_I = \mathcal{X}_P \sqcup \mathcal{X}_A$. Note that the set of pistil singular points has not been changed.

For each $x \in \widehat{\mathcal{X}}_A \backslash \mathcal{X}_A = \mathcal{X}_O \backslash \widehat{\mathcal{X}}_O$, we define $\sigma^u_x$ to be the unstable star at $x$ that is maximal with respect to the property that its interior is disjoint from $\bigcup_{y \in \mathcal{X}_P} \widehat{\sigma}^s_y$. That is, each $\sigma^u_x$ is defined by extending the unstable prongs at $x$ until it bumps into $\widehat{\sigma}^s_y$ for some $y \in \mathcal{X}_P$.
Here again we have implicitly applied \Cref{prop:halfleafdense}.
Note that the interior of $\sigma^u_x$ can intersect the interior of some $\widehat{\sigma}^s_y$, for $y \in \mathcal{X}_A$.

Next, for each $x \in \mathcal{X}_A$, we define $\sigma^s_x$ to be the stable star at $x$ that is maximal with respect to the property that its interior is disjoint from $\bigcup_{y \in \mathcal{X}_O} \sigma^u_y$. That is, we retract the prongs of $\widehat{\sigma}^s_x$ until its interior no longer intersects $\bigcup_{y \in \widehat{\mathcal{X}}_A \backslash \mathcal{X}_A} \sigma^u_y$.

Finally, for each $x \in \mathcal{X}_P$, we define $\sigma^s_x = \widehat{\sigma}^s_x$, and for each $x \in \widehat{\mathcal{X}}_O$, we define $\sigma^u_x = \widehat{\sigma}^u_x$.

We claim that $((\sigma^s_x)_{x \in \mathcal{X}_I}, (\sigma^u_x)_{x \in \mathcal{X}_O})$ is a train track partition with respect to $(\mathcal{X}_I,\mathcal{X}_O)$:
By construction the interior of the stable and unstable stars are disjoint from each other.
The complementary regions of $\bigcup_{x \in \widehat{\mathcal{X}}_I} \widehat{\sigma}^s_x \cup \bigcup_{x \in \mathcal{X}_O} \sigma^u_x$ are those of $\bigcup_{x \in \widehat{\mathcal{X}}_I} \widehat{\sigma}^s_x \cup \bigcup_{x \in \widehat{\mathcal{X}}_O} \widehat{\sigma}^u_x$ divided along unstable leaves, hence are rectangles. 
Since no endpoint of an unstable star $\sigma^u_x$ lay on a stable star $\widehat{\sigma}^s_y$ for $y \in \widehat{\mathcal{X}}_A$, the complementary regions of $\bigcup_{x \in \mathcal{X}_I} \sigma^s_x \cup \bigcup_{x \in \mathcal{X}_O} \sigma^u_x$ are in turn those of $\bigcup_{x \in \widehat{\mathcal{X}}_I} \widehat{\sigma}^s_x \cup \bigcup_{x \in \mathcal{X}_O} \sigma^u_x$ with certain stable sides identified, hence they are rectangles as well.

Since $\overline{f}(\bigcup_{x \in \mathcal{X}_P} \sigma^s_x) \subset \bigcup_{x \in \mathcal{X}_P} \sigma^s_x$, we have $\overline{f}(\bigcup_{x \in \mathcal{X}_O \backslash \widehat{\mathcal{X}}_O} \sigma^u_x) \supset \bigcup_{x \in \mathcal{X}_O \backslash \widehat{\mathcal{X}}_O} \sigma^u_x$. This implies that $((\sigma^s_x)_{x \in \mathcal{X}_I}, (\sigma^u_x)_{x \in \mathcal{X}_O})$ is Markov.

Finally, $((\sigma^s_x)_{x \in \mathcal{X}_I}, (\sigma^u_x)_{x \in \mathcal{X}_O})$ is based at $\mathcal{X}_P$ by construction.
\end{constr}

The crucial property about (based) train track partitions is that they give rise to (based) standardly embedded train tracks.

\begin{constr}[{\cite[Construction 3.8]{HT22}}] \label{constr:ttpartitiontott}
Let $\mathcal{M}=((\sigma^s_x)_{x \in \mathcal{X}_I}, (\sigma^u_x)_{x \in \mathcal{X}_O})$ be a train track partition with respect to $(\mathcal{X}_I,\mathcal{X}_O)$.

Define a graph $\tau_\mathcal{M}$ by taking a set of vertices in one-to-one correspondence with the sides of the stable stars $\sigma^s_x$. The edges of $\tau_\mathcal{M}$ will come in two types: infinitesimal and real. 
The infinitesimal edges correspond to the prongs of $\sigma^s_x$. The endpoints of the edge corresponding to a prong $P$ lies at the two vertices corresponding to the two sides that $P$ belongs to.
The real edges are in one-to-one correspondence with the rectangles $R$ in the complement of $\bigcup_{x \in \mathcal{X}_I} \sigma^s_x \cup \bigcup_{x \in \mathcal{X}_O} \sigma^u_x$. The endpoints of the edge corresponding to a rectangle $R$ lies at the two vertices corresponding to the two sides the stable sides of $R$ lie along.

We embed $\tau_\mathcal{M}$ into the surface $S$ by placing the infinitesimal edges corresponding to the prongs of a single star $\sigma^s_x$ around the singular point $x$ in the cyclic order of the prongs, and placing the real edge along their corresponding rectangle.

See each of the two columns of \Cref{fig:partitiontott} for local examples of the construction.

Observe that the cusps of the outer boundary components of $\tau_\mathcal{M}$ are in one-to-one correspondence with the endpoints of the unstable stars of $\mathcal{M}$. A cusp lies on the inner boundary component of $\tau_\mathcal{M}$ that corresponds to the stable star on which the corresponding endpoint lies.
In particular, if $\mathcal{M}$ is based at $\mathcal{X}_P$, then $\tau_\mathcal{M}$ is based at $\partial_P \tau$.
\end{constr}

\begin{figure}
    \centering
    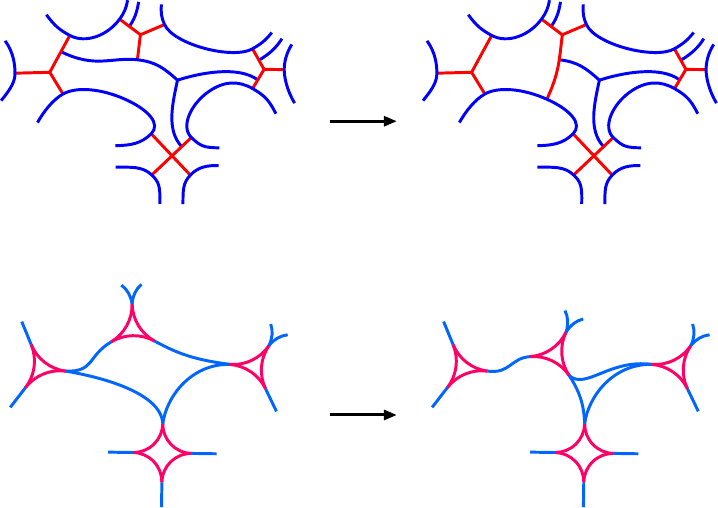
    \caption{Each of the two columns: Local examples of \Cref{constr:ttpartitiontott}. From the left column to the right column: Retracting an unstable prong and extending a stable prong translates to an elementary fold.}
    \label{fig:partitiontott}
\end{figure}

\begin{constr}[{\cite[Construction 3.9]{HT22}}] \label{constr:markovttpartitiontottmap}
Continuing from \Cref{constr:ttpartitiontott}, suppose in addition that $\mathcal{M}$ is Markov. We define a standardly embedded train track map $g_\mathcal{M}:\tau_\mathcal{M} \to \tau_\mathcal{M}$ by mapping the vertex corresponding to a side $s$ of a stable star to the vertex corresponding to the side containing $\overline{f}(s)$, mapping the infinitesimal edge corresponding to a stable prong $P$ to the infinitesimal edge corresponding to the prong containing $\overline{f}(P)$, and mapping the real edge corresponding to a rectangle $R$ to the edge path corresponding to the sequence of rectangles and prongs passed through by $\overline{f}(R)$.

Observe that the $(e',e)$-entry of the real transition matrix $(g_\mathcal{M})^\real_*$ is the number of appearances of $e'$ as an edge in the edge path $g(e)$. In other words, the edges from $e$ to $e'$ in the directed graph $\Gamma$ associated to $(g_\mathcal{M})^\real_*$ are in one-to-one correspondence with these appearances, which in turn correspond to the sub-rectangles in $\overline{f}(R) \cap R'$, where $R$ and $R'$ are the rectangles corresponding to $e$ and $e'$ respectively.
\end{constr}

\begin{thm} \label{thm:basedsettexists}
Let $f$ be a pseudo-Anosov map.
Let $\mathcal{X}=\mathcal{X}_P \sqcup \mathcal{X}_A \sqcup \mathcal{X}_O$ be a partition of $\mathcal{X}$ into $f$-invariant subsets, where $\mathcal{X}_P$ and $\mathcal{X}_O$ are nonempty. 
Suppose $\mathcal{M}$ is a Markov train track partition with respect to $(\mathcal{X}_P \sqcup \mathcal{X}_A, \mathcal{X}_O)$ and based at $\mathcal{X}_P$.
Then the based standardly embedded train track $\tau = \tau_{\mathcal{M}}$ and the standardly embedded train track map $g = g_{\mathcal{M}}$ satisfy:
\begin{enumerate}
    \item $g$ is induced by a map that sends a tie neighborhood of $\tau$ into itself and that is homotopic to $f$.
    \item $\tau$ carries the unstable foliation of $f$. Each boundary component in $\partial_{P/A/O} \tau$ corresponds to one element of $\mathcal{X}_{P/A/O}$ respectively.
    \item $g$ can be factored into based folds followed by a train track isomorphism.
    \item The real transition matrix $g^\real_*$ is Perron-Frobenius and reciprocal.
    \item The spectral radius of $g^\real_*$ is the dilatation of $f$.
\end{enumerate}
\end{thm}
\begin{proof}
(1) and (2) are clear from construction.

To show (3), we will define a sequence of based train track partitions 
$$((\sigma^s_x(k))_{x \in \mathcal{X}_I}, (\sigma^u_x(k))_{x \in \mathcal{X}_O})$$
that interpolates between $\mathcal{M}$ and $\overline{f}^{-1}(\mathcal{M})$.

We start by defining $((\sigma^s_x(0))_{x \in \mathcal{X}_I}, (\sigma^u_x(0))_{x \in \mathcal{X}_O}) = \mathcal{M}$. 
Inductively, suppose 
$$((\sigma^s_x(k))_{x \in \mathcal{X}_I}, (\sigma^u_x(k))_{x \in \mathcal{X}_O})$$
has been defined. 
If every stable prong $P^s$ in $\sigma^s_x(k)$ at a point $x \in \mathcal{X}_P$ is a stable prong in $\overline{f}^{-1}(\mathcal{M})$, then by \Cref{prop:pistilstardeterminettpartition}, $\overline{f}^{-1}(\mathcal{M}) = ((\sigma^s_x(k))_{x \in \mathcal{X}_I}, (\sigma^u_x(k))_{x \in \mathcal{X}_O})$ in which case we are done with the interpolation. 

Otherwise, there exists a stable prong $P^s$ in $\sigma^s_x(k)$ at a point $x \in \mathcal{X}_P$ that is not a stable prong in $\overline{f}^{-1}(\mathcal{M})$, i.e. $P^s$ is a strict subset of a stable prong in $\overline{f}^{-1}(\mathcal{M})$. The endpoint of $P^s$ lies on some unstable prong $P^u$. We define $\sigma^{s/u}_x(k+1)$ from $\sigma^{s/u}_x(k)$ by retracting $P^u$ until its endpoint coincides with the endpoint of $P^s$, then extending all stable prongs until their endpoints lie on the interior of some unstable star. Arguing as in \Cref{constr:convertaottpartition}, one can check that $((\sigma^s_x(k+1))_{x \in \mathcal{X}_I}, (\sigma^u_x(k+1))_{x \in \mathcal{X}_O})$ is a based train track partition.

We claim that the train tracks $\tau_{\mathcal{M}(k)}$ and $\tau_{\mathcal{M}(k+1)}$ are related by a sequence of based folds. Indeed, we can further interpolate between $\mathcal{M}(k)$ and $\mathcal{M}(k+1)$ by, in our construction, performing the retraction of $P^u$ step-by-step by pausing whenever the endpoint of $P^u$ coincides with the endpoint of \emph{some} stable prong, and extending that stable prong until its endpoint lies on the interior of some unstable star. The intermediate train tracks are then related by elementary folds at a single cusp that corresponds to the endpoint of $P^u$. See \Cref{fig:partitiontott}, and compare with \cite[Proposition 3.18]{HT22}. 

We also observe that the initial and final train tracks $\tau_{\mathcal{M}(k)}$ and $\tau_{\mathcal{M}(k+1)}$ are based at $\mathcal{X}_P$. Thus by grouping together the elementary folds that go through non-based train tracks, we have our desired sequence of based folds.

Finally, $\overline{f}$ induces a train track isomorphism from $\tau_{\overline{f}^{-1}(\mathcal{M})}$ to $\tau_\mathcal{M}$.

(4) and (5) can be shown as in \cite{HT22}. We sketch an outline of the argument here:

To show that $g^\real_*$ is Perron-Frobenius, one applies \Cref{prop:halfleafdense} to see that for every rectangle $R$, $\overline{f}^N(R)$ passes through every rectangle for large $N$. Correspondingly, for every real edge $e$, $(g^\real_*)^N(e)$ passes through every real edge.

To show that $g^\real_*$ is reciprocal, one defines the weight space $\mathcal{W}(\tau) \subset \mathbb{R}^{E(\tau)}$ (\cite[Definition 4.1]{HT22}). One defines an antisymmetric bilinear form $\omega$ on $\mathcal{W}(\tau)$, known as the \textbf{Thurston symplectic form} (\cite[Definition 4.3]{HT22}), and shows that $g$ preserves $\omega$ (\cite[Proposition 4.6]{HT22}). One then shows that the actions of $g$ on the following spaces are reciprocal:
\begin{itemize}
    \item the quotient $\mathbb{R}^{E(\tau)}/\mathcal{W}(\tau)$ (\cite[Proposition 4.2]{HT22}),
    \item the radical $\rad(\omega)$ (\cite[Propositions 5.4 and 5.5]{HT22}), and
    \item the quotient $\mathcal{W}(\tau)/\rad(\omega)$ (\cite[Proposition 2.13(1)]{HT22}),
\end{itemize}
and deduce that the action of $g$ on $\mathbb{R}^{E(\tau)}$, i.e. the transition matrix $g_*$, is reciprocal (\cite[Proposition 2.13(3)]{HT22}). Since $g$ acts by a permutation on the infinitesimal edges, the real transition matrix $g^\real_*$ is also reciprocal ((\cite[Proposition 2.13(2) and (3)]{HT22})).

To show that the spectral radius of $g^\real_*$ is the dilatation $\lambda(f)$, note that the $\mu^u$-measures of the rectangles provide a positive $\lambda(f)$-eigenvector, so we can conclude by the Perron-Frobenius theorem.
\end{proof}

In the sequel, we will refer to a standardly embedded train track map $g$ as in \Cref{thm:basedsettexists} as a train track map that \textbf{carries} $f$ with respect to the partition $\mathcal{X}=\mathcal{X}_P \sqcup \mathcal{X}_A \sqcup \mathcal{X}_O$. 
In view of item (2) we will often identify a boundary component of $\tau$ with its corresponding singular point.

\subsection{Periodic orbits and cycles} \label{subsec:periodicorbitscorr}

In this subsection, we supplement \Cref{thm:basedsettexists} by showing that if $g$ is a train track map that carries $f$, then there is a correspondence between the periodic orbits of $\overline{f}$ and the cycles in the directed graph $\Gamma$ associated to the real transition matrix of $g$.

Let $\mathbf{a} = \{\overline{f}^k(a) \mid k \in \mathbb{Z}/p\}$ be a periodic orbit of period $p$. First suppose that $\mathbf{a}$ is not a singularity orbit, i.e. $\mathbf{a} \not\subset \mathcal{X}$. Then by \Cref{prop:periodicpointsuniqueleaves}, each $\overline{f}^k(a)$ does not lie on a stable or unstable star, hence must lie in the interior of a unique complementary rectangle $R_k$. 

Let $e_k$ be the real edge corresponding to $R_k$. 
Recall that the edges from $e_k$ to $e_{k+1}$ in $\Gamma$ correspond to the sub-rectangles in $\overline{f}(R_k) \cap R_{k+1}$.
Let $\epsilon_k$ be the edge that corresponds to the sub-rectangle containing $\overline{f}^{k+1}(a)$. 
We refer to the cycle $(\epsilon_k)_{k \in \mathbb{Z}/p}$ in $\Gamma$ as the cycle \textbf{determined by $\mathbf{a}$}.

Now suppose instead that $\mathbf{a} \subset \mathcal{X}_I$, and suppose $\mathbf{a}$ is $n$-pronged. 
Let $P$ be the rotationless period of $\mathbf{a}$. Recall that $P \mid pn$.
The action of $\overline{f}$ on the set of the $pn$ unstable half-leaves at the $p$ points of $\mathbf{a}$ has $\frac{pn}{P}$ orbits, each of size $P$.
We label these orbits as $\{\overline{f}^k(\overline{\ell}^u_i) \mid k \in \mathbb{Z}/P \}$, where $i=1,...,\frac{pn}{P}$.
Since each $\overline{f}^k(a)$ lies at the center of a stable star, the initial portion of each $\overline{f}^k(\overline{\ell}^u_i)$ is contained in a unique rectangle $R_{i,k}$.

Let $e_{i,k}$ be the real edge corresponding to $R_{i,k}$. 
Let $\epsilon_{i,k}$ be the edge from $e_{i,k}$ to $e_{i,k+1}$ that corresponds to the sub-rectangle containing the initial portion of $\overline{f}^{k+1}(\overline{\ell}^u_i)$.
We refer to the cycles $(\epsilon_{1,k})_{k \in \mathbb{Z}/P},...,(\epsilon_{\frac{pn}{P},k})_{k \in \mathbb{Z}/P}$ in $\Gamma$ as the cycles \textbf{determined by $\mathbf{a}$}.

Finally, the case when $\mathbf{a} \subset \mathcal{X}_O$ is similar to the case when $\mathbf{a} \subset \mathcal{X}_I$, but with the unstable half-leaves replaced by stable half-leaves:
The action of $\overline{f}$ on the set of the $pn$ stable half-leaves at the $p$ points of $\mathbf{a}$ has $\frac{pn}{P}$ orbits $\{\overline{f}^k(\overline{\ell}^s_i) \mid k \in \mathbb{Z}/P \}$, where $i=1,...,\frac{pn}{P}$. The initial portion of each $\overline{f}^k(\overline{\ell}^s_i)$ is contained in a unique rectangle $R_{i,k}$.

Let $e_{i,k}$ be the real edge corresponding to $R_{i,k}$, and let $\epsilon_{i,k}$ be the edge from $e_{i,k}$ to $e_{i,k+1}$ that corresponds to the sub-rectangle containing the initial portion of $\overline{f}^{k+1}(\overline{\ell}^s_i)$.
We refer to the cycles $(\epsilon_{1,k})_{k \in \mathbb{Z}/P},...,(\epsilon_{\frac{pn}{P},k})_{k \in \mathbb{Z}/P}$ in $\Gamma$ as the cycles \textbf{determined by $\mathbf{a}$}.

\begin{prop} \label{prop:periodicorbittocycle}
Let $g$ be a train track map that carries $f$.
Then:
\begin{enumerate}
    \item Every cycle determined by a periodic orbit is primitive.
    \item The cycles determined by distinct periodic orbits are distinct.
    \item The cycles determined by a singularity orbit are distinct.
    \item Every primitive cycle is determined by some periodic orbit.
\end{enumerate}
\end{prop}
\begin{proof}
Suppose $(\epsilon_k:e_k \to e_{k+1})_{k \in \mathbb{Z}/P}$ is a cycle determined by a periodic orbit $\mathbf{a} = \{f^k(a)\}$. Let $R_k$ be the rectangle corresponding to $e_k$.
Up to relabeling, we can assume that $f^k(a) \in R_k$ for all $k$.

For the duration of this proof, we take the intersection $\overline{f}^{-1}(R_{k+1}) \cap R_k$ to mean only the sub-rectangle corresponding to the given edge $\epsilon_k$.
More generally, we interpret $\overline{f}^{-k}(R_{j+k}) \cap R_j$ to mean only the sub-rectangle corresponding to the sub-path $\epsilon_j \cdots \epsilon_{j+k}$.

Consider the intersection $\overline{R}_j = \bigcap_{k=-\infty}^\infty \overline{f}^{-k}(R_{j+k})$. By definition, we have $f^j(a) \in \overline{R}_j$. 
The $\overline{\mu}^s$-measure of $\overline{f}^{-kP}(R_{j+kP}) = \overline{f}^{-kP}(R_j)$ decreases to $0$ as $k \to \infty$, while the $\overline{\mu}^u$-measure of $\overline{f}^{-kP}(R_{j+kP}) = \overline{f}^{-kP}(R_j)$ decreases to $0$ as $k \to -\infty$, so $\overline{R}_j$ must only consist of the point $f^j(a)$.

If $\mathbf{a}$ is nonsingular, we can conclude that the edge cycle $(\epsilon_k)$ must be primitive, since otherwise $f^j(a)=a$ for $j$ equal to the length of the primitive cycle of which $(\epsilon_k)$ is a multiple. 

We can show primitivity in the case when $\mathbf{a}$ is singular by a similar argument:
Suppose $(\epsilon_k)_{k \in \mathbb{Z}/P}$ is determined by an orbit of an unstable half-leaf $\{\overline{f}^k(\overline{\ell}^u)\}$. 
Up to relabeling, we can assume that the initial portion of $\overline{f}^k(\overline{\ell}^u)$ lies in $R_k$ for all $k$.

Consider the intersection $\overline{R}_j = \bigcap_{k=-\infty}^0 \overline{f}^{-k}(R_{j+k})$.
By definition, the initial portion of $\overline{f}^j(\overline{\ell}^u)$ lies in $\overline{R}_j$, and since the $\overline{\mu}^u$-measure of $\overline{f}^{-kP}(R_{j+kP})$ decreases to $0$ as $k \to -\infty$, $\overline{R}_j$ must be equal to this initial portion.

These arguments also show that we can recover the periodic orbit $\mathbf{a}$ and, if $\mathbf{a}$ is singular, the orbit of the stable/unstable half-leaf, from the edge cycle $(\epsilon_k)$. This implies (2) and (3).

Finally, to show (4), suppose we are given a cycle $(\epsilon_k:e_k \to e_{k+1})_{k \in \mathbb{Z}/P}$. Let $R_k$ be the rectangle corresponding to $e_k$. Then the intersection $\bigcap_{k=-\infty}^\infty \overline{f}^{-k}(R_k)$ is a point $a$ in $R_0$. One check that $(\epsilon_k)$ is a cycle determined by the periodic orbit of $a$.
\end{proof}

Next, we explain how the notion of a rotated/unrotated fixed point translates under this correspondence. We will not need the full knowledge of this in this paper, but we include this discussion in anticipation for future applications.

Suppose we have chosen orientations on $e$ and $e'$. These induce orientations on the unstable leaves in $R$ and $R'$. Let $\epsilon$ be an edge in the directed graph $\Gamma$ from $e$ to $e'$. Let $R_\epsilon$ be the subrectangle in $\overline{f}(R) \cap R'$ corresponding to $\epsilon$. We say that $\epsilon$ is \textbf{orientation-preserving} if $\overline{f}$ restricted to a map from $\overline{f}^{-1}(R_\epsilon) \subset R$ to $R_\epsilon \subset R'$ preserves the orientations on the unstable leaves. Otherwise we say that $\epsilon$ is \textbf{orientation-reversing}.
Equivalently, $\epsilon$ is orientation-preserving if and only if the corresponding appearance of $e'$ in the edge path $g(e)$ is positive.

We say that a cycle in $\Gamma$ is \textbf{orientation-preserving/reversing} if it consists of an even/odd number of orientation-reversing edges, respectively.
It is straightforward to deduce the following proposition from these definitions.

\begin{prop} \label{prop:pointtocyclerotatedness}
Suppose $(\epsilon_k)$ is a primitive cycle of length $p$. If $(\epsilon_k)$ is determined by a nonsingular periodic orbit $\mathbf{a}$, then $(\epsilon_k)$ is orientation-preserving if and only if the elements of $\mathbf{a}$ are unrotated fixed points of $f^p$. If $(\epsilon_k)$ is determined by a singular periodic orbit, then $(\epsilon_k)$ is always orientation-preserving.
\end{prop}

With this terminology, we can state the following proposition.

\begin{prop} \label{prop:curvesenterexit}
Let $g$ be a train track map carrying $f$. Let $\alpha =(\epsilon_k:e_k \to e_{k+1})_{k \in \mathbb{Z}/P}$ be an orientation-preserving cycle determined by a periodic orbit $\mathbf{a}$. 
\begin{itemize}
    \item If $\mathbf{a} \not\subset \mathcal{X}_O$, then there are at least two edges entering $\alpha$, i.e. 
    $$\sum_k \sum_e ((g_*^\real)_{e_k,e}-1) \geq 2.$$
    \item If $\mathbf{a} \not\subset \mathcal{X}_I$, then there are at least two edges exiting $\alpha$, i.e.
    $$\sum_k \sum_e ((g_*^\real)_{e,e_k}-1) \geq 2.$$
\end{itemize}
Here the inner sums are taken over all vertices $e$ of $\Gamma$.
\end{prop}
\begin{proof}
Let $R_k$ be the rectangle corresponding to $e_k$, and write $\mathbf{a} = \{f^k(a)\}$ where $f^k(a) \in R_k$. As in the proof of \Cref{prop:periodicorbittocycle}, we will write $\overline{f}^k(R_{j+k}) \cap R_j$ to denote the subrectangle corresponding to the path $\epsilon_j \cdots \epsilon_{j+k}$.
Also, the orientation on the real edges $e_k$ induces orientations on the unstable leaves in $R_k$, thus we can coherently refer to the two unstable sides of $R_k$ as the top and bottom unstable sides.

Suppose $\mathbf{a} \not\subset \mathcal{X}_O$.
Then no unstable side of $\overline{f}^{p}(R_0) \cap R_0$ can lie on an unstable side of $R_0$, since otherwise $a$ would lie on an unstable side of $R_0$.
Note that the assumption that $\alpha$ is orientation-preserving is used here.

Let $k_1$ be the minimum number where the top unstable side of $\overline{f}^{k_1}(R_0) \cap R_{k_1}$ does not lie on an unstable side of $R_{k_1}$.
Then $\overline{f}$ must map some rectangle other than $R_{k_1-1}$ through $R_{k_1}$.
Similarly, let $k_2$ be the minimum number where the bottom unstable side of $\overline{f}^{k_2}(R_0) \cap R_{k_2}$ does not lie on an unstable side of $R_{k_2}$. 
Then $\overline{f}$ must map some rectangle other than $R_{k_2-1}$ through $R_{k_2}$. Hence $\sum_k \sum_e ((g_*^\real)_{e_k,e}-1) \geq 2$, with positive contributions coming from $k = k_1, k_2$.

A similar argument using stable sides proves the second statement.
\end{proof}

The cycles that are determined by singularity orbits will play an outsized role in the proof of \Cref{thm:braiddillowerbound}. We show the following additional properties regarding these cycles.

\begin{prop} \label{prop:innerouteremb}
Each vertex of $\Gamma$ can appear at most twice in the collection of cycles determined by the periodic orbits in $\mathcal{X}_I$. 
The same statement holds with $\mathcal{X}_I$ replaced by $\mathcal{X}_O$.
\end{prop}
\begin{proof}
Each time a vertex $e$ appears in a cycle determined by $\mathbf{a} \subset \mathcal{X}_I$, the rectangle $R$ corresponding to $e$ contains a point in $\mathbf{a}$ in one of its stable side. Since $R$ has only two stable sides, and no two periodic points can lie along the same stable leaf, $e$ cannot appear more than twice in the collection of cycles determined by periodic points in $\mathcal{X}_I$.

A similar argument using unstable sides proves the second statement.
\end{proof}

\begin{cor} \label{cor:inneroutercurveemb}
Each curve of $\Gamma$ can be contained in at most one cycle determined by the periodic orbits in $\mathcal{X}_I$.
The same statement holds with $\mathcal{X}_I$ replaced by $\mathcal{X}_O$.
\end{cor}
\begin{proof}
If a curve $\gamma$ is contained in a cycle $\alpha$ determined by a periodic orbit in $\mathcal{X}_I$, but $\gamma \neq \alpha$, then there must be a vertex on $\gamma$ that appears twice in $\alpha$. For example, the vertex where $\alpha$ enters $\gamma$ is such a vertex.
In this case, by \Cref{prop:innerouteremb}, $\gamma$ cannot be contained in another cycle determined by a periodic orbit in $\mathcal{X}_I$.

On the other hand, the cycles determined by periodic orbits in $\mathcal{X}_I$ are all distinct, thus in particular those cycles that are embedded are also all distinct. 
\end{proof}

\section{Floral train tracks} \label{sec:floraltt}

In this section, we develop the theory of floral train tracks. We will show that every pseudo-Anosov braid (up to possibly puncturing an extra point) is carried by a floral train track map. We also establish some properties regarding the real transition matrix of a floral train track map.

\subsection{Jointless train track maps} \label{subsec:jointlesstt}

Let $f$ be a pseudo-Anosov braid and let $\overline{f}$ be the completion of $f$.
By definition, $f$ has a fixed puncture, which we denote by $b$.
Note that $b$ may or may not be 1-pronged. We denote the union of the 1-pronged singular points of $f$ other than (possibly) $b$ by $\mathbf{a}$.

\begin{lemma} \label{lemma:pabraidtworotfixedpoint}
Aside from $b$ and (possibly) $\mathbf{a}$, $\overline{f}$ has at least one other fixed point $c$.
\end{lemma}
\begin{proof}
From \Cref{prop:lefschetzbraid} we know that $\sum_x \indL(\overline{f},x) = 2$.
From \Cref{eq:indL}, we observe that every unrotated fixed point contributes a nonpositive number, while every rotated fixed point contributes a $1$. Thus there are at least two rotated fixed points. 
Every 1-pronged singular point is unrotated, so there is at least one rotated fixed point other than (possibly) $b$.
\end{proof}

Note that a priori $c$ may be nonsingular. If this is the case, we puncture it out to make it into a 2-pronged punctured singular point.
(The only reason for doing this is to stay consistent with our convention that boundary components of train tracks correspond to singular points.)

We first discuss a type of train track that is more general than floral train tracks.

\begin{defn}
A train track map $g:\tau \to \tau$ that carries $f$ is \textbf{jointless} if it is based with respect to a partition $\mathcal{X}=\mathcal{X}_P \sqcup \mathcal{X}_A \sqcup \mathcal{X}_O$ where $\mathbf{a} \subset \mathcal{X}_A$.
\end{defn}

We refer to the train track on which a jointless train track map is defined on to be a \textbf{jointless} train track.
Our terminology is inspired by the fact that a 1-pronged infinitesimal polygon that meets more than one real edge is called a \textbf{joint} in \cite{FRW22}.

Let $\tau$ be a jointless train track.
We refer to a real edge of $\tau$ that has one endpoint lying on a 1-pronged infinitesimal polygon as a \textbf{filament}. 
Since no cusps of outer boundary components can lie on 1-pronged infinitesimal polygon, the other endpoint of a filament cannot lie on a 1-pronged infinitesimal polygon, otherwise $\tau$ just consists of the filament between two 1-pronged infinitesimal polygons, contradicting nonemptiness of $\mathcal{X}_P$.

We will always orient a filament from the endpoint that lies on a 1-pronged infinitesimal polygon to the endpoint that does not do so.

We refer to the rest of the real edges of $\tau$ as the \textbf{petals}. 
We will not place any canonical orientation on the petals.

\Cref{thm:basedsettexists}(3) states that any jointless train track map $g$ can be factorized into $g=hg_k...g_1$ where $g_i$ are based folds and $h$ is a train track isomorphism.
The following lemma classifies each of the based folds.

\begin{lemma} \label{lemma:basedfoldtype}
Each based fold in $g$ is of one of the following forms:
\begin{enumerate}
    \item An elementary fold of a filament $e_1$ onto another filament $e_0$, followed by $k \geq 0$ elementary folds of $e_0$ onto $e_1$, followed by an elementary fold of $e_1$ onto $e_0$. 
    The real transition matrix of this based fold is the identity matrix except with a block of 
    $\begin{bmatrix} k+1 & k+2 \\ k & k+1 \end{bmatrix}$
    in the $\{e_0,e_1\} \times \{e_0,e_1\}$ entries.
    After the based fold, $e_0$ and $e_1$ remain as filaments.
    
    \item An elementary fold of a petal $e_1$ onto a filament $e_0$, followed by another elementary fold of $e_1$ onto $e_0$. 
    The real transition matrix of this based fold is the identity matrix except with a block of 
    $\begin{bmatrix} 1 & 2 \\ 0 & 1 \end{bmatrix}$
    in the $\{e_0,e_1\} \times \{e_0,e_1\}$ entries.
    After the based fold, $e_0$ remains as a filament and $e_1$ remains as a petal.
    
    \item An elementary fold of a petal $e_1$ onto a filament $e_0$, followed by an elementary fold of $e_0$ onto $e_1$.
    The real transition matrix of this based fold is the identity matrix except with a block of 
    $\begin{bmatrix} 1 & 1 \\ 1 & 2 \end{bmatrix}$
    in the $\{e_0,e_1\} \times \{e_0,e_1\}$ entries.
    After the based fold, $e_0$ becomes a petal while $e_1$ becomes a filament.
    
    \item An elementary fold of a filament or a petal $e_1$ onto a petal $e_0$.
    The real transition matrix of this based fold is the identity matrix except with a block of 
    $\begin{bmatrix} 1 & 1 \\ 0 & 1 \end{bmatrix}$
    in the $\{e_0,e_1\} \times \{e_0,e_1\}$ entries.
    After the based fold, $e_0$ remains as a petal and $e_1$ remains as a filament or a petal respectively.
\end{enumerate}
\end{lemma}
\begin{proof}
Consider the first elementary fold of a based fold. Using the notation of \Cref{constr:elementaryfold}, $e_0$ and $e_1$ are each either a filament or a petal.

If $e_0$ and $e_1$ are filaments that meet 1-pronged infinitesimal polygons $a_0$ and $a_1$ respectively, then after the first elementary fold, $e_1$ now connects $a_0$ to $a_1$. The next elementary fold must fold $e_1$ onto $e_0$ or fold $e_0$ onto $e_1$. In the former case, the based fold is concluded. In the latter case, we return to the same setting, and we repeat the argument. Eventually, the based fold must conclude. 
See \Cref{fig:basefoldtypes} top row.
This gives case (1) of the lemma. The transition matrix for this based fold can be computed by multiplying the transition matrices for the individual elementary folds.

If $e_0$ is a filament meeting a 1-pronged infinitesimal polygon $a_0$ while $e_1$ is a petal, then after the first elementary fold $e_1$ connects $\partial_P \tau$ to $a_0$. The next elementary fold must fold $e_1$ onto $e_0$ or fold $e_0$ onto $e_1$. In either case the based fold concludes afterwards.
See \Cref{fig:basefoldtypes} second row.
This gives cases (2) and (3) of the lemma respectively. The transition matrix for this based fold can be computed directly.

If $e_0$ is a petal, then after the first elementary fold the train track is already based, so the based fold just consists of one elementary fold.
See \Cref{fig:basefoldtypes} bottom two rows.
This gives case (4) of the lemma. The transition matrix for this based fold can be computed directly.
\end{proof}

\begin{figure}
    \centering
    \selectfont \fontsize{8pt}{8pt}
    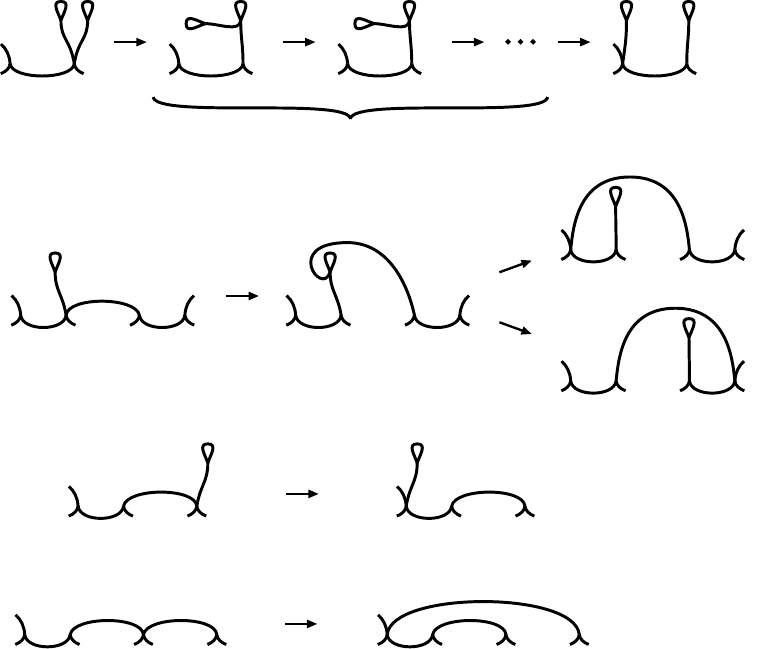
    \caption{The types of based folds for jointless train track maps, as classified by \Cref{lemma:basedfoldtype}.}
    \label{fig:basefoldtypes}
\end{figure}

We associate to each based fold $g_i$ a permutation $\zeta_i$ on the labels of the real edges: 
\begin{itemize}
    \item If $g_i$ is of type (1) where $k$ is odd, then $\zeta_i$ switches $e_0$ and $e_1$.
    \item If $g_i$ is of type (3) then $\zeta_i$ switches $e_0$ and $e_1$.
    \item In every other case, $\zeta_i$ is the identity.
\end{itemize}

Each $\zeta_i$ satisfies the following properties:
\begin{enumerate}
    \item A real edge $e$ is a filament or a petal before the based fold $g_i$  if and only if the real edge $\zeta_i(e)$ is a filament or a petal after the based fold respectively.
    \item The $(\zeta_i(e),e)$-entry of the real transition matrix $(g_i)^\real_*$ is always positive. In fact, there is a canonical positive appearance of $\zeta_i(e)$ in the edge path $g_i(e)$: If $e$ is a filament, then this is the appearance of $\zeta_i(e)$ as the initial edge of $g_i(e)$. If $e$ is a petal, then this is the unique appearance of $\zeta_i(e)$ in $g_i(e)$.
    \item The $(e',e)$-entry of $(g_i)^\real_* - \zeta_i$ is even if $e'$ is a filament. In fact, if $e = \zeta_i^{-1}(e')$, then the number of positive appearances of $e'$ in $g_i(e)$ is exactly one more than the number of negative appearances of $e'$ in $g_i(e)$. If $e \neq \zeta_i^{-1}(e')$, then the number of positive appearances of $e'$ in $g_i(e)$ equals the number of negative appearances of $e'$ in $g_i(e)$.
\end{enumerate}

We let $\zeta$ be the composition of the $\zeta_i$ and the permutation induced by the final train track isomorphism $h$ in the factorization of $g$. Then the previous properties imply the following proposition.

\begin{prop} \label{prop:jointlessttprop}
Let $g$ be a jointless train track map that carries $f$.
\begin{enumerate}
    \item A real edge $e$ is a filament or a petal if and only if $\zeta(e)$ is a filament or a petal respectively.
    \item The $(\zeta(e),e)$-entry of the real transition matrix $g^\real_*$ is always positive. In fact, there is a canonical positive appearance of $\zeta(e)$ in the edge path $g(e)$.
    \item The $(e',e)$-entry of $g^\real_* - \zeta$ is even if $e'$ is a filament. In fact, if $e = \zeta^{-1}(e')$, then the number of positive appearances of $e'$ in $g(e)$ is exactly one more than the number of negative appearances of $e'$ in $g(e)$. If $e \neq \zeta^{-1}(e')$, then the number of positive appearances of $e'$ in $g(e)$ equals the number of negative appearances of $e'$ in $g(e)$.
\end{enumerate}
\end{prop}

\Cref{prop:jointlessttprop}(2) gives us a canonical edge $\epsilon_e$ in the directed graph $\Gamma$ associated to $g_*^\real$ from each $e$ to $\zeta(e)$. 

\begin{defn} \label{defn:filamentcurves}
The concatenation of $\epsilon_e$ over the filaments $e$ is a union of curves. We refer to these as the \textbf{filament curves}.
\end{defn}

The following lemma identifies the filament curves.

\begin{lemma} \label{lemma:filamentcurvesid}
The filament curves are exactly the cycles determined by the 1-pronged puncture orbits in $\mathbf{a}$.
\end{lemma}
\begin{proof}
For each $a \in \mathbf{a}$, the initial portion of the unstable leaf $\overline{\ell}^u$ at $a$ must lie in the rectangle $R$ corresponding to the filament $e$ that meets the 1-pronged infinitesimal polygon $a$, since $e$ is the only real edge meeting $a$. This implies that the filaments are exactly the vertices in the cycles determined by the 1-pronged puncture orbits.

In fact, suppose $\{\overline{f}^k(a)\}_{k \in \mathbb{Z}/p}$ is the puncture orbit containing $a$, and suppose $R_k$ are the rectangles containing the initial portions of the unstable leaves of $\overline{f}^k(a)$, with corresponding real edges $e_k$. Then since $g(e_k)$ is incident to $\overline{f}^{k+1}(a)$ at its initial point, the initial portion of $\overline{f}^{k+1}(\overline{\ell}^u)$ lies in the subrectangle of $\overline{f}(R_k) \cap R_{k+1}$ that corresponds to the appearance of $e_{k+1}$ as the initial edge of $g(e_k)$. See \Cref{fig:filamentcurvesid}.

\begin{figure}
    \centering
    \selectfont \fontsize{8pt}{8pt}
    %% Creator: Inkscape 1.3 (0e150ed6c4, 2023-07-21), www.inkscape.org
%% PDF/EPS/PS + LaTeX output extension by Johan Engelen, 2010
%% Accompanies image file 'filamentcurvesid.pdf' (pdf, eps, ps)
%%
%% To include the image in your LaTeX document, write
%%   \input{<filename>.pdf_tex}
%%  instead of
%%   \includegraphics{<filename>.pdf}
%% To scale the image, write
%%   \def\svgwidth{<desired width>}
%%   \input{<filename>.pdf_tex}
%%  instead of
%%   \includegraphics[width=<desired width>]{<filename>.pdf}
%%
%% Images with a different path to the parent latex file can
%% be accessed with the `import' package (which may need to be
%% installed) using
%%   \usepackage{import}
%% in the preamble, and then including the image with
%%   \import{<path to file>}{<filename>.pdf_tex}
%% Alternatively, one can specify
%%   \graphicspath{{<path to file>/}}
%% 
%% For more information, please see info/svg-inkscape on CTAN:
%%   http://tug.ctan.org/tex-archive/info/svg-inkscape
%%
\begingroup%
  \makeatletter%
  \providecommand\color[2][]{%
    \errmessage{(Inkscape) Color is used for the text in Inkscape, but the package 'color.sty' is not loaded}%
    \renewcommand\color[2][]{}%
  }%
  \providecommand\transparent[1]{%
    \errmessage{(Inkscape) Transparency is used (non-zero) for the text in Inkscape, but the package 'transparent.sty' is not loaded}%
    \renewcommand\transparent[1]{}%
  }%
  \providecommand\rotatebox[2]{#2}%
  \newcommand*\fsize{\dimexpr\f@size pt\relax}%
  \newcommand*\lineheight[1]{\fontsize{\fsize}{#1\fsize}\selectfont}%
  \ifx\svgwidth\undefined%
    \setlength{\unitlength}{153.51577999bp}%
    \ifx\svgscale\undefined%
      \relax%
    \else%
      \setlength{\unitlength}{\unitlength * \real{\svgscale}}%
    \fi%
  \else%
    \setlength{\unitlength}{\svgwidth}%
  \fi%
  \global\let\svgwidth\undefined%
  \global\let\svgscale\undefined%
  \makeatother%
  \begin{picture}(1,0.64396879)%
    \lineheight{1}%
    \setlength\tabcolsep{0pt}%
    \put(0,0){\includegraphics[width=\unitlength,page=1]{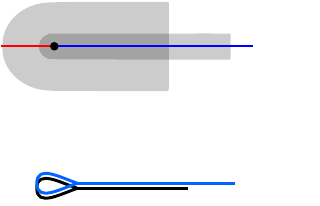}}%
    \put(0.53415621,0.00498729){\color[rgb]{0,0,0}\makebox(0,0)[lt]{\lineheight{1.25}\smash{\begin{tabular}[t]{l}$e_{k+1}$\end{tabular}}}}%
    \put(0.7132094,0.02171641){\color[rgb]{0,0.4,1}\makebox(0,0)[lt]{\lineheight{1.25}\smash{\begin{tabular}[t]{l}$g(e_k)$\end{tabular}}}}%
    \put(0.79936414,0.48934387){\color[rgb]{0,0,1}\makebox(0,0)[lt]{\lineheight{1.25}\smash{\begin{tabular}[t]{l}$\overline{f}^{k+1}(\ell^u)$\end{tabular}}}}%
    \put(0.6073077,0.39209046){\color[rgb]{0,0,0}\makebox(0,0)[lt]{\lineheight{1.25}\smash{\begin{tabular}[t]{l}$\overline{f}(R_k)$\end{tabular}}}}%
    \put(0.27830327,0.29119921){\color[rgb]{0,0,0}\makebox(0,0)[lt]{\lineheight{1.25}\smash{\begin{tabular}[t]{l}$R_{k+1}$\end{tabular}}}}%
    \put(0,0){\includegraphics[width=\unitlength,page=2]{filamentcurvesid.pdf}}%
  \end{picture}%
\endgroup%

    \caption{Since $g(e_k)$ is incident to $\overline{f}^{k+1}(a)$ at its initial point, the initial portion of $\overline{f}^{k+1}(\overline{\ell}^u)$ lies in the subrectangle of $\overline{f}(R_k) \cap R_{k+1}$ that corresponds to the appearance of $e_{k+1}$ as the initial edge of $g(e_k)$.}
    \label{fig:filamentcurvesid}
\end{figure}

This implies that $e_{k+1} = \zeta(e_k)$, and that this appearance of $e_{k+1}$ in $g(e_k)$ is the one corresponding to the edge $\epsilon_{e_k}$ from $e_k$ to $e_{k+1}$. 
Hence the filament curves are exactly the cycles determined by the 1-pronged puncture orbits.
\end{proof}

Unlike filaments, the terminology for petals is more complicated.

\begin{defn} \label{defn:petalcurves}
We define a \textbf{system of petal curves} to be a maximal nonempty collection of disjoint curves that each pass through petals only. 
We refer to the elements of a system of petal curves as the \textbf{petal curves}.
We say that a curve that passes through petals only is \textbf{persistent} if it is contained in every system of petal curves.

For example, the concatenation of the canonical edges $\epsilon_e$ over the petals $e$ determines a system of petal curves.
We refer to this system as the \textbf{canonical system} of petal curves.
The canonical system has the special property that every petal lies in exactly one petal curve.
\end{defn}

Unlike the case for the filament curves, it is unclear to us whether the petal curves in the canonical system depend on the choice of factorization of $g$ into based folds. See \Cref{subsec:petalcurves} for more discussion.

\begin{lemma} \label{lemma:petalpersistentchar}
A curve $\alpha$ that passes through petals only is non-persistent if and only if there is another curve $\delta$ that passes through petals only and which intersects $\alpha$.
\end{lemma}
\begin{proof}
For the forward direction, we can pick $\delta$ to be one of the petal curves in a system that does not contain $\alpha$.
For the backward direction, we can build a system of petal curves by starting with $\delta$ and adding in disjoint curves that pass through petals only until we reach a maximal collection. Such a collection will not contain $\alpha$ since it intersects $\delta$. 
\end{proof}

\Cref{prop:jointlessttprop}(3) says that the edges that enter filament curves always occur in pairs of opposite sign.
We fix for once and for all such a pairing of edges.
Then if $\alpha$ is a curve that passes through filaments and is not a filament curve itself, then one can define another curve $\alpha'$ by following $\alpha$ except for an edge where $\alpha$ enters a filament curve, at which point $\alpha'$ instead takes the other edge in the pair of edges. We note that $\alpha$ and $\alpha'$ have the same length, but have opposite parities in terms of orientation-preserving/reversing. 
In this context, we say that $\alpha'$ is a \textbf{double} of $\alpha$.
Note that $\alpha$ can have multiple doubles. In general, it has as many doubles as the number of times it enters a filament curve.
See \Cref{fig:jointlessttgraph} for an example.

\begin{figure}
    \centering
    \selectfont \fontsize{8pt}{8pt}
    \resizebox{!}{4.5cm}{%% Creator: Inkscape 1.3 (0e150ed6c4, 2023-07-21), www.inkscape.org
%% PDF/EPS/PS + LaTeX output extension by Johan Engelen, 2010
%% Accompanies image file 'jointlessttgraph.pdf' (pdf, eps, ps)
%%
%% To include the image in your LaTeX document, write
%%   \input{<filename>.pdf_tex}
%%  instead of
%%   \includegraphics{<filename>.pdf}
%% To scale the image, write
%%   \def\svgwidth{<desired width>}
%%   \input{<filename>.pdf_tex}
%%  instead of
%%   \includegraphics[width=<desired width>]{<filename>.pdf}
%%
%% Images with a different path to the parent latex file can
%% be accessed with the `import' package (which may need to be
%% installed) using
%%   \usepackage{import}
%% in the preamble, and then including the image with
%%   \import{<path to file>}{<filename>.pdf_tex}
%% Alternatively, one can specify
%%   \graphicspath{{<path to file>/}}
%% 
%% For more information, please see info/svg-inkscape on CTAN:
%%   http://tug.ctan.org/tex-archive/info/svg-inkscape
%%
\begingroup%
  \makeatletter%
  \providecommand\color[2][]{%
    \errmessage{(Inkscape) Color is used for the text in Inkscape, but the package 'color.sty' is not loaded}%
    \renewcommand\color[2][]{}%
  }%
  \providecommand\transparent[1]{%
    \errmessage{(Inkscape) Transparency is used (non-zero) for the text in Inkscape, but the package 'transparent.sty' is not loaded}%
    \renewcommand\transparent[1]{}%
  }%
  \providecommand\rotatebox[2]{#2}%
  \newcommand*\fsize{\dimexpr\f@size pt\relax}%
  \newcommand*\lineheight[1]{\fontsize{\fsize}{#1\fsize}\selectfont}%
  \ifx\svgwidth\undefined%
    \setlength{\unitlength}{121.89256443bp}%
    \ifx\svgscale\undefined%
      \relax%
    \else%
      \setlength{\unitlength}{\unitlength * \real{\svgscale}}%
    \fi%
  \else%
    \setlength{\unitlength}{\svgwidth}%
  \fi%
  \global\let\svgwidth\undefined%
  \global\let\svgscale\undefined%
  \makeatother%
  \begin{picture}(1,0.86143413)%
    \lineheight{1}%
    \setlength\tabcolsep{0pt}%
    \put(0,0){\includegraphics[width=\unitlength,page=1]{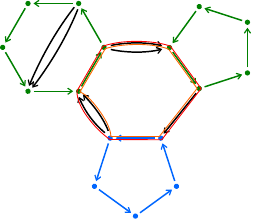}}%
    \put(0.71202898,0.35464945){\color[rgb]{1,0,0}\makebox(0,0)[lt]{\lineheight{1.25}\smash{\begin{tabular}[t]{l}$\alpha$\end{tabular}}}}%
    \put(0.62766631,0.4077049){\color[rgb]{1,0.4,0}\makebox(0,0)[lt]{\lineheight{1.25}\smash{\begin{tabular}[t]{l}$\alpha'$\end{tabular}}}}%
  \end{picture}%
\endgroup%
}
    \caption{A possible directed graph associated to the real transition matrix of a jointless train track map. Here, the green vertices and edges are the filament curves, and the blue vertices and edges are the (canonical) petal curves. For the red curve $\alpha$ in the figure, the orange curve $\alpha'$ is one of its double. Note that $\alpha$ has one other double.}
    \label{fig:jointlessttgraph}
\end{figure}

We define $G_{e',e}=
\begin{cases}
\frac{1}{2}(g^\real_* - \zeta)_{e',e} & \text{if $e'$ is a filament} \\
(g^\real_* - \zeta)_{e',e} & \text{if $e'$ is a petal}
\end{cases}$ 
as a measure of the number of edges with this doubling phenomenon excluded. This notation will come in handy in the next subsection.

\subsection{Floral train track maps} \label{subsec:floraltt}

We are now ready to discuss floral train track maps.

\begin{defn}
A train track map $g:\tau \to \tau$ that carries $f$ is \textbf{floral} if it is based with respect to a partition $\mathcal{X}=\mathcal{X}_P \sqcup \mathcal{X}_A \sqcup \mathcal{X}_O$ where $\mathcal{X}_P = \{c\}$ and $\mathcal{X}_A = \mathbf{a}$ (so that $g$ is in particular jointless), and if the two endpoints of every petal in $\tau$ coincide.
\end{defn}

We refer to the train track on which a floral train track map is defined on to be a \textbf{floral} train track.

\begin{prop} \label{prop:floralttexists}
The pseudo-Anosov braid $f$ (after possibly having punctured the fixed point $c$) is carried by some floral train track map.
\end{prop}
\begin{proof}
We first consider the partition $\mathcal{X}=\mathcal{X}_P \sqcup \widehat{\mathcal{X}}_A \sqcup \widehat{\mathcal{X}}_O$ where $\mathcal{X}_P = \{c\}$ and $\widehat{\mathcal{X}}_O = \{b\}$ and take some based train track partition $\widehat{\mathcal{M}}$ with respect to this partition.

We then apply \Cref{constr:convertaottpartition} to convert $\widehat{\mathcal{M}}$ into a based train track partition $\mathcal{M}$ with respect to the partition $\mathcal{X}=\mathcal{X}_P \sqcup \mathcal{X}_A \sqcup \mathcal{X}_O$ where $\mathcal{X}_P = \{c\}$ and $\mathcal{X}_A = \mathbf{a}$. 

We argue that $\tau = \tau_\mathcal{M}$ is floral. 
We first observe that $\widehat{\tau} = \tau_{\widehat{\mathcal{M}}}$ is of a specific form: We define a \textbf{limb} of $\widehat{\tau}$ to be a maximal connected subset that is disjoint from the infinitesimal polygon $c$. 
If we collapse each infinitesimal polygon in a limb to a point, then we get a tree (with one vertex where the limb is attached to $c$ removed). 
The infinitesimal polygons that correspond to elements in $\mathcal{X}_O \backslash \widehat{\mathcal{X}}_O$ are exactly the infinitesimal polygons in the limbs that get collapsed down to non-leaf vertices of these trees.
Moreover, each limb is disjoint from the cusps of $b$.

To obtain $\mathcal{M}$ from $\widehat{\mathcal{M}}$, we split the rectangles along the unstable stars of elements in $\mathcal{X}_O \backslash \widehat{\mathcal{X}}_O$ until we hit the stable star of $c$. Correspondingly, we split the cusps of the infinitesimal polygons in $\widehat{\tau}$ that correspond to elements in $\mathcal{X}_O \backslash \widehat{\mathcal{X}}_O$ until we hit the infinitesimal polygon $c$.

This procedure turns each limb into a union of filaments and petals (and 1-pronged infinitesimal polygons), see \Cref{fig:antennaltofloral}. Since each limb is attached to $b$ at one vertex, the endpoints of each petal have to lie at that same vertex.
\end{proof}

\begin{figure}
    \centering
    %% Creator: Inkscape 1.3 (0e150ed6c4, 2023-07-21), www.inkscape.org
%% PDF/EPS/PS + LaTeX output extension by Johan Engelen, 2010
%% Accompanies image file 'antennaltofloral.pdf' (pdf, eps, ps)
%%
%% To include the image in your LaTeX document, write
%%   \input{<filename>.pdf_tex}
%%  instead of
%%   \includegraphics{<filename>.pdf}
%% To scale the image, write
%%   \def\svgwidth{<desired width>}
%%   \input{<filename>.pdf_tex}
%%  instead of
%%   \includegraphics[width=<desired width>]{<filename>.pdf}
%%
%% Images with a different path to the parent latex file can
%% be accessed with the `import' package (which may need to be
%% installed) using
%%   \usepackage{import}
%% in the preamble, and then including the image with
%%   \import{<path to file>}{<filename>.pdf_tex}
%% Alternatively, one can specify
%%   \graphicspath{{<path to file>/}}
%% 
%% For more information, please see info/svg-inkscape on CTAN:
%%   http://tug.ctan.org/tex-archive/info/svg-inkscape
%%
\begingroup%
  \makeatletter%
  \providecommand\color[2][]{%
    \errmessage{(Inkscape) Color is used for the text in Inkscape, but the package 'color.sty' is not loaded}%
    \renewcommand\color[2][]{}%
  }%
  \providecommand\transparent[1]{%
    \errmessage{(Inkscape) Transparency is used (non-zero) for the text in Inkscape, but the package 'transparent.sty' is not loaded}%
    \renewcommand\transparent[1]{}%
  }%
  \providecommand\rotatebox[2]{#2}%
  \newcommand*\fsize{\dimexpr\f@size pt\relax}%
  \newcommand*\lineheight[1]{\fontsize{\fsize}{#1\fsize}\selectfont}%
  \ifx\svgwidth\undefined%
    \setlength{\unitlength}{300.77073705bp}%
    \ifx\svgscale\undefined%
      \relax%
    \else%
      \setlength{\unitlength}{\unitlength * \real{\svgscale}}%
    \fi%
  \else%
    \setlength{\unitlength}{\svgwidth}%
  \fi%
  \global\let\svgwidth\undefined%
  \global\let\svgscale\undefined%
  \makeatother%
  \begin{picture}(1,0.63810698)%
    \lineheight{1}%
    \setlength\tabcolsep{0pt}%
    \put(0,0){\includegraphics[width=\unitlength,page=1]{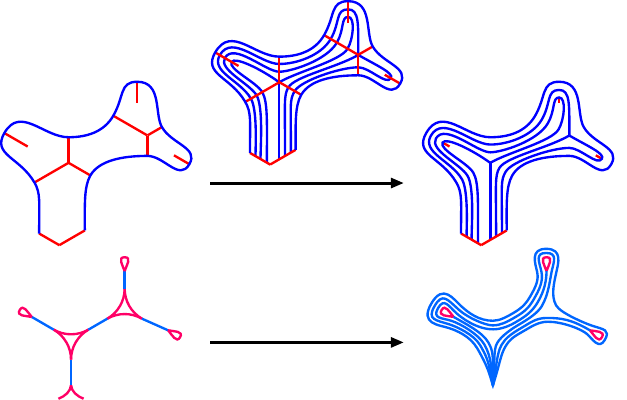}}%
    \put(0.22346439,0.2784118){\color[rgb]{0,0,0}\makebox(0,0)[lt]{\lineheight{1.25}\smash{\begin{tabular}[t]{l}$\widehat{\mathcal{M}}$\end{tabular}}}}%
    \put(0.22081919,0.01742466){\color[rgb]{0,0,0}\makebox(0,0)[lt]{\lineheight{1.25}\smash{\begin{tabular}[t]{l}$\widehat{\tau}$\end{tabular}}}}%
    \put(0.89673677,0.2784118){\color[rgb]{0,0,0}\makebox(0,0)[lt]{\lineheight{1.25}\smash{\begin{tabular}[t]{l}$\mathcal{M}$\end{tabular}}}}%
    \put(0,0){\includegraphics[width=\unitlength,page=2]{antennaltofloral.pdf}}%
    \put(0.89409156,0.01742466){\color[rgb]{0,0,0}\makebox(0,0)[lt]{\lineheight{1.25}\smash{\begin{tabular}[t]{l}$\tau$\end{tabular}}}}%
    \put(0,0){\includegraphics[width=\unitlength,page=3]{antennaltofloral.pdf}}%
  \end{picture}%
\endgroup%

    \caption{Constructing a floral train track.}
    \label{fig:antennaltofloral}
\end{figure}

Recall that, by \Cref{prop:jointlessttprop}(4), the edges that enter filaments always occur in pairs.
The following proposition can be thought of as a dual property where the edges that exit petals always occur in `pairs'.

\begin{prop} \label{prop:floralttpetalexit}
Let $g:\tau \to \tau$ be a floral train track map carrying $f$. For every petal $p$, $\sum_{e} G_{e,p}$ is even, where the sum is taken over all vertices $e$ of $\Gamma$.
\end{prop}
\begin{proof}
A smooth edge path $E$ in $\tau$ that starts and ends at vertices of the infinitesimal polygon $c$ must be a concatenation $g_1 c_1 g_2 c_2 \cdots g_L c_L g_{L+1}$, where each $g_i$ is of the form
\begin{itemize}
    \item $f_i^{-1} a_i f_i$, where $f_i$ is a filament meeting a 1-pronged infinitesimal polygon $a_i$, or
    \item $p_i$, where $p_i$ is a petal,
\end{itemize}
and each $c_i$ is an edge in $c$.
In particular, the endpoints of each $g_i$ lie on the same vertex, while the endpoints of each $c_i$ lie on adjacent vertices.

We apply this observation to the smooth edge path $g(e)$. 
Here we have $\sum_{e'} G_{e',e} = L$.
But since the two endpoints of $e$ lie on the same vertex, the same is true for $g(e)$, so $L$ must be even.
\end{proof}

However, notice that \Cref{prop:floralttpetalexit} is not a perfect dual of \Cref{prop:jointlessttprop}(4) since the `pair' of edges may not have the same terminal points, i.e. we do not claim that each individual $G_{e,p}$ is even.
See \Cref{fig:floralttgraph} for an example.

\begin{figure}
    \centering
    \selectfont \fontsize{8pt}{8pt}
    \resizebox{!}{4.5cm}{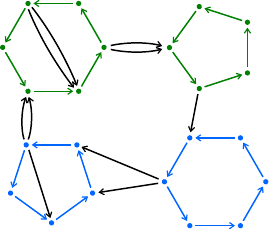}
    \caption{A possible directed graph associated to the real transition matrix of a floral train track map. As in \Cref{fig:jointlessttgraph}, the green vertices and edges are the filament curves, and the blue vertices and edges are the (canonical) petal curves.}
    \label{fig:floralttgraph}
\end{figure}

Finally, we record the following fact.

\begin{lemma} \label{lemma:primarysingemb}
The cycles $\beta_1,...,\beta_{\frac{p_0}{p}}$ determined by $b$ are each embedded and are mutually disjoint. Similarly, the cycles $\gamma_1,...,\gamma_{\frac{q_0}{q}}$ determined by $c$ are each embedded and are mutually disjoint.
\end{lemma}
\begin{proof}
For some $\beta_i$ to be non-embedded, there must be some real edge of $\tau$ that lies on two sides of the boundary component $b$. But each filament, despite possibly appearing twice on $b$, lies on at most a single side, and each petal appears at most once on $b$, so this is impossible. The same reason implies that the cycles determined by $b$ are mutually disjoint.

For some $\gamma_i$ to be non-embedded, there must be some real edge of $\tau$ that meets two vertices of the infinitesimal polygon $c$, but this is not the case by the definition of floral. The same reason implies that the cycles determined by $c$ are mutually disjoint.
\end{proof}

\subsection{Examples} \label{subsec:examples}

We record two examples of floral train track maps in this subsection. We thank Eiko Kin for suggesting to us these examples.

\begin{eg}
Let $\beta_{1,1}$ be the simplest hyperbolic 3-strand braid, see \Cref{fig:beta11braid} left. The monodromy $f_{1,1}$ of $\beta_{1,1}$ is carried by the train track map in \Cref{fig:beta11braid} center. In \Cref{fig:beta11braid} right we show the directed graph associated to the real transition matrix of this train track map.

\begin{figure}
    \centering
    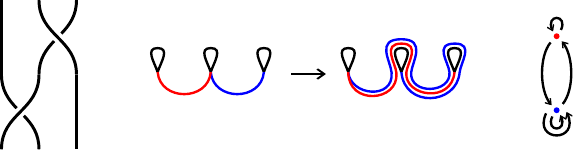
    \caption{Left: The braid $\beta_{1,1}$. Center: A standardly embedded train track map carrying the monodromy $f_{1,1}$ of $\beta_{1,1}$. Right: The directed graph associated to the real transition matrix of this train track map.}
    \label{fig:beta11braid}
\end{figure}

It can be shown that the completion $\overline{f_{1,1}}$ has three fixed points. One of the fixed points is a puncture $b$, namely, the outer boundary component of $\beta_{1,1}$, and it is 1-pronged. The remaining two fixed points are both non-punctured and 2-pronged, i.e. they are non-singular points. To apply the floral train track machinery, we puncture at one of the fixed points $c$.

Actually, it can be shown that the two pseudo-Anosov maps obtained by puncturing each of these two fixed points are conjugate, so the choice of $c$ here is immaterial.

The floral train track map in \Cref{fig:beta11floral} left carries the punctured pseudo-Anosov map $f^\circ_{1,1}$. In \Cref{fig:beta11floral} right we show the directed graph associated to the real transition matrix of this floral train track map. In this example there is one filament curve and there are no petals.

\begin{figure}
    \centering
    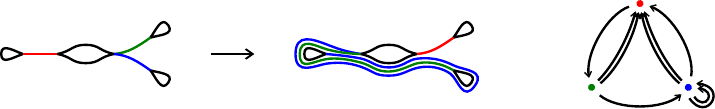
    \caption{Left: A floral train track map carrying $f^\circ_{1,1}$. Right: The directed graph associated to the real transition matrix of this floral train track map.}
    \label{fig:beta11floral}
\end{figure}

\end{eg}

\begin{eg}
Let $\beta_{2,3}$ be the 6-strand braid shown in \Cref{fig:beta23braid} left. The monodromy $f_{2,3}$ of $\beta_{2,3}$ is carried by the train track map in \Cref{fig:beta23braid} center, as shown in \cite{HK06}. In \Cref{fig:beta23braid} right we show the directed graph associated to the real transition matrix of this train track map.

\begin{figure}
    \centering
    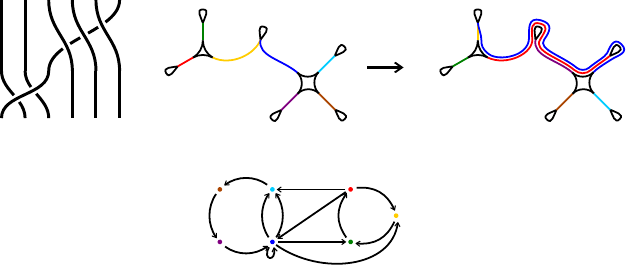
    \caption{Left: The braid $\beta_{2,3}$. Right: A standardly embedded train track map carrying the monodromy $f_{2,3}$ of $\beta_{2,3}$. Bottom: The directed graph associated to the real transition matrix of this train track map.}
    \label{fig:beta23braid}
\end{figure}

It can be shown that the completion $\overline{f_{2,3}}$ has three fixed points. One of the fixed points is a puncture $b$, namely, the outer boundary component of $\beta_{1,1}$, and it is 1-pronged. The remaining two fixed points are the 4-pronged singular point and the 3-pronged singular point.
Here, we do not puncture these fixed points out since they are already singular.

We follow the algorithm of \Cref{prop:floralttexists} in order to construct a floral train track map carrying $f_{2,3}$ with the choice of $c$ as the 4-pronged singular point:
In \Cref{fig:beta234prongfloral} top row, we first take a jointless train track map carrying $f_{2,3}$ that is based at the 4-pronged singular point. This can be obtained by a straightforward modification of the train track from \Cref{fig:beta23braid}. 
In \Cref{fig:beta234prongfloral} second row, we then split the cusps of the 3-pronged singular point until they reach the vertices of the 4-pronged singular point. This gives the desired floral train track map.

\begin{figure}
    \centering
    \selectfont \fontsize{6pt}{6pt}
    %% Creator: Inkscape 1.3 (0e150ed6c4, 2023-07-21), www.inkscape.org
%% PDF/EPS/PS + LaTeX output extension by Johan Engelen, 2010
%% Accompanies image file 'beta234prongfloral.pdf' (pdf, eps, ps)
%%
%% To include the image in your LaTeX document, write
%%   \input{<filename>.pdf_tex}
%%  instead of
%%   \includegraphics{<filename>.pdf}
%% To scale the image, write
%%   \def\svgwidth{<desired width>}
%%   \input{<filename>.pdf_tex}
%%  instead of
%%   \includegraphics[width=<desired width>]{<filename>.pdf}
%%
%% Images with a different path to the parent latex file can
%% be accessed with the `import' package (which may need to be
%% installed) using
%%   \usepackage{import}
%% in the preamble, and then including the image with
%%   \import{<path to file>}{<filename>.pdf_tex}
%% Alternatively, one can specify
%%   \graphicspath{{<path to file>/}}
%% 
%% For more information, please see info/svg-inkscape on CTAN:
%%   http://tug.ctan.org/tex-archive/info/svg-inkscape
%%
\begingroup%
  \makeatletter%
  \providecommand\color[2][]{%
    \errmessage{(Inkscape) Color is used for the text in Inkscape, but the package 'color.sty' is not loaded}%
    \renewcommand\color[2][]{}%
  }%
  \providecommand\transparent[1]{%
    \errmessage{(Inkscape) Transparency is used (non-zero) for the text in Inkscape, but the package 'transparent.sty' is not loaded}%
    \renewcommand\transparent[1]{}%
  }%
  \providecommand\rotatebox[2]{#2}%
  \newcommand*\fsize{\dimexpr\f@size pt\relax}%
  \newcommand*\lineheight[1]{\fontsize{\fsize}{#1\fsize}\selectfont}%
  \ifx\svgwidth\undefined%
    \setlength{\unitlength}{234.05647578bp}%
    \ifx\svgscale\undefined%
      \relax%
    \else%
      \setlength{\unitlength}{\unitlength * \real{\svgscale}}%
    \fi%
  \else%
    \setlength{\unitlength}{\svgwidth}%
  \fi%
  \global\let\svgwidth\undefined%
  \global\let\svgscale\undefined%
  \makeatother%
  \begin{picture}(1,1.28447837)%
    \lineheight{1}%
    \setlength\tabcolsep{0pt}%
    \put(0,0){\includegraphics[width=\unitlength,page=1]{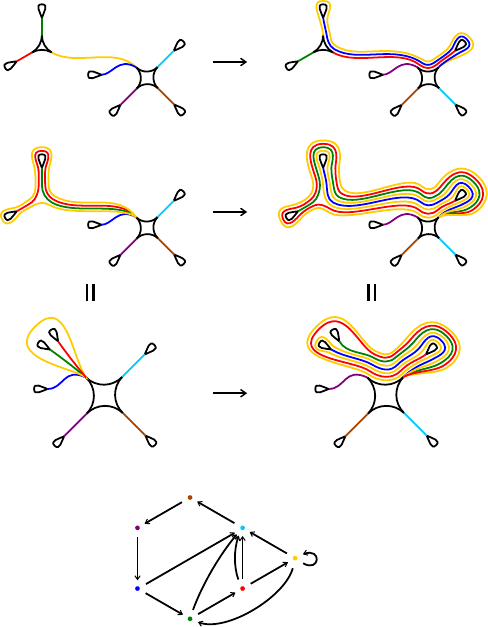}}%
    \put(0.531924,0.01950371){\color[rgb]{0,0,0}\makebox(0,0)[lt]{\lineheight{1.25}\smash{\begin{tabular}[t]{l}$\times 2$\end{tabular}}}}%
    \put(0.54470078,0.18255678){\color[rgb]{0,0,0}\makebox(0,0)[lt]{\lineheight{1.25}\smash{\begin{tabular}[t]{l}$\times 6$\end{tabular}}}}%
    \put(0.43582403,0.09341364){\color[rgb]{0,0,0}\makebox(0,0)[lt]{\lineheight{1.25}\smash{\begin{tabular}[t]{l}$\times 2$\end{tabular}}}}%
    \put(0.37708337,0.10582137){\color[rgb]{0,0,0}\makebox(0,0)[lt]{\lineheight{1.25}\smash{\begin{tabular}[t]{l}$\times 2$\end{tabular}}}}%
    \put(0.3305442,0.13658864){\color[rgb]{0,0,0}\makebox(0,0)[lt]{\lineheight{1.25}\smash{\begin{tabular}[t]{l}$\times 2$\end{tabular}}}}%
  \end{picture}%
\endgroup%

    \caption{Constructing a floral train track map carrying $f_{2,3}$ with the choice of $c$ as the 4-pronged singular point. We also show the directed graph associated to the real transition matrix of this floral train track map.}
    \label{fig:beta234prongfloral}
\end{figure}

In \Cref{fig:beta234prongfloral} third row, we isotope the floral train track map so that the train tracks look more `floral'. In the last row, we also show the directed graph associated to the real transition matrix of this floral train track map.
In this example there is one filament curve of length 6 and one petal curve of length 1.
It is an instructive exercise to locate the cycles determined by each singular point. 

\Cref{fig:beta233prongfloral} shows the result of repeating this computation with the choice of $c$ as the 3-pronged singular point instead.

\begin{figure}
    \centering
    \selectfont \fontsize{6pt}{6pt}
    %% Creator: Inkscape 1.3 (0e150ed6c4, 2023-07-21), www.inkscape.org
%% PDF/EPS/PS + LaTeX output extension by Johan Engelen, 2010
%% Accompanies image file 'beta233prongfloral.pdf' (pdf, eps, ps)
%%
%% To include the image in your LaTeX document, write
%%   \input{<filename>.pdf_tex}
%%  instead of
%%   \includegraphics{<filename>.pdf}
%% To scale the image, write
%%   \def\svgwidth{<desired width>}
%%   \input{<filename>.pdf_tex}
%%  instead of
%%   \includegraphics[width=<desired width>]{<filename>.pdf}
%%
%% Images with a different path to the parent latex file can
%% be accessed with the `import' package (which may need to be
%% installed) using
%%   \usepackage{import}
%% in the preamble, and then including the image with
%%   \import{<path to file>}{<filename>.pdf_tex}
%% Alternatively, one can specify
%%   \graphicspath{{<path to file>/}}
%% 
%% For more information, please see info/svg-inkscape on CTAN:
%%   http://tug.ctan.org/tex-archive/info/svg-inkscape
%%
\begingroup%
  \makeatletter%
  \providecommand\color[2][]{%
    \errmessage{(Inkscape) Color is used for the text in Inkscape, but the package 'color.sty' is not loaded}%
    \renewcommand\color[2][]{}%
  }%
  \providecommand\transparent[1]{%
    \errmessage{(Inkscape) Transparency is used (non-zero) for the text in Inkscape, but the package 'transparent.sty' is not loaded}%
    \renewcommand\transparent[1]{}%
  }%
  \providecommand\rotatebox[2]{#2}%
  \newcommand*\fsize{\dimexpr\f@size pt\relax}%
  \newcommand*\lineheight[1]{\fontsize{\fsize}{#1\fsize}\selectfont}%
  \ifx\svgwidth\undefined%
    \setlength{\unitlength}{197.87356639bp}%
    \ifx\svgscale\undefined%
      \relax%
    \else%
      \setlength{\unitlength}{\unitlength * \real{\svgscale}}%
    \fi%
  \else%
    \setlength{\unitlength}{\svgwidth}%
  \fi%
  \global\let\svgwidth\undefined%
  \global\let\svgscale\undefined%
  \makeatother%
  \begin{picture}(1,1.60892255)%
    \lineheight{1}%
    \setlength\tabcolsep{0pt}%
    \put(0,0){\includegraphics[width=\unitlength,page=1]{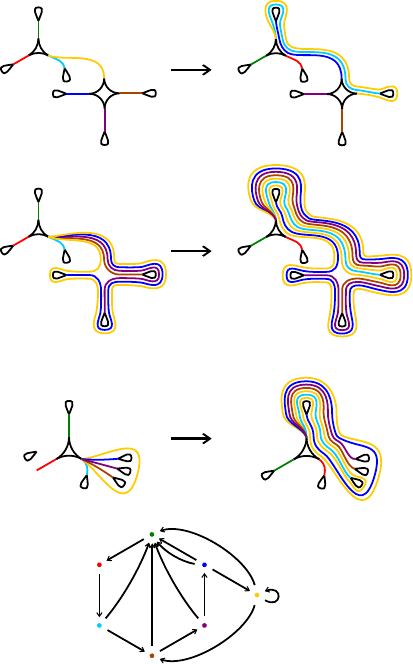}}%
    \put(0.37498284,0.16513391){\color[rgb]{0,0,0}\makebox(0,0)[lt]{\lineheight{1.25}\smash{\begin{tabular}[t]{l}$\times 2$\end{tabular}}}}%
    \put(0.53075155,0.29753029){\color[rgb]{0,0,0}\makebox(0,0)[lt]{\lineheight{1.25}\smash{\begin{tabular}[t]{l}$\times 6$\end{tabular}}}}%
    \put(0.3077828,0.15460485){\color[rgb]{0,0,0}\makebox(0,0)[lt]{\lineheight{1.25}\smash{\begin{tabular}[t]{l}$\times 2$\end{tabular}}}}%
    \put(0.24250985,0.16513347){\color[rgb]{0,0,0}\makebox(0,0)[lt]{\lineheight{1.25}\smash{\begin{tabular}[t]{l}$\times 2$\end{tabular}}}}%
    \put(0.41430179,0.21697787){\color[rgb]{0,0,0}\makebox(0,0)[lt]{\lineheight{1.25}\smash{\begin{tabular}[t]{l}$\times 2$\end{tabular}}}}%
    \put(0.53075155,0.01663475){\color[rgb]{0,0,0}\makebox(0,0)[lt]{\lineheight{1.25}\smash{\begin{tabular}[t]{l}$\times 2$\end{tabular}}}}%
    \put(0,0){\includegraphics[width=\unitlength,page=2]{beta233prongfloral.pdf}}%
  \end{picture}%
\endgroup%

    \caption{Constructing a floral train track map carrying $f_{2,3}$ with the choice of $c$ as the 3-pronged singular point. We also show the directed graph associated to the real transition matrix of this floral train track map.}
    \label{fig:beta233prongfloral}
\end{figure}

\end{eg}

\section{Setup for computations} \label{sec:prelimobs}

In this section, we set up some notation and explain the casework for proving \Cref{thm:braiddillowerbound} and \Cref{thm:braiddilequality}.

\subsection{Inequalities}

Let $f$ be a pseudo-Anosov braid with $n$ strands. Let $\overline{f}$ be the completion of $f$.
Let $\lambda$ be the dilatation of $f$.

We define the singular orbits $\mathbf{a}$, $\{b\}$, $\{c\}$ as in \Cref{subsec:jointlesstt}.
Recall that in defining $c$, we might have created an additional puncture. When we do so, we do not change the value of $n$; we simply allow $f$ to become a braid with $n+1$ strands.

We denote by $m$ the cardinality of $\mathbf{a}$.
We denote by $p_0$ and $q_0$ the number of prongs of $b$ and $c$, respectively.

We refer to the singular points in $\mathbf{a} \cup \{b,c\}$ as the \textbf{primary} singular points, and refer to the remaining singular points as the \textbf{secondary} singular points.
We denote by $r_0$ the number of secondary singular points. This number may be zero or positive.

\begin{lemma} \label{lemma:numberoffilamentspetals}
If $g:\tau \to \tau$ is a floral train track map that carries $f$. Then the number of filaments in $\tau$ equals $m$ and the number of petals in $\tau$ equals $r_0$.
\end{lemma}
\begin{proof}
The number of filaments equals the cardinality of $\mathbf{a}$, which is $m$.
Meanwhile, the number of boundary components of $\tau$ equals the number of singular points, which is $m+r_0+2$.
By \Cref{prop:realedgenumber}, the number of real edges equals $(m+r_0+2)-2 = m+r_0$. Hence the number of petals equals $r_0$.
\end{proof}

\begin{lemma} \label{lemma:accounting}
The numbers $m,n,p_0,q_0,r_0$ satisfy $m \leq n$ and $p_0+q_0+r_0 \leq n$.
\end{lemma}
\begin{proof}
That $m \leq n$ follows from the definition of $m$ and $n$.

For the second inequality, let us first suppose that $c$ is unpunctured.
The secondary singular points can be separated into the punctured ones and the unpunctured ones. 
Since we are assuming that $c$ is unpunctured, there are $n-m$ punctured secdonary singular points.
Each of these must be $\geq 2$-pronged, hence have Poincaré-Hopf index $\leq -1$. 
Meanwhile, there are $r_0-(n-m)$ unpunctured secondary singular points. Each of these must be $\geq 3$-pronged, hence have Poincaré-Hopf index $\leq -\frac{1}{2}$.

Applying this observation to \Cref{thm:poincarehopf}, we have 
$$\chi(S) = \sum_{a \in \mathbf{a}} \indPH(a) + \indPH(b) + \indPH(c) + \sum_{\text{punc.}} \indPH + \sum_{\text{unpunc.}} \indPH$$

\begin{equation} \label{eq:accounting}
1-n \leq -\frac{m}{2} -\frac{p_0}{2} + (1-\frac{q_0}{2}) - (n-m) - \frac{r_0-(n-m)}{2}
\end{equation}
Hence $p_0+q_0+r_0 \leq n$.

If $c$ is punctured and this is not because of an additional puncture we did in order to make $c$ singular, then in \Cref{eq:accounting}, one replaces $1-\frac{q_0}{2}$ with $-\frac{q_0}{2}$ and replaces $n-m$ with $n-m-1$. Thus both sides of \Cref{eq:accounting} stay the same, so one ends up with the same inequality.

Finally, if $c$ is punctured and this is because of an additional puncture we did in order to make $c$ singular, then in \Cref{eq:accounting}, one replaces $1-n$ with $-n$ and replaces $1-\frac{q_0}{2}$ with $-\frac{q_0}{2}$. Thus one ends up with the same inequality again.
\end{proof}

\subsection{Division of cases} \label{subsec:casedivision}

Let $g: \tau \to \tau$ be a floral train track that carries $f$. 
Let $\Gamma$ be the directed graph associated to the real transition matrix $g_*^\real$.
Let $(G,w)$ be the curve complex of $\Gamma$.

Let $\beta$ be some cycle determined by $b$ and $\gamma$ be some cycle determined by $c$ (see \Cref{subsec:periodicorbitscorr}).

Our analysis is divided into the following six cases.

\begin{enumerate}[label=(\Roman*)]

    \item Both $\beta$ and $\gamma$ pass through filaments.

    \item One of $\beta$ and $\gamma$ passes through both filaments and petals while the other passes through petals only.

    \item One of $\beta$ and $\gamma$ passes through filaments only while the other passes through petals only and is not persistent.

    \item One of $\beta$ and $\gamma$ passes through filaments only while the other passes through petals only and is persistent.

    \item Both $\beta$ and $\gamma$ pass through petals only, and at least one of them is not persistent.

    \item Both $\beta$ and $\gamma$ pass through petals only, and both are persistent.
\end{enumerate}

Amomg the six cases, case I is the most important since it is the only one where we will encounter the values $\underline{\delta}_n$.
For a better reading experience, we will include the details of only this case in the main paper. See \Cref{sec:computation} for the remaining cases.

\section{Proof of main theorems} \label{sec:mainthmproof}

This section consists of three subsections. In \Cref{sec:caseI}, we will explain the computations in case I. The result of the computations is recorded in \Cref{prop:caseI}. In \Cref{subsec:casesII-VI}, we will state \Cref{prop:casesII-VI}, which summarizes the computations in cases II - VI. In \Cref{subsec:mainthmproof}, we will explain how \Cref{thm:braiddillowerbound} and \Cref{thm:braiddilequality} follow from these two propositions.

\subsection{Case I: Both $\beta$ and $\gamma$ pass through filaments} \label{sec:caseI}

The goal of this subsection is to prove the following proposition.

\begin{prop} \label{prop:caseI}
Suppose there are two cycles $\beta$ and $\gamma$ in the directed graph $\Gamma$ satisfying:
\begin{itemize}
    \item $\len(\beta) + \len(\gamma) + r_0 \leq n$,
    \item $\beta$ and $\gamma$ are orientation-preserving, and
    \item $\beta$ and $\gamma$ pass through filaments and are not filament curves.
\end{itemize}
Then $\lambda^n \geq \min\{14.5, \underline{\delta}_n^n\}$.

Moreover, if $\lambda^n < 14.5$, then:
\begin{enumerate}
    \item $\Gamma$ consists of one filament curve $\alpha$ and three other pairs of edges $\{e_i, e'_i\}$, $i=1,2,3$, with endpoints on $\alpha$. In particular there are no petals. 
    \item For $i=1,2,3$, let $\mu_i$ be the curve obtained by traversing $e_i$, then following $\alpha$ until the curve closes up. Then up to a permutation:
    \begin{enumerate}
        \item $\mu_1$ and $\mu_2$ are disjoint, but $\mu_3$ meets $\mu_1$ and $\mu_2$. In particular the only curves in $\Gamma$ are $\alpha$, $\mu_1$, $\mu'_1$, $\mu_2$, $\mu'_2$, $\mu_3$, $\mu'_3$.
        \item $\beta = \mu_1$ and $\gamma = \mu_2$, and
        \item the lengths of $\mu_1,\mu_2,\mu_3$ are, respectively, 
        $$\begin{cases}
        (k,k+1,2k+1) & \text{if $m=2k+1$} \\
        (2k-1,2k+1,4k) & \text{if $m=4k$} \\
        (2k-1,2k+3,4k+2) & \text{if $m=4k+2$}
        \end{cases}.$$
    \end{enumerate}
\end{enumerate}

\end{prop}

We first argue that we can assume $\beta$ and $\gamma$ are embedded.

\begin{lemma} \label{lemma:halfnemb}
If there is a cycle $\mu$ in $\Gamma$ with length $\leq \frac{n}{2}$ and which consists of at least two distinct curves, then $\lambda^n \geq 16$.
\end{lemma}
\begin{proof}
If $\mu$ contains at least two distinct curves, then we can find two curves $\mu_1$, $\mu_2$ in $\mu$ such that $\mu_1$ intersects $\mu_2$.

Thus the following subgraph $G_1$ must be an induced subgraph of the curve complex $G$:

\begin{center}
\scalebox{\graphscale}{
\begin{tikzpicture}
\draw[draw=black, fill=black, thin, solid] (0,0) circle (0.1);
\node[black, anchor=south] at (0,0.2) {$u_1$};
\draw[draw=black, fill=black, thin, solid] (1,0) circle (0.1);
\node[black, anchor=south] at (1,0.2) {$u_2$};
\end{tikzpicture}}
\end{center}
where $u_i = \len(\mu_i)$ have sum $\leq \len(\mu) \leq \frac{n}{2}$.

By convexity and monotonicity (\Cref{prop:growthrateprop}(5) and (2)), the minimum growth rate of $G_1$ subject to the condition $u_1+u_2 \leq \frac{n}{2}$ is attained when $u_1=u_2=\frac{n}{4}$. In this case, the clique polynomial of $G_1$ is 
\begin{align*}
Q_1(t) &= 1-2t^{\frac{n}{4}} \\
&= 1-2x^{\frac{1}{4}}
\end{align*}
where we changed variables $x=t^n$.

The reciprocal of the smallest positive root of this polynomial in $x$ is $16$. Thus $\lambda^n \geq \lambda(G_1,w_1)^n \geq 16$.
\end{proof}

\begin{lemma} \label{lemma:leq2curves}
If there is a cycle $\mu$ in $\Gamma$ with length $\leq n$ and which consists of at least three distinct curves, then $\lambda^n \geq 16$.
\end{lemma}
\begin{proof}
If $\mu$ contains at least three distinct curves, then we can find three curves $\mu_1$, $\mu_2$, $\mu_3$ in $\mu$ such that $\mu_2$ intersects $\mu_1$ and $\mu_3$.

Thus, depending on whether $\mu_1$ intersects $\mu_3$, one of the following two subgraphs must be an induced subgraph of the curve complex $G$:

\begin{center}
\scalebox{\graphscale}{
\begin{tikzpicture}
\draw[draw=black, fill=black, thin, solid] (0,0) circle (0.1);
\node[black, anchor=east] at (-0.2,0) {$u_1$};
\draw[draw=black, fill=black, thin, solid] (1,1) circle (0.1);
\node[black, anchor=south] at (1,1.2) {$u_2$};
\draw[draw=black, fill=black, thin, solid] (2,0) circle (0.1);
\node[black, anchor=west] at (2.2,0) {$u_3$};
\end{tikzpicture}
\hspace{3cm}
\begin{tikzpicture}
\draw[draw=black, fill=black, thin, solid] (0,0) circle (0.1);
\node[black, anchor=east] at (-0.2,0) {$u_1$};
\draw[draw=black, fill=black, thin, solid] (1,1) circle (0.1);
\node[black, anchor=south] at (1,1.2) {$u_2$};
\draw[draw=black, fill=black, thin, solid] (2,0) circle (0.1);
\node[black, anchor=west] at (2.2,0) {$u_3$};
\draw[draw=black, thin, solid] (0,0) -- (2,0);
\end{tikzpicture}}
\end{center}
where $u_i = \len(\mu_i)$ have sum $\leq \len(\mu) \leq n$.

The graph on the left is a wide subgraph of the one on the right, so by \Cref{prop:growthrateprop}(4), the growth rate of $G$ is always bounded from below by the growth rate of the graph $G_1$ on the right.

By convexity and monotonicity, the minimum growth rate of $G_1$ subject to the condition $u_1+u_2+u_3 \leq n$ is attained when $u_1=u_3$ and $u_2 = n-2u_1$. In this case, the clique polynomial of $G_1$ is 
\begin{align*}
Q_1(t) &= 1-2t^{u_1}-t^{n-2u_1}+t^{2u_1} \\
&= 1-2x^a-x^{1-2a}+x^{2a}
\end{align*}
where we changed variables $x=t^n$ and $a = \frac{u_1}{n}$.

We compute the minimum value of the reciprocal of the smallest positive root of this polynomial as follows:
For each $a$, the smallest positive root satisfies 
\begin{equation} \label{eq:polynomial}
1-2x^a-x^{1-2a}+x^{2a} = 0.
\end{equation}
Meanwhile, by \cite[Theorem 4.1]{McM15}, this root is an analytic function in $a$, hence at the minimum, we have $\frac{\partial x}{\partial a} = 0$. Taking the partial derivative of \Cref{eq:polynomial} with respect to $a$ and plugging this in, we have
\begin{equation} \label{eq:polynomialpartiala}
-2x^a+2x^{1-2a}+2x^{2a} = 0.
\end{equation}
It remains to solve for $x$ in the system consisting of \Cref{eq:polynomial} and \Cref{eq:polynomialpartiala}.
To do so, we add two copies of \Cref{eq:polynomial} to \Cref{eq:polynomialpartiala} to get
\begin{align*}
2-6x^a+4x^{2a} &= 0 \\
x^a &= \frac{1}{2}
\end{align*}
using the fact that $x^a < 1$.
Putting this back in \Cref{eq:polynomialpartiala}, we get 
\begin{align*}
-1+8x+\frac{1}{2} &= 0 \\
x &= \frac{1}{16}.
\end{align*}
Thus the minimum value of the reciprocal of the smallest positive root of this polynomial is $16$.
\end{proof}

Since $\beta$ passes through filaments and is not a filament curve, it contains at least one curve $\beta_1$ that passes through filaments and is not a filament curve. Similarly, $\gamma$ contains at least one curve $\gamma_1$ that passes through filaments and is not a filament curve.

We let $\beta'_1$ be a double of $\beta_1$ and $\gamma'_1$ be a double of $\gamma_1$.
If $\beta_1, \beta'_1, \gamma_1, \gamma'_1$ are distinct, then, depending on which one of $\beta_1$, $\beta'_1$ and which one of $\gamma_1$, $\gamma'_1$ is orientation-preserving, we can replace $\beta$ by $\beta_1$ or $\beta'_1$ and $\gamma$ by $\gamma_1$ or $\gamma'_1$ to arrange it so that $\beta$ and $\gamma$ are embedded.

Otherwise, suppose $\beta_1 = \gamma_1$. Then since $\beta \neq \gamma$, we have $\beta \neq \beta_1$ or $\gamma \neq \gamma_1$. 
Without loss of generality suppose $\beta \neq \beta_1$. Then $\beta$ contains another curve $\beta_2$.

If $\beta$ contains a third curve $\beta_3$, then \Cref{lemma:leq2curves} implies that $\lambda^n \geq 16$. 

If $\beta$ does not contain a third curve, but $\beta_2 = \beta'_1$, then $\beta_1$ and $\beta_2$ would have opposite parity, and $\beta$ must traverse $\beta_1$ or $\beta_2$ at least twice, for otherwise it would be orientation-reversing.
Thus $4\len(\beta_1) = 3\len(\beta_1) + \len(\gamma_1) \leq \len(\beta) + \len(\gamma) \leq n$.
The curve obtained by composing $\beta_1$ and $\beta'_1$ has length $2\len(\beta_1) \leq \frac{n}{2}$, so by \Cref{lemma:halfnemb} we have $\lambda^n \geq 16$.

If $\beta$ does not contain a third curve and $\beta_2 \neq \beta'_1$, then the curve obtained by composing $\beta_1$, $\beta_2$, $\beta'_1$ has length $(\len(\beta_1) + \len(\beta_2))+\len(\gamma_1) \leq \len(\beta) + \len(\gamma) \leq n$, so by \Cref{lemma:leq2curves} we have $\lambda^n \geq 16$.

Thus we can assume from this point onwards that $\beta$ and $\gamma$ are embedded. 
We denote the lengths of $\beta$ and $\gamma$ by $p$ and $q$ respectively. 

Next, we argue that we can assume $\beta$ and $\gamma$ each enter filament curves only once.

\begin{lemma} \label{lemma:caseIenterfilamenttwice}
If $\beta$ enters filament curves at least twice, then $\lambda^n \geq 14.6$.
\end{lemma}
\begin{proof}
If $\beta$ enters filament curves at least twice, then by following $\beta$ but taking the other edge in each of the two pairs of edges, we get four curves $\beta=\beta_1, \beta_2, \beta_3, \beta_4$ consisting of two pairs of doubles. 

If $\gamma$ is not one of the $\beta_i$ then we can pick a double $\gamma'$ of $\gamma$ that is also not one of the $\beta_i$. 
Thus, depending on whether $\beta$ and $\gamma$ intersect, the curve complex must contain an induced subgraph of one of the following two forms:
\begin{center}
\scalebox{\graphscale}{
\begin{tikzpicture}
\draw[draw=black, fill=black, thin, solid] (0,0) circle (0.1);
\node[black, anchor=east] at (-0.2,0) {$p$};
\node[black, anchor=north] at (0,-0.2) {$\times 4$};
\draw[draw=black, fill=black, thin, solid] (0.4,1) circle (0.1);
\node[black, anchor=south] at (0.4,1.2) {$m_1$};
\node[black, anchor=center] at (1,1) {$\cdots$};
\draw[draw=black, fill=black, thin, solid] (1.6,1) circle (0.1);
\node[black, anchor=south] at (1.6,1.2) {$m_k$};
\draw[draw=black, fill=black, thin, solid] (2,0) circle (0.1);
\node[black, anchor=west] at (2.2,0) {$q$};
\node[black, anchor=north] at (2,-0.2) {$\times 2$};
\draw[draw=black, thin, solid] (0.4,1) -- (0.7,1);
\draw[draw=black, thin, solid] (1.3,1) -- (1.6,1);
\draw[draw=black, thin, solid] (0,0) -- (0.9,0.9);
\draw[draw=black, thin, solid] (0,0) -- (1.6,1);
\draw[draw=black, thin, solid] (2,0) -- (0.4,1);
\draw[draw=black, thin, solid] (2,0) -- (1.6,1);
\end{tikzpicture}
\hspace{3cm}
\begin{tikzpicture}
\draw[draw=black, fill=black, thin, solid] (0,0) circle (0.1);
\node[black, anchor=east] at (-0.2,0) {$p$};
\node[black, anchor=north] at (0,-0.2) {$\times 4$};
\draw[draw=black, fill=black, thin, solid] (0.4,1) circle (0.1);
\node[black, anchor=south] at (0.4,1.2) {$m_1$};
\node[black, anchor=center] at (1,1) {$\cdots$};
\draw[draw=black, fill=black, thin, solid] (1.6,1) circle (0.1);
\node[black, anchor=south] at (1.6,1.2) {$m_k$};
\draw[draw=black, fill=black, thin, solid] (2,0) circle (0.1);
\node[black, anchor=west] at (2.2,0) {$q$};
\node[black, anchor=north] at (2,-0.2) {$\times 2$};
\draw[draw=black, thin, solid] (0,0) -- (2,0);
\draw[draw=black, thin, solid] (0.4,1) -- (0.7,1);
\draw[draw=black, thin, solid] (1.3,1) -- (1.6,1);
\draw[draw=black, thin, solid] (0,0) -- (0.9,0.9);
\draw[draw=black, thin, solid] (0,0) -- (1.6,1);
\draw[draw=black, thin, solid] (2,0) -- (0.4,1);
\draw[draw=black, thin, solid] (2,0) -- (1.6,1);
\end{tikzpicture}}
\end{center}
where
\begin{itemize}
    \item $\times N$ denotes $N$ vertices of the denoted weight, 
    \item an edge between two vertices labelled $\times N_1$ and $\times N_2$ denotes an edge going between each of the $N_1$ vertices and each of the $N_2$ vertices,
    \item the vertices in the top row are the filament curves, and
    \item the edges between the vertices in top row and those in the bottom row depend on which ones of the filament curves $\beta$ and $\gamma$ intersect.
\end{itemize}

The graph on the left is a wide subgraph of the one on the right, so the growth rate of $G$ is always bounded from below by the growth rate of the graph on the right. Furthermore, the growth rate of the graph on the right is bounded from below by that of the following graph $G_1$: 

\begin{center}
\scalebox{\graphscale}{
\begin{tikzpicture}
\draw[draw=black, fill=black, thin, solid] (0,0) circle (0.1);
\node[black, anchor=east] at (-0.2,0) {$p$};
\node[black, anchor=north] at (0,-0.2) {$\times 4$};
\draw[draw=black, fill=black, thin, solid] (1,1) circle (0.1);
\node[black, anchor=south] at (1,1.2) {$m$};
\draw[draw=black, fill=black, thin, solid] (2,0) circle (0.1);
\node[black, anchor=west] at (2.2,0) {$q$};
\node[black, anchor=north] at (2,-0.2) {$\times 2$};
\draw[draw=black, thin, solid] (0,0) -- (2,0);
\end{tikzpicture}}
\end{center}
obtained by summing together the vertices in the top row, by \Cref{prop:vertexsumgrowthrate}.

By monotonicity, the minimum growth rate of $G_1$ subject to the conditions $m \leq n$ and $p+q \leq n$ is attained when $m=n$ and $q=n-p$. In this case, the clique polynomial of $G_1$ is
\begin{align*}
Q_1(t) &= 1-4t^{p}-2t^{n-p}+7t^{n} \\
&= 1-4x^a-2x^{1-a}+7x
\end{align*}
where we changed variables $x=t^n$ and $a = \frac{p}{n}$.

We compute as in \Cref{lemma:leq2curves} that the minimum value of the reciprocal of the smallest positive root of this polynomial is $9+4\sqrt{2} \geq 14.6$. 
Thus $\lambda^n \geq \lambda(G_1,w_1)^n \geq 14.6$.

If instead $\gamma$ is one of the $\beta_i$, then $p=q$. Reasoning as above, the growth rate of the curve complex $G$ is bounded from below by that of the following graph $G_1$:

\begin{center}
\scalebox{\graphscale}{
\begin{tikzpicture}
\draw[draw=black, fill=black, thin, solid] (0,0) circle (0.1);
\node[black, anchor=east] at (-0.2,0) {$p$};
\node[black, anchor=north] at (0,-0.2) {$\times 4$};
\draw[draw=black, fill=black, thin, solid] (0,1) circle (0.1);
\node[black, anchor=south] at (0,1.2) {$m$};
\end{tikzpicture}}
\end{center}

By monotonicity, the minimum growth rate of $G_1$ subject to the conditions $m \leq n$ and $2p \leq n$ is attained when $m=n$ and $p = \frac{n}{2}$. In this case, the clique polynomial of $G_1$ is
\begin{align*}
Q_1(t) &= 1-4t^{\frac{n}{2}}-t^{n} \\
&= 1-4x^\frac{1}{2}-x
\end{align*}
where we changed variables $x=t^n$.

The reciprocal of the smallest positive root of this polynomial is $9+4\sqrt{5} \geq 17.9$. 
Thus $\lambda^n \geq \lambda(G_1,w_1)^n \geq 14.6$.
\end{proof}

Thus we can assume from this point onwards that $\beta$ and $\gamma$ each enter filament curves only once. In particular, each of them only intersect one filament curve. We define $\beta'$ and $\gamma'$ to be the (unique) doubles of $\beta$ and $\gamma$ respectively.

\begin{lemma}
If $\beta$ and $\gamma$ intersect, then $\lambda^n \geq 17.9$
\end{lemma}
\begin{proof}
If $\beta$ and $\gamma$ intersect, then so do the doubles $\beta'$ and $\gamma'$, so the curve complex $G$ must contain an induced subgraph of the following form:
\begin{center}
\scalebox{\graphscale}{
\begin{tikzpicture}
\draw[draw=black, fill=black, thin, solid] (0,0) circle (0.1);
\node[black, anchor=east] at (-0.2,0) {$p$};
\node[black, anchor=north] at (0,-0.2) {$\times 2$};
\draw[draw=black, fill=black, thin, solid] (0.4,1) circle (0.1);
\node[black, anchor=south] at (0.4,1.2) {$m_1$};
\node[black, anchor=center] at (1,1) {$\cdots$};
\draw[draw=black, fill=black, thin, solid] (1.6,1) circle (0.1);
\node[black, anchor=south] at (1.6,1.2) {$m_k$};
\draw[draw=black, fill=black, thin, solid] (2,0) circle (0.1);
\node[black, anchor=west] at (2.2,0) {$q$};
\node[black, anchor=north] at (2,-0.2) {$\times 2$};
\draw[draw=black, thin, solid] (0.4,1) -- (0.7,1);
\draw[draw=black, thin, solid] (1.3,1) -- (1.6,1);
\draw[draw=black, thin, solid] (0,0) -- (0.9,0.9);
\draw[draw=black, thin, solid] (0,0) -- (1.6,1);
\draw[draw=black, thin, solid] (2,0) -- (0.4,1);
\draw[draw=black, thin, solid] (2,0) -- (1.6,1);
\end{tikzpicture}}
\end{center}
where vertices in the top row are the filament curves as in \Cref{lemma:caseIenterfilamenttwice}. By summing these vertices, the growth rate of $G$ is bounded from below by that of the following graph $G_1$: 

\begin{center}
\scalebox{\graphscale}{
\begin{tikzpicture}
\draw[draw=black, fill=black, thin, solid] (0,0) circle (0.1);
\node[black, anchor=east] at (-0.2,0) {$p$};
\node[black, anchor=north] at (0,-0.2) {$\times 2$};
\draw[draw=black, fill=black, thin, solid] (1,1) circle (0.1);
\node[black, anchor=south] at (1,1.2) {$m$};
\draw[draw=black, fill=black, thin, solid] (2,0) circle (0.1);
\node[black, anchor=west] at (2.2,0) {$q$};
\node[black, anchor=north] at (2,-0.2) {$\times 2$};
\end{tikzpicture}}
\end{center}

By convexity and monotonicity, the minimum growth rate of $G_1$ subject to the conditions $m \leq n$ and $p+q \leq n$ is attained when $m=n$ and $p=q=\frac{n}{2}$. In this case, the clique polynomial of $G_1$ is
\begin{align*}
Q_1(t) &= 1-4t^{\frac{n}{2}}-t^{n} \\
&= 1-4x^{\frac{1}{2}}-x
\end{align*}
where we changed variables $x=t^n$.

The reciprocal of the smallest positive root of this polynomial is $9+4\sqrt{5} \geq 17.9$. 
Thus $\lambda^n \geq \lambda(G_1,w_1)^n \geq 17.9$.
\end{proof}

Thus we can assume from this point onwards that $\beta$ and $\gamma$ are disjoint.
Next, we rule out the existence of petals.

\begin{lemma}
If $\beta$ passes through petals, then $\lambda^n \geq 25$.
\end{lemma}
\begin{proof}
We fix the canonical system of petal curves (recall \Cref{defn:petalcurves}). The growth rate of the curve complex $G$ is bounded from below by that of the following graph $G_1$: 

\begin{center}
\scalebox{\graphscale}{
\begin{tikzpicture}
\draw[draw=black, fill=black, thin, solid] (0,0) circle (0.1);
\node[black, anchor=east] at (-0.2,0) {$p$};
\node[black, anchor=north] at (0,-0.2) {$\times 2$};
\draw[draw=black, fill=black, thin, solid] (1,1) circle (0.1);
\node[black, anchor=south] at (1,1.2) {$m$};
\draw[draw=black, fill=black, thin, solid] (2,0) circle (0.1);
\node[black, anchor=west] at (2.2,0) {$q$};
\node[black, anchor=north] at (2,-0.2) {$\times 2$};
\draw[draw=black, thin, solid] (0,0) -- (2,0);
\draw[draw=black, fill=black, thin, solid] (1,-1) circle (0.1);
\node[black, anchor=north] at (1,-1.2) {$r_0$};
\draw[draw=black, thin, solid] (1,1) -- (1,-1);
\draw[draw=black, thin, solid] (2,0) -- (1,-1);
\end{tikzpicture}}
\end{center}
where the top vertex is the sum of the filament curves and the bottom vertex is the sum of the petal curves.

By monotonicity, the minimum growth rate of $G_1$ subject to the conditions $m \leq n$ and $p+q+r_0 \leq n$ is attained when $m=n$ and $r_0=n-p-q$. In this case, the clique polynomial of $G_1$ is
\begin{align*}
Q_1(t) &= (1-2t^p)(1-2t^q)-t^n-t^{n-p-q}(1-2t^q-t^n) \\
&= (1-2x^a)(1-2x^b)-x-x^{1-a-b}(1-2x^b-x)
\end{align*}
where we changed variables $x=t^n$, $a=\frac{p}{n}$, and $b=\frac{q}{n}$.

We compute as in \Cref{lemma:leq2curves} (but now dealing with a system with 3 equations since we have to set the partial derivative with respect to $b$ to be $0$ as well) that the minimum value of the reciprocal of the smallest positive root of this polynomial is $25$. 
Thus $\lambda^n \geq \lambda(G_1,w_1)^n \geq 25$.
\end{proof}

Similarly, if $\gamma$ passes through petals, then $\lambda^n \geq 25$ as well.

\begin{lemma} \label{lemma:caseImufilamentspetals}
If $\tau$ has any petals, but $\beta$ and $\gamma$ do not pass through petals, then $\lambda^n \geq 19.9$.
\end{lemma}
\begin{proof}
We fix the canonical system of petal curves. Since the directed graph $\Gamma$ is strongly connected, there must exist a curve $\mu$ intersecting some filament curve and some petal curve, for otherwise the union of filament curves and the union of petal curves lie in distinct strongly connected components.
In particular, since $\mu$ enters a filament curve, it has a double $\mu'$.

If $\mu$ meets $\beta$ and $\gamma$, then the growth rate of the curve complex $G$ is bounded from below by that of the following graph $G_1$: 

\begin{center}
\scalebox{\graphscale}{
\begin{tikzpicture}
\draw[draw=black, fill=black, thin, solid] (0,0) circle (0.1);
\node[black, anchor=east] at (-0.2,0) {$p$};
\node[black, anchor=north] at (0,-0.2) {$\times 2$};
\draw[draw=black, fill=black, thin, solid] (1,1) circle (0.1);
\node[black, anchor=south] at (1,1.2) {$m$};
\draw[draw=black, fill=black, thin, solid] (2,0) circle (0.1);
\node[black, anchor=west] at (2.2,0) {$q$};
\node[black, anchor=north] at (2,-0.2) {$\times 2$};
\draw[draw=black, thin, solid] (0,0) -- (2,0);
\draw[draw=black, fill=black, thin, solid] (0.4,-1) circle (0.1);
\node[black, anchor=east] at (0.2,-1) {$r_0$};
\draw[draw=black, fill=black, thin, solid] (1.6,-1) circle (0.1);
\node[black, anchor=west] at (1.8,-1) {$u$};
\node[black, anchor=north] at (1.6,-1.2) {$\times 2$};
\draw[draw=black, thin, solid] (1,1) -- (0.4,-1);
\draw[draw=black, thin, solid] (0,0) -- (0.4,-1);
\draw[draw=black, thin, solid] (2,0) -- (0.4,-1);
\end{tikzpicture}}
\end{center}
where $u=\len(\mu) \leq m+r_0 \leq n+r_0$.

By convexity and monotonicity, the minimum growth rate of $G_1$ subject to the conditions $m \leq n$, $p+q+r_0 \leq n$, and $u \leq n+r_0$ is attained when $m=n$, $p=q$, $r_0=n-2p$, and $u=n+r_0=2n-2p$. In this case, the clique polynomial of $G_1$ is
\begin{align*}
Q_1(t) &= ((1-2t^p)^2-t^n)(1-t^{n-2p})-2t^{2n-2p} \\
&= ((1-2x^a)^2-x)(1-x^{1-2a})-2x^{2-2a}
\end{align*}
where we changed variables $x=t^n$ and $a=\frac{p}{n}$.

The minimum value of the reciprocal of the smallest positive root of this polynomial is too difficult to compute by hand. As such, we use computational methods to approximate it instead. 
This is done under the list \texttt{Case\_I} in the auxiliary file \texttt{pAbraid\_rootcheck.ipynb}, where we are able to verify that this minimum value is $\geq 20.7$.
See \Cref{subsec:code} for an explanation of the code.

If $\mu$ meets $\beta$ but not $\gamma$, then the growth rate of the curve complex $G$ is bounded from below by that of the following graph $G_1$: 

\begin{center}
\scalebox{\graphscale}{
\begin{tikzpicture}
\draw[draw=black, fill=black, thin, solid] (0,0) circle (0.1);
\node[black, anchor=east] at (-0.2,0) {$p$};
\node[black, anchor=north] at (0,-0.2) {$\times 2$};
\draw[draw=black, fill=black, thin, solid] (1,1) circle (0.1);
\node[black, anchor=south] at (1,1.2) {$m$};
\draw[draw=black, fill=black, thin, solid] (2,0) circle (0.1);
\node[black, anchor=west] at (2.2,0) {$q$};
\node[black, anchor=north] at (2,-0.2) {$\times 2$};
\draw[draw=black, thin, solid] (0,0) -- (2,0);
\draw[draw=black, fill=black, thin, solid] (0.4,-1) circle (0.1);
\node[black, anchor=east] at (0.2,-1) {$r_0$};
\draw[draw=black, fill=black, thin, solid] (1.6,-1) circle (0.1);
\node[black, anchor=west] at (1.8,-1) {$u$};
\node[black, anchor=north] at (1.6,-1.2) {$\times 2$};
\draw[draw=black, thin, solid] (1,1) -- (0.4,-1);
\draw[draw=black, thin, solid] (0,0) -- (0.4,-1);
\draw[draw=black, thin, solid] (2,0) -- (0.4,-1);
\draw[draw=black, thin, solid] (2,0) -- (1.6,-1);
\end{tikzpicture}}
\end{center}
where $u=\len(\mu)$ satisfies $q+u \leq m+r_0 \leq n+r_0$.

By monotonicity, the minimum growth rate of $G_1$ subject to the conditions $m \leq n$, $p+q+r_0 \leq n$, and $q+u \leq n+r_0$ is attained when $m=n$, $r_0=n-p-q$, and $u=n+r_0-q=2n-p-2q$. In this case, the clique polynomial of $G_1$ is
\begin{align*}
Q_1(t) &= ((1-2t^p)(1-2t^q)-t^n)(1-t^{n-p-q})-2t^{2n-p-2q}(1-2t^q) \\
&= ((1-2x^a)(1-2x^b)-x)(1-x^{1-a-b})-2x^{2-a-2b}(1-2x^b)
\end{align*}
where we changed variables $x=t^n$, $a=\frac{p}{n}$ and $b=\frac{q}{n}$.

Using our code as above, we compute that the minimum value of the reciprocal of the smallest positive root of this polynomial is $\geq 19.9$.

If $\mu$ does not meet $\beta$ and $\gamma$, then the growth rate of the curve complex $G$ is bounded from below by that of the following graph $G_1$: 

\begin{center}
\scalebox{\graphscale}{
\begin{tikzpicture}
\draw[draw=black, fill=black, thin, solid] (0,0) circle (0.1);
\node[black, anchor=east] at (-0.2,0) {$p$};
\node[black, anchor=north] at (0,-0.2) {$\times 2$};
\draw[draw=black, fill=black, thin, solid] (1,1) circle (0.1);
\node[black, anchor=south] at (1,1.2) {$m$};
\draw[draw=black, fill=black, thin, solid] (2,0) circle (0.1);
\node[black, anchor=west] at (2.2,0) {$q$};
\node[black, anchor=north] at (2,-0.2) {$\times 2$};
\draw[draw=black, thin, solid] (0,0) -- (2,0);
\draw[draw=black, fill=black, thin, solid] (0.4,-1) circle (0.1);
\node[black, anchor=east] at (0.2,-1) {$r_0$};
\draw[draw=black, fill=black, thin, solid] (1.6,-1) circle (0.1);
\node[black, anchor=west] at (1.8,-1) {$u$};
\node[black, anchor=north] at (1.6,-1.2) {$\times 2$};
\draw[draw=black, thin, solid] (1,1) -- (0.4,-1);
\draw[draw=black, thin, solid] (0,0) -- (0.4,-1);
\draw[draw=black, thin, solid] (2,0) -- (0.4,-1);
\draw[draw=black, thin, solid] (0,0) -- (1.6,-1);
\draw[draw=black, thin, solid] (2,0) -- (1.6,-1);
\end{tikzpicture}}
\end{center}
where $u=\len(\mu)$ satisfies $p+q+u \leq m+r_0 \leq n+r_0$.

By convexity and monotonicity, the minimum growth rate of $G_1$ subject to the conditions $m \leq n$, $p+q+r_0 \leq n$, and $p+q+u \leq n+r_0$ is attained when $m=n$, $p=q$, $r_0=n-2p$, and $u=n+r_0-p-q=2n-4p$. In this case, the clique polynomial of $G_1$ is
\begin{align*}
Q_1(t) &= ((1-2t^p)^2-t^n)(1-t^{n-2p})-2t^{2n-4p}(1-2t^p)^2 \\
&= ((1-2x^a)^2-x)(1-x^{1-2a})-2x^{2-4a}(1-2x^a)^2
\end{align*}
where we changed variables $x=t^n$ and $a=\frac{p}{n}$.

Using our code as above, we compute that the minimum value of the reciprocal of the smallest positive root of this polynomial is $\geq 21.2$.
\end{proof}

Thus we can assume from this point onwards that there are no petals.
Next, we rule out the existence of more than one filament curve. Recall that we have reduced to the case when $\beta$ and $\gamma$ each intersect a unique filament curve.

\begin{lemma}
If $\beta$ and $\gamma$ intersect distinct filament curves, then $\lambda^n \geq 16$.
\end{lemma}
\begin{proof}
Let $\alpha_1$ be the filament curve that $\beta$ intersects. Since $\Gamma$ is strongly connected, and since we have reduced to the case when there are no petals, there is a curve $\mu$ intersecting $\alpha_1$ and some other filament curve. Since $\mu$ enters filament curves at least twice, we can find four curves $\mu=\mu_1,\mu_2,\mu_3,\mu_4$ consisting of two pairs of doubles by following $\mu$.

Thus the growth rate of the curve complex $G$ is bounded from below by that of the following graph $G_1$: 

\begin{center}
\scalebox{\graphscale}{
\begin{tikzpicture}
\draw[draw=black, fill=black, thin, solid] (0,0) circle (0.1);
\node[black, anchor=east] at (-0.2,0) {$p$};
\node[black, anchor=north] at (0,-0.2) {$\times 2$};
\draw[draw=black, fill=black, thin, solid] (0.4,1) circle (0.1);
\node[black, anchor=south] at (0.4,1.2) {$m_1$};
\draw[draw=black, fill=black, thin, solid] (1.6,1) circle (0.1);
\node[black, anchor=south] at (1.6,1.2) {$m_2$};
\draw[draw=black, fill=black, thin, solid] (2,0) circle (0.1);
\node[black, anchor=west] at (2.2,0) {$q$};
\node[black, anchor=north] at (2,-0.2) {$\times 2$};
\draw[draw=black, thin, solid] (0,0) -- (2,0);
\draw[draw=black, thin, solid] (0,0) -- (1.6,1);
\draw[draw=black, thin, solid] (0.4,1) -- (1.6,1);
\draw[draw=black, thin, solid] (0.4,1) -- (2,0);
\draw[draw=black, fill=black, thin, solid] (1,-1) circle (0.1);
\node[black, anchor=west] at (1.2,-1) {$u$};
\node[black, anchor=north] at (1,-1.2) {$\times 4$};
\draw[draw=black, thin, solid] (1,-1) -- (0,0);
\draw[draw=black, thin, solid] (1,-1) -- (2,0);
\end{tikzpicture}}
\end{center}
where the top left vertex is $\alpha_1$ while the top right vertex is the sum of the rest of the filament curves, so $m_1 + m_2 = m \leq n$.

By convexity and monotonicity, the minimum growth rate of $G_1$ subject to the conditions $m_1+m_2 \leq n$, $p+q \leq n$, and $u \leq n$ is attained when $m_1=m_2=p=q=\frac{n}{2}$ and $u=n$. In this case, the clique polynomial of $G_1$ is
\begin{align*}
Q_1(t) &= (1-3t^{\frac{n}{2}})^2-4t^n(1-2t^{\frac{n}{2}})^2 \\
&= (1-3x^{\frac{1}{2}})^2-4x(1-2x^{\frac{1}{2}})^2
\end{align*}
where we changed variables $x=t^n$.

The reciprocal of the smallest positive root of this polynomial is $16$.
\end{proof}

\begin{lemma}
If $\beta$ and $\gamma$ intersect the same filament curve but there are other filament curves, then $\lambda^n \geq 68.2$.
\end{lemma}
\begin{proof}
Let $\alpha_1$ be the filament curve that $\beta$ and $\gamma$ intersect. Since $\Gamma$ is strongly connected, and since we have reduced to the case when there are no petals, there is a curve $\mu$ intersecting $\alpha_1$ and some other filament curve. Since $\mu$ enters filament curves at least twice, we can find four curves $\mu=\mu_1,\mu_2,\mu_3,\mu_4$ consisting of two pairs of doubles by following $\mu$.

Thus the growth rate of the curve complex $G$ is bounded from below by that of the following graph $G_1$: 

\begin{center}
\scalebox{\graphscale}{
\begin{tikzpicture}
\draw[draw=black, fill=black, thin, solid] (0,0) circle (0.1);
\node[black, anchor=east] at (-0.2,0) {$p$};
\node[black, anchor=north] at (0,-0.2) {$\times 2$};
\draw[draw=black, fill=black, thin, solid] (0.4,1) circle (0.1);
\node[black, anchor=south] at (0.4,1.2) {$m_1$};
\draw[draw=black, fill=black, thin, solid] (1.6,1) circle (0.1);
\node[black, anchor=south] at (1.6,1.2) {$m_2$};
\draw[draw=black, fill=black, thin, solid] (2,0) circle (0.1);
\node[black, anchor=west] at (2.2,0) {$q$};
\node[black, anchor=north] at (2,-0.2) {$\times 2$};
\draw[draw=black, thin, solid] (0,0) -- (2,0);
\draw[draw=black, thin, solid] (0,0) -- (1.6,1);
\draw[draw=black, thin, solid] (0.4,1) -- (1.6,1);
\draw[draw=black, thin, solid] (1.6,1) -- (2,0);
\draw[draw=black, fill=black, thin, solid] (1,-1) circle (0.1);
\node[black, anchor=west] at (1.2,-1) {$u$};
\node[black, anchor=north] at (1,-1.2) {$\times 4$};
\draw[draw=black, thin, solid] (1,-1) -- (0,0);
\draw[draw=black, thin, solid] (1,-1) -- (2,0);
\end{tikzpicture}}
\end{center}

By convexity and monotonicity, the minimum growth rate of $G_1$ subject to the conditions $m_1+m_2 \leq n$, $p+q \leq m_1$, and $u \leq n$ is attained when $m_2=n-m_1$, $p=q=\frac{m_1}{2}$, and $u=n$. In this case, the clique polynomial of $G_1$ is
\begin{align*}
Q_1(t) &= (1-2t^{\frac{m_1}{2}})^2(1-t^{m_1}-4t^n)-t^{n-m_1}(1-t^{m_1}) \\
&= (1-2x^a)^2(1-x^{2a}-4x)-x^{1-2a}(1-x^{2a})
\end{align*}
where we changed variables $x=t^n$ and $a =\frac{m_1}{2n}$.

Using our code, we compute that the minimum value of the reciprocal of the smallest positive root of this polynomial is $\geq 68.2$.
\end{proof}

Thus we can assume from this point onwards that there is only one filament curve.

To summarize, we have reduced to the situation where we have

\begin{itemize}
    \item a single filament curve $\alpha$,
    \item a curve $\beta$ following $\alpha$ except for one edge $e_1$, which has a double $e'_1$, and
    \item a curve $\gamma$ following $\alpha$ except for one edge $e_2$, which has a double $e'_2$,
\end{itemize}
where $\beta$ and $\gamma$ are disjoint, thus in particular $e_1$, $e'_1$, $e_2$, and $e'_2$ are distinct. 

\begin{lemma} \label{lemma:caseInonreciprocal}
Either $\lambda^n \geq 17.9$ or there is an edge $e_3$ other from $e_1$, $e'_1$, $e_2$, and $e'_2$.
\end{lemma}
\begin{proof}
Suppose there is no edge $e_3$ other from $e_1$, $e'_1$, $e_2$, and $e'_2$, then the curve complex $G$ is exactly

\begin{center}
\scalebox{\graphscale}{
\begin{tikzpicture}
\draw[draw=black, fill=black, thin, solid] (0,0) circle (0.1);
\node[black, anchor=east] at (-0.2,0) {$p$};
\node[black, anchor=north] at (0,-0.2) {$\times 2$};
\draw[draw=black, fill=black, thin, solid] (1,1) circle (0.1);
\node[black, anchor=south] at (1,1.2) {$m$};
\draw[draw=black, fill=black, thin, solid] (2,0) circle (0.1);
\node[black, anchor=west] at (2.2,0) {$q$};
\node[black, anchor=north] at (2,-0.2) {$\times 2$};
\draw[draw=black, thin, solid] (0,0) -- (2,0);
\end{tikzpicture}}
\end{center}

Without loss of generality suppose $p \leq q$. The clique polynomial of $G$ is
$$Q(t)=1-2t^p-2t^q+4t^{p+q}-t^m,$$
where we have ordered the terms in increasing order of their exponents, except possibly $p=q$ or $p+q=m$. 

By \Cref{thm:cliquepoly=charpoly} and \Cref{thm:basedsettexists}(4), we know that $Q(t)$ is reciprocal. This can only happen if $p=q=\frac{m}{3}$, in which case we have
\begin{align*}
Q(t) &= 1-4t^{\frac{m}{3}}+4t^{\frac{2m}{3}}-t^m \\
&= 1-4x^{\frac{1}{3}}+4x^{\frac{2}{3}}-x
\end{align*}
where we changed variables $x=t^m$.

The reciprocal of the smallest positive root of this polynomial is $9+4\sqrt{5} \geq 17.9$. 
Thus $\lambda^n \geq \lambda^m \geq \lambda(G,w)^m \geq 17.9$.
\end{proof}

Let $\mu$ be the curve obtained by traversing $e_3$ then following $\alpha$ until the path closes up, and let $\mu'$ be the curve defined similarly but using the double $e'_3$ of $e_3$.

\begin{lemma} \label{lemma:caseImudisjointbetagamma}
If $\mu$ does not intersect $\beta$ and $\gamma$, then $\lambda^n \geq 27$.
\end{lemma}
\begin{proof}
In this case the growth rate of the curve complex $G$ is bounded from below by that of the following graph $G_1$: 

\begin{center}
\scalebox{\graphscale}{
\begin{tikzpicture}
\draw[draw=black, fill=black, thin, solid] (0,0) circle (0.1);
\node[black, anchor=east] at (-0.2,0) {$p$};
\node[black, anchor=north] at (0,-0.2) {$\times 2$};
\draw[draw=black, fill=black, thin, solid] (1,1) circle (0.1);
\node[black, anchor=south] at (1,1.2) {$m$};
\draw[draw=black, fill=black, thin, solid] (2,0) circle (0.1);
\node[black, anchor=west] at (2.2,0) {$q$};
\node[black, anchor=north] at (2,-0.2) {$\times 2$};
\draw[draw=black, thin, solid] (0,0) -- (2,0);
\draw[draw=black, fill=black, thin, solid] (1,-1) circle (0.1);
\node[black, anchor=west] at (1.2,-1) {$u$};
\node[black, anchor=north] at (1,-1.2) {$\times 2$};
\draw[draw=black, thin, solid] (1,-1) -- (0,0);
\draw[draw=black, thin, solid] (1,-1) -- (2,0);
\end{tikzpicture}}
\end{center}

By convexity and monotonicity, the minimum growth rate of $G_1$ subject to the conditions $m \leq n$ and $p+q+u \leq n$ is attained when $m=n$ and $p=q=u=\frac{n}{3}$. In this case, the clique polynomial of $G_1$ is
\begin{align*}
Q_1(t) &= (1-2t^{\frac{n}{3}})^3-t^n \\
&= (1-2x^{\frac{1}{3}})^3-x
\end{align*}
where we changed variables $x=t^n$.

The reciprocal of the smallest positive root of this polynomial is $27$. Thus $\lambda^n \geq \lambda(G_1,w_1)^n \geq 27$.
\end{proof}

\begin{lemma} \label{lemma:caseImudisjointgamma}
If $\mu$ intersects $\beta$ but not $\gamma$, then $\lambda^n \geq 14.6$.
\end{lemma}
\begin{proof}
In this case the growth rate of the curve complex $G$ is bounded from below by that of the following graph $G_1$: 

\begin{center}
\scalebox{\graphscale}{
\begin{tikzpicture}
\draw[draw=black, fill=black, thin, solid] (0,0) circle (0.1);
\node[black, anchor=east] at (-0.2,0) {$p$};
\node[black, anchor=north] at (0,-0.2) {$\times 2$};
\draw[draw=black, fill=black, thin, solid] (1,1) circle (0.1);
\node[black, anchor=south] at (1,1.2) {$m$};
\draw[draw=black, fill=black, thin, solid] (2,0) circle (0.1);
\node[black, anchor=west] at (2.2,0) {$q$};
\node[black, anchor=north] at (2,-0.2) {$\times 2$};
\draw[draw=black, thin, solid] (0,0) -- (2,0);
\draw[draw=black, fill=black, thin, solid] (1,-1) circle (0.1);
\node[black, anchor=west] at (1.2,-1) {$u$};
\node[black, anchor=north] at (1,-1.2) {$\times 2$};
\draw[draw=black, thin, solid] (1,-1) -- (2,0);
\end{tikzpicture}}
\end{center}

By monotonicity, the minimum growth rate of $G_1$ subject to the conditions $m \leq n$, $p+q \leq n$, and $q+u \leq n$ is attained when $m=n$ and $p=u=n-q$. In this case, the clique polynomial of $G_1$ is
\begin{align*}
Q_1(t) &= (1-2t^q)(1-4t^{n-q})-t^n \\
&= (1-2x^a)(1-4x^{1-a})-x
\end{align*}
where we changed variables $x=t^n$ and $a=\frac{q}{n}$.

We compute as in \Cref{lemma:leq2curves} that the minimum value of the reciprocal of the smallest positive root of this polynomial is $9+4\sqrt{2} \geq 14.6$, attained near $a \approx 0.37$. Thus $\lambda^n \geq \lambda(G_1,w_1)^n \geq 14.6$.
\end{proof}

Thus we can assume $\mu$ intersects both $\beta$ and $\gamma$.
In other words, $e_3$ must go from a vertex of $\beta$ to a vertex of $\gamma$, or vice versa. Let us call an edge satisfying this property a \textbf{bridge}. We can thus assume that every edge other than $e_1$, $e'_1$, $e_2$, and $e'_2$ is a bridge.

\begin{lemma}
If there are at least two bridges that are not doubles of each other, then $\lambda^n \geq 17.8$.
\end{lemma}
\begin{proof}
Let $e_3$ and $e_4$ be two bridges that are not doubles of each other. We can construct curves $\mu$ and $\nu$ passing through $e_3$ and $e_4$ as above.

Thus the growth rate of the curve complex $G$ is bounded from below by that of the following graph $G_1$: 

\begin{center}
\scalebox{\graphscale}{
\begin{tikzpicture}
\draw[draw=black, fill=black, thin, solid] (0,0) circle (0.1);
\node[black, anchor=east] at (-0.2,0) {$p$};
\node[black, anchor=north] at (0,-0.2) {$\times 2$};
\draw[draw=black, fill=black, thin, solid] (1,1) circle (0.1);
\node[black, anchor=south] at (1,1.2) {$m$};
\draw[draw=black, fill=black, thin, solid] (2,0) circle (0.1);
\node[black, anchor=west] at (2.2,0) {$q$};
\node[black, anchor=north] at (2,-0.2) {$\times 2$};
\draw[draw=black, thin, solid] (0,0) -- (2,0);
\draw[draw=black, fill=black, thin, solid] (0.4,-1) circle (0.1);
\node[black, anchor=east] at (0.2,-1) {$u$};
\node[black, anchor=north] at (0.4,-1.2) {$\times 2$};
\draw[draw=black, fill=black, thin, solid] (1.6,-1) circle (0.1);
\node[black, anchor=west] at (1.8,-1) {$v$};
\node[black, anchor=north] at (1.6,-1.2) {$\times 2$};
\draw[draw=black, thin, solid] (0.4,-1) -- (1.6,-1);
\end{tikzpicture}}
\end{center}

By convexity and monotonicity, the minimum growth rate of $G_1$ subject to the conditions $m \leq n$, $p+q \leq n$, $u \leq n$, and $v \leq n$ is attained when $m=n$ and $p=q=\frac{n}{2}$ and $u=v=n$. In this case, the clique polynomial of $G_1$ is
\begin{align*}
Q_1(t) &= (1-2t^{\frac{n}{2}})^2-5t^n+t^{2n} \\
&= (1-2x^{\frac{1}{2}})^2-5x+x^2
\end{align*}
where we changed variables $x=t^n$.

The reciprocal of the smallest positive root of this polynomial in $x$ is $\geq 17.8$. Thus $\lambda^n \geq \lambda(G_1,w_1)^n \geq 17.8$.
\end{proof}

Thus we can assume that $e_3$ is the only bridge. We obtain our lower bound in this specific scenario.

\begin{lemma} \label{lemma:caseIdeltan}
If there are only two bridges (which are necessarily doubles of each other), then $\lambda^n \geq \min \{14.5, \underline{\delta}_n^n \}$.
\end{lemma}
\begin{proof}
In this case the growth rate of the curve complex $G$ is bounded from below by that of the following graph $G_1$: 

\begin{center}
\scalebox{\graphscale}{
\begin{tikzpicture}
\draw[draw=black, fill=black, thin, solid] (0,0) circle (0.1);
\node[black, anchor=east] at (-0.2,0) {$p$};
\node[black, anchor=north] at (0,-0.2) {$\times 2$};
\draw[draw=black, fill=black, thin, solid] (1,1) circle (0.1);
\node[black, anchor=south] at (1,1.2) {$m$};
\draw[draw=black, fill=black, thin, solid] (2,0) circle (0.1);
\node[black, anchor=west] at (2.2,0) {$q$};
\node[black, anchor=north] at (2,-0.2) {$\times 2$};
\draw[draw=black, thin, solid] (0,0) -- (2,0);
\draw[draw=black, fill=black, thin, solid] (1,-1) circle (0.1);
\node[black, anchor=west] at (1.2,-1) {$u$};
\node[black, anchor=north] at (1,-1.2) {$\times 2$};
\end{tikzpicture}}
\end{center}

By convexity and monotonicity, the minimum growth rate of $G_1$ subject to the conditions
\begin{itemize}
    \item $p+q \leq m$,
    \item $u \leq m$,
    \item $p,q,u$ are integers, and
    \item the characteristic polynomial $P(t)=Q(t^{-1})$ of $g^\real_*$ has a unique root of maximum modulus (which follows from the Perron-Frobenius theorem)
\end{itemize} 
is attained when $u=m$ and
$$(p,q) = \begin{cases}
(k,k+1) & \text{if $m=2k+1$} \\
(2k-1,2k+1) & \text{if $m=4k$} \\
(2k-1,2k+3) & \text{if $m=4k+2$}
\end{cases}$$

In this case, the reciprocal of the smallest positive root of the clique polynomial is $\underline{\delta}_m$.

For $n \in \underline{N}$, we have $\underline{\delta}_m \geq \underline{\delta}_n$ for $m \leq n$, thus $\lambda \geq \lambda(G_1,w_1) \geq \underline{\delta}_m \geq \underline{\delta}_n$.
For $n \not\in \underline{N}$, we have $\underline{\delta}_m^n \geq 14.5$ for $m \leq n$, thus $\lambda^n \geq \lambda(G_1,w_1)^n \geq \underline{\delta}_m^n \geq 14.5$.
\end{proof}

\subsection{Summary of cases II - VI} \label{subsec:casesII-VI}

We defer the computations in cases II - VI to the appendix. For now, we summarize the result of those computations in the following proposition.

\begin{prop} \label{prop:casesII-VI}
In each of cases II - VI, the normalized dilatation satisfies $\lambda^n \geq 14.5$.
\end{prop}

\Cref{prop:casesII-VI} will be implied by \Cref{prop:caseII}, \Cref{prop:caseIII}, \Cref{prop:caseIV}, \Cref{prop:caseV}, and \Cref{prop:caseVI}.

\subsection{Deducing \Cref{thm:braiddillowerbound} and \Cref{thm:braiddilequality}} \label{subsec:mainthmproof}

\begin{proof}[Proof of \Cref{thm:braiddillowerbound}]
As explained in \Cref{sec:prelimobs}, there are six cases. In case I, \Cref{prop:caseI} implies that $\lambda^n \geq \min\{14.5, \underline{\delta}_n^n\}$. In cases II - VI, \Cref{prop:casesII-VI} implies that $\lambda^n \geq 14.5$.
\end{proof}

\begin{proof}[Proof of \Cref{thm:braiddilequality}]
As explained in \Cref{sec:prelimobs}, there are six cases. If $n \in \underline{N}$, then $\underline{\delta}_n^n < 14.5$, thus if $\lambda(f) = \underline{\delta}_n$, then $\lambda(f) < 14.5$. By \Cref{prop:casesII-VI}, we can only be in case I. It now suffices to deduce the items in \Cref{thm:braiddilequality} from the items in \Cref{prop:caseI}.

By \Cref{prop:caseI}(1), there are no petals, thus \Cref{lemma:numberoffilamentspetals} implies that there are no secondary singular points. In other words, the only singular points are the 1-pronged singular points in $\mathbf{a}$, and $b$ and $c$. 

Next, by \Cref{prop:caseI}(2a) and \Cref{lemma:primarysingemb}, $b$ and $c$ each cannot determine more than one cycle. Thus the action of $f$ on the collection of half-leaves at $b$ and $c$ has a single orbit. 
Furthermore, the number of prongs of $b$ and $c$ equals the length of $\beta$ and $\gamma$ respectively. 
From the values of these from \Cref{prop:caseI}(2c) and \Cref{eq:indPH}, we deduce that $c$ is unpunctured. Thus the strands are exactly the elements of $\mathbf{a}$, which are all punctured 1-pronged singular points.

Finally, \Cref{lemma:filamentcurvesid} implies that, in general, the number of orbits of the action of $f$ on the strands equals the number of filament curves. By \Cref{prop:caseI}(1), this number is 1.
\end{proof}

\section{Discussion and further questions} \label{sec:questions}

\subsection{Classification and persistence of petal curves} \label{subsec:petalcurves}

Let $f$ be a pseudo-Anosov braid. Suppose $g: \tau \to \tau$ is a floral train track map that carries $f$.
In \Cref{subsec:jointlesstt}, we showed that by factorizing $g$ into based folds, we can define a collection of filament curves and (canonical) petal curves in the directed graph $\Gamma$ associated to $g^\real_*$.

In \Cref{lemma:filamentcurvesid}, we showed that the filament curves are the cycles determined by the 1-pronged singular points. On the other hand, it is unclear whether the petal curves admit such a natural description. In fact, it is not even clear if the (canonical) petal curves depend on the factorization of $g$ into based folds, or the choice of $g$ itself.

If it is possible to obtain different (canonical) petal curves by varying the factorization or choice of $g$, then it would seem that persistent petal curves should be rare. In particular, we ask the following question.

\begin{quest} \label{quest:primarypersistence}
Are the cycles determined by the primary singular points $b$ and $c$ persistent?
\end{quest}

If the answer is `no', then one can significantly reduce the amount of computations in showing \Cref{thm:braiddillowerbound}. For example, Case VI would be made completely redundant.

\subsection{Reciprocity} \label{subsec:reciprocity}

One feature of the standardly embedded train track machinery which played a key role in \cite{HT22} is the fact that the real transition matrix is reciprocal (\Cref{thm:settexists}(4)). This aspect of the machinery is highly under-utilized in this paper. Indeed, the only place in our analysis where we use reciprocity is in \Cref{lemma:caseInonreciprocal}.

In general, the minimal growth rate of a graph $G$ is much larger if one only varies over the \textbf{reciprocal} weights, i.e. the weights $w:V(G) \to \mathbb{R}_+$ that make the clique polynomial $Q_{(G,w)}(t)$ reciprocal. However, a big problem here is that certain properties of the growth rate fail to make sense if one restricts to such reciprocal weights. More specifically, with respect to the items in \Cref{prop:growthrateprop}:
\begin{itemize}
    \item[(3)] If $H$ is an induced subgraph of $G$ and $w$ is a reciprocal weight for $G$, $w|_H$ may not be a reciprocal weight for $H$.
    \item[(4)] If $H$ is a wide subgraph of $G$ and $w$ is a reciprocal weight for $G$, $w$ may not be a reciprocal weight for $H$.
    \item[(5)] If $w_1$ and $w_2$ are two reciprocal weights, $tw_1+(1-t)w_2$ may not be a reciprocal weight.
\end{itemize}
This makes the minimization problems in our computations extremely difficult. 

If one is able to formulate suitable generalizations of these monotonicity and convexity properties within the subset of reciprocal weights, then one should be able to obtain stronger lower bounds or bypass some of the computations,

\subsection{The remaining values of $\delta_n$} \label{subsec:remainminbraiddil}

Recall that $\delta_n$ is the minimum dilatation among pseudo-Anosov braids with $n$ strands. We now know the value of $\delta_n$ for all $n$ except $n=10, 12, 14, 18, 22, 26$. An obvious problem is to compute $\delta_n$ for these remaining values of $n$.

Venzke conjectured in \cite[Conjecture 5.4]{Ven08} that $\delta_n = \underline{\delta}_n$ for these remaining values of $n$. Note that this conjecture is false for $n=10$. This is because one can construct a pseudo-Anosov braid with 10 strands and dilatation $\underline{\delta}_9 < \underline{\delta}_{10}$ by taking the pseudo-Anosov braid $\sigma_{3,5}$ with 9 strands in \cite{HK06} (which is conjugate to the braid $\psi_9$ in \cite{Ven08}) and puncturing out a fixed point.
In view of this example, it is natural to conjecture the following instead.

\begin{conj}[{\cite[Conjecture 4.1(2)]{KT11}}]
The minimum dilatation among pseudo-Anosov braids with 10 strands equals $\underline{\delta}_9 = |x^9-2x^5-2x^4+1|$.
\end{conj}

Theoretically, one should be able to repeat the computations in \cite{LT11b} to determine the values of these $\delta_n$. However, the process might be rather tedious for these larger values of $n$.

\subsection{The braids attaining minimum dilatation} \label{subsec:mindilequality}

Suppose that one can determine the value of $\delta_n$ for all $n$. The next natural question would then be the following:

\begin{quest} \label{quest:mindilbraids}
Which pseudo-Anosov braids with $n$ strands have dilatation $\delta_n$? 
\end{quest}

We believe that for $n \in \underline{N}$, the techniques for answering \Cref{quest:mindilbraids} are already in place: 
\Cref{prop:caseI} implies that such a pseudo-Anosov braid is carried by a floral train track map, where the directed graph associated to the real transition matrix is of a very specific form.
In particular, the train track map cannot consist of more than 3 based folds. 

Theoretically, one should be able list out all floral train tracks with the correct amount of filaments and for each of them list out all maps with $\leq 3$ based folds. 
For each map on the compiled list, one can then compute its transition matrix, and extract the maps whose transition matrix determines the desired directed graph. 
This would then be a complete list of pseudo-Anosov braids with dilatation $\delta_n$, thereby answering \Cref{quest:mindilbraids} for these large values of $n$.

\subsection{Other minimum dilatation problems} \label{subsec:othermindilproblem}

Let $\lambda_{g,n}$ be the minimum dilatation among pseudo-Anosov maps on the genus $g$ surface with $n$ punctures. The values of $\lambda_{g,n}$ have previously been only known for finitely many values of $(g,n)$, namely:
\begin{itemize}
    \item $\lambda_{0,4}=\underline{\delta}_3$, $\lambda_{0,5}=\lambda_{0,6}=\underline{\delta}_5$, $\lambda_{0,7}=\underline{\delta}_7$, and $\lambda_{0,8}=\underline{\delta}_8$ by \cite{KLS02}, \cite{HS07}, and \cite{LT11b}, as remarked in the introduction,
    \item $\lambda_{1,0}=\lambda_{1,1}=\underline{\delta}_3$ from classical arguments, and
    \item $\lambda_{2,0}=\underline{\delta}_5$ by \cite{CH08}.
\end{itemize} 
In this paper, we determine $\lambda_{0,n}$ for all large enough values of $n$. However, there is still a whole 2-parameter family of values of $(g,n)$ that remain mysterious. 

Of these, the 1-parameter family where $n=0$ has received the most interest. In particular, we would be remiss not to mention the following conjecture:

\begin{conj}[{Hironaka, \cite[Question 1.12]{Hir10}}] \label{conj:goldenratio}
The minimum dilatations $\delta_{g,0}$ on the closed surfaces of genus $g$ grow as
$$\lim_{g \to \infty} \delta_{g,0}^g = \mu^2 \approx 2.618$$
where $\mu = \frac{1 + \sqrt{5}}{2} \approx 1.618$ is the golden ratio.
\end{conj}

A recurring phenomenon for $\lambda_{g,n}$ is the asymptotics 
\begin{equation} \label{eq:normdilconst}
\log \lambda_{g,n} \asymp \frac{1}{|\chi(S_{g,n})|}
\end{equation}
along certain 1-parameter families.
For example, this is observed on the lines
\begin{itemize}
    \item $g=1$ by \cite{Tsa09},
    \item $n=n_0$ for fixed $n_0 > 0$ by \cite{Yaz20}, and
    \item $n=m_0 g$ for fixed $m_0 > 0$ by \cite{Val12}.
\end{itemize}
Along each of these 1-parameter families, one can ask whether $\lambda_{g,n}^{|\chi(S_{g,n})|}$ converges to any number, and if yes, what the value of that number is.

We caution however, that \Cref{eq:normdilconst} does not hold along all 1-parameter families. For example, on the lines $g = g_0$ for fixed $g_0 \geq 2$, \cite{Tsa09} shows that
$$\log \lambda_{g,n} \asymp \frac{\log|\chi(S_{g,n})|}{|\chi(S_{g,n})|}.$$

\appendix

\numberwithin{thm}{subsection}
\numberwithin{thm}{subsection}

\section{Computational details} \label{sec:computation}

In the arXiv version of this paper, we included an auxiliary file \texttt{pAbraid\_rootcheck.ipynb}. This is a Jupyter Notebook file running on a Python3 kernel. It contains code and computations that go into proving the main theorems. The goal of this appendix is to explain the content of this file.

The appendix is divided into three parts. 
The first part is \Cref{subsec:code}, where we explain the code we use for performing the computations.

The second part contains \Cref{subsec:extraprop} and \Cref{subsec:extrasetup}. In \Cref{subsec:extraprop}, we prove an additional technical proposition regarding floral train tracks, which will be used in our analysis. In \Cref{subsec:extrasetup}, we build upon \Cref{sec:prelimobs} and introduce more setup. This includes explaining how to find a further pair of periodic orbits, and inequalities concerning the periods of these orbits.

The third part spans \Cref{sec:caseII} to \Cref{sec:caseVI}. In each of these subsections, we will walk through the main steps in showing \Cref{prop:casesII-VI} for one of the cases.

\subsection{Code for computing minimum growth rate} \label{subsec:code}

The file \texttt{pAbraid\_rootcheck.ipynb} contains two functions, \texttt{newton\_approx} and \texttt{MinPFRootApprox}.

The function \texttt{newton\_approx} is a straightforward implementation of Newton's method in finding a root of a differentiable function $f(x)$. More importantly, it is able to certify the found root $x_0$ within a pre-specified margin of error $\epsilon$. This is done by checking that $f(x_0 - \epsilon) < 0$ and $f(x_0 + \epsilon) > 0$, or vice versa, and appealing to the intermediate value theorem.

The function \texttt{MinPFRootApprox} is used for finding the minimum value of the reciprocal of the minimum positive root of a function $f(x,a,b,c,d)$, as $a,b,c,d$ ranges over real values. In the context of this paper, $f$ is the clique polynomial of the curve complex $(G,w)$ of a strongly connected directed graph (with a variable change $x = t^{\frac{1}{n}}$) and $a,b,c,d$ are linear parameters for the weight $w$. 
Thus by \Cref{thm:cliquepolycomputesgrowthrate}, the reciprocal of the minimum positive root of $f(x,a,b,c,d)$ is the growth rate $\lambda(G,w)$, and by \Cref{prop:growthrateprop}(5), this value is a convex function in $a,b,c,d$.

The function \texttt{MinPFRootApprox} works by calling \texttt{newton\_approx} to compute (good enough approximations of) $\lambda(G,w) = \lambda(a,b,c,d)$ for each value of $(a,b,c,d)$ in a grid. It then determines the value of $(a,b,c,d)$ that gives the minimum of $\lambda$ within this grid, and repeats the process with a smaller grid centered at this value of $(a,b,c,d)$. The process terminates once we are able to certify that we are within a pre-specified margin of error $\epsilon$, using the following lemma.

\begin{lemma} \label{lemma:convex}
Let $\lambda:\mathbb{R}^n \to \mathbb{R}$ be a convex function. Suppose there exists $\mathbf{a}_0 \in \mathbb{R}^n$, $\epsilon > 0$, and $\delta > 0$ such that for every $i=1,...,n$ and $\sigma = \pm 1$,
$$\lambda(\mathbf{a}_0 + \sigma \delta e_i) > \lambda(\mathbf{a}_0) > \lambda(\mathbf{a}_0 + \sigma \delta e_i) - \epsilon$$
where $e_i$ denotes the $i^\text{th}$ standard basis vector in $\mathbb{R}^n$.

Then $\lambda$ attains its global minimum at a point within $\delta$ of $\mathbf{a}_0$ (in the $\ell^1$ metric), and the minimum value is $\geq \lambda(\mathbf{a}_0) - \epsilon$. 
\end{lemma}
\begin{proof}
The $\delta$-neighborhood of $\mathbf{a}_0$ (in the $\ell^1$ metric) is the union of $2^n$ regions
$$E(\sigma^{(1)},...,\sigma^{(n)}) = \{(a^{(1)},...,a^{(n)}) \mid \sum_{i=1}^n \sigma_i (a^{(i)} - a_0^{(i)}) \in [0,\delta] \}$$
where $(\sigma^{(1)},...,\sigma^{(n)})$ ranges over elements of $\{\pm 1\}^n$ and parametrizes the quadrants at $\mathbf{a}_0$.

For each $(\sigma^{(1)},...,\sigma^{(n)})$, by convexity, the value of $\lambda$ outside of $E(\sigma^{(1)},...,\sigma^{(n)})$ is bounded from below by the affine function 
$$\lambda(\mathbf{a}_0) + \frac{1}{\sigma^{(i)}\delta} \sum_{i=1}^n (\lambda(\mathbf{a}_0 + \sigma^{(i)} \delta e_i)-\lambda(\mathbf{a}_0))(a^{(i)} - a_0^{(i)}).$$
In particular, if $\mathbf{a}$ lies outside of the $\delta$-neighborhood of $\mathbf{a}_0$, then 
$$\lambda(\mathbf{a}) \geq \lambda(\mathbf{a}_0) + \frac{1}{\sigma^{(i)}\delta} \sum_{i=1}^n (\lambda(\mathbf{a}_0 + \sigma^{(i)} \delta e_i)-\lambda(\mathbf{a}_0))(a^{(i)} - a_0^{(i)}) > \lambda(\mathbf{a}_0)$$
for a quadrant $(\sigma^{(1)},...,\sigma^{(n)})$ that contains $\mathbf{a}$, using the fact that $\lambda(\mathbf{a}_0 + \sigma^{(i)} \delta e_i) > \lambda(\mathbf{a}_0)$.

Meanwhile, for $\mathbf{a}$ lying within the $\delta$-neighborhood of $\mathbf{a}_0$, we have 
\begin{align*}
\lambda(\mathbf{a}) &\geq \lambda(\mathbf{a}_0) + \frac{1}{\sigma^{(i)}\delta} \sum_{i=1}^n (\lambda(\mathbf{a}_0 + \sigma^{(i)} \delta e_i)-\lambda(\mathbf{a}_0))(a^{(i)} - a_0^{(i)}) \\
&> \lambda(\mathbf{a}_0) + \frac{1}{\sigma^{(i)}\delta} \sum_{i=1}^n \epsilon(a^{(i)} - a_0^{(i)}) \\
&\geq \lambda(\mathbf{a}_0) - \epsilon
\end{align*}
for the quadrant $(\sigma^{(1)},...,\sigma^{(n)})$ \emph{opposite} to one containing $\mathbf{a}$, using the fact that $\lambda(\mathbf{a}_0) > \lambda(\mathbf{a}_0 + \sigma^{(i)} \delta e_i) - \epsilon$.

See \Cref{fig:convex} for a pictorial summary of the proof.
\end{proof}

\begin{figure}
    \centering
    \selectfont \fontsize{10pt}{10pt}
    %% Creator: Inkscape 1.3 (0e150ed6c4, 2023-07-21), www.inkscape.org
%% PDF/EPS/PS + LaTeX output extension by Johan Engelen, 2010
%% Accompanies image file 'convex.pdf' (pdf, eps, ps)
%%
%% To include the image in your LaTeX document, write
%%   \input{<filename>.pdf_tex}
%%  instead of
%%   \includegraphics{<filename>.pdf}
%% To scale the image, write
%%   \def\svgwidth{<desired width>}
%%   \input{<filename>.pdf_tex}
%%  instead of
%%   \includegraphics[width=<desired width>]{<filename>.pdf}
%%
%% Images with a different path to the parent latex file can
%% be accessed with the `import' package (which may need to be
%% installed) using
%%   \usepackage{import}
%% in the preamble, and then including the image with
%%   \import{<path to file>}{<filename>.pdf_tex}
%% Alternatively, one can specify
%%   \graphicspath{{<path to file>/}}
%% 
%% For more information, please see info/svg-inkscape on CTAN:
%%   http://tug.ctan.org/tex-archive/info/svg-inkscape
%%
\begingroup%
  \makeatletter%
  \providecommand\color[2][]{%
    \errmessage{(Inkscape) Color is used for the text in Inkscape, but the package 'color.sty' is not loaded}%
    \renewcommand\color[2][]{}%
  }%
  \providecommand\transparent[1]{%
    \errmessage{(Inkscape) Transparency is used (non-zero) for the text in Inkscape, but the package 'transparent.sty' is not loaded}%
    \renewcommand\transparent[1]{}%
  }%
  \providecommand\rotatebox[2]{#2}%
  \newcommand*\fsize{\dimexpr\f@size pt\relax}%
  \newcommand*\lineheight[1]{\fontsize{\fsize}{#1\fsize}\selectfont}%
  \ifx\svgwidth\undefined%
    \setlength{\unitlength}{186.59791721bp}%
    \ifx\svgscale\undefined%
      \relax%
    \else%
      \setlength{\unitlength}{\unitlength * \real{\svgscale}}%
    \fi%
  \else%
    \setlength{\unitlength}{\svgwidth}%
  \fi%
  \global\let\svgwidth\undefined%
  \global\let\svgscale\undefined%
  \makeatother%
  \begin{picture}(1,0.70566599)%
    \lineheight{1}%
    \setlength\tabcolsep{0pt}%
    \put(0,0){\includegraphics[width=\unitlength,page=1]{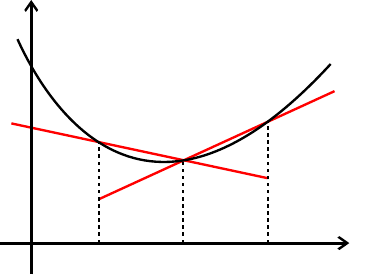}}%
    \put(0.44561471,0.0202888){\color[rgb]{0,0,0}\makebox(0,0)[lt]{\lineheight{1.25}\smash{\begin{tabular}[t]{l}$\mathbf{a}_0$\end{tabular}}}}%
    \put(0.17578402,0.0202888){\color[rgb]{0,0,0}\makebox(0,0)[lt]{\lineheight{1.25}\smash{\begin{tabular}[t]{l}$\mathbf{a}_0-\delta$\end{tabular}}}}%
    \put(0.60987242,0.0202888){\color[rgb]{0,0,0}\makebox(0,0)[lt]{\lineheight{1.25}\smash{\begin{tabular}[t]{l}$\mathbf{a}_0+\delta$\end{tabular}}}}%
  \end{picture}%
\endgroup%

    \caption{For a convex function $f$ as in \Cref{lemma:convex}, we can use the affine functions in red to argue that $\lambda(\mathbf{a}) > \lambda(\mathbf{a}_0)$ for $\mathbf{a}$ lying outside of the $\delta$-neighborhood of $\mathbf{a}_0$, and $\lambda(\mathbf{a}) > \lambda(\mathbf{a}_0) - \epsilon$ for $\mathbf{a}$ lying within the $\delta$-neighborhood of $\mathbf{a}_0$.}
    \label{fig:convex}
\end{figure}

We remark that the choice for \texttt{MinPFRootApprox} to work for functions with 4 parameters $a,b,c,d$ comes from a matter of practicality; it is the maximum number of parameters we need for our computations. One can easily rewrite the function to work with more parameters.

\subsection{An extra proposition} \label{subsec:extraprop}

Let $g:\tau \to \tau$ be a floral train track map as in \Cref{subsec:floraltt}.
Suppose $\alpha =(\epsilon_k:e_k \to e_{k+1})_{k \in \mathbb{Z}/p}$ is an orientation-preserving cycle determined by a periodic orbit $\mathbf{a}$, where $\mathbf{a} \not\subset \mathcal{X}_I$, and suppose $\alpha$ passes through petals only. 
Then the second item of \Cref{prop:curvesenterexit} implies that $\sum_k \sum_e G_{e,e_k} \geq 1$. 
With \Cref{prop:floralttpetalexit} we can upgrade this to $\sum_k \sum_e G_{e,e_k} \geq 2$.

In fact, we can often do even better: If some $\sum_e G_{e,e_k} > 2$ for some value of $k$, or if $\sum_e G_{e,e_k} > 0$ for two values of $k$, then by \Cref{prop:floralttpetalexit}, $\sum_k \sum_e G_{e,e_k} \geq 4$. 
The following proposition says that if this is not true, then we can also extract some information.
We will need this proposition in our analysis in cases IV and VI.

\begin{prop} \label{prop:floraltttwoedgesexitouterpetal}
Suppose $\alpha =(\epsilon_k:e_k \to e_{k+1})_{k \in \mathbb{Z}/p}$ is an orientation-preserving cycle determined by a periodic orbit $\mathbf{a}$, where $\mathbf{a} \not\subset \mathcal{X}_I$, and suppose $\alpha$ passes through petals only. If $\sum_k \sum_e G_{e,e_k} = 2$, then either
\begin{itemize}
    \item $G_{p,e_k} > 0$ for some petal $p$ and some (unique) $k$, or
    \item $G_{f,e_k} = 2$ for some filament $f$ and some (unique) $k$, and there exists some filament $f'$ such that $G_{f,f'} \geq 1$.
\end{itemize}
\end{prop}
\begin{proof}
By \Cref{prop:floralttpetalexit}, if $\sum_k \sum_e G_{e,e_k} = 2$ then $\sum_e G_{e,e_k} = 2$ for some unique $k$. Moreover, in the notation of the proof of that proposition, the smooth edge path $g(e_k)$ must be of the form $g_1 c_1 e_{k+1} c_2 g_3$, where each $g_i$ is of the form
\begin{itemize}
    \item $f_i^{-1} a_i f_i$, where $f_i$ is a filament, or
    \item $p_i$, where $p_i$ is a petal,
\end{itemize}
If $g_1$ or $g_3$ is of the second form, then we belong to the first case. If $g_1$ and $g_3$ are both of the first form, we claim that $f_1=f_3$.

To see this, let $R$ be the rectangle corresponding to $e_k$. Let $u$ be the unstable side of $R$ that lies closer to the outer boundary component $b$. Let $v$ be the vertex of the pistil boundary component $c$ that is the common endpoint of $e_k$. There is a segment $s$ of the side of the stable star at $c$ corresponding to $v$ such that $u \cup s$ bounds a disc $D$ not containing $b$ in the sphere $\overline{S}$. See \Cref{fig:floraltttwoedgesexitpetal1} top.

\begin{figure}
    \centering
    \selectfont \fontsize{6pt}{6pt}
    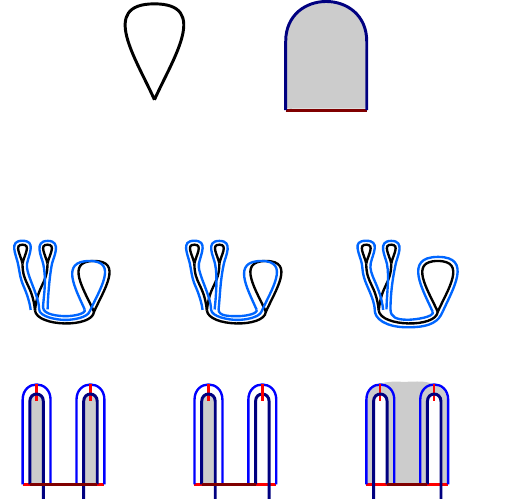
    \caption{If $f_1 \neq f_3$, then either $\overline{f}(u)$ intersects $\overline{f}(s)$ (bottom left and middle) or $\overline{f}(D)$ contains $c$ (bottom right).}
    \label{fig:floraltttwoedgesexitpetal1}
\end{figure}

Taking the image under $\overline{f}$, $\overline{f}(u) \cup \overline{f}(s)$ must bound a disc $\overline{f}(D)$ that does not contain $b$ as well. 

Now assume $f_1 \neq f_3$. Then depending on the relative positions of the subpaths of $\overline{f}(u)$ in the rectangles corresponding to $f_1$ and $f_3$, either $\overline{f}(u)$ intersects $\overline{f}(s)$ as in \Cref{fig:floraltttwoedgesexitpetal1} bottom left and middle, or $\overline{f}(D)$ contains $b$ as in \Cref{fig:floraltttwoedgesexitpetal1} bottom right. In both cases we arrive at a contradiction, thus $f_1$ and $f_3$ is some common filament $f$. 

Meanwhile, $e_k$ must enclose some 1-pronged puncture, i.e. $D$ must contain some 1-pronged puncture. Indeed, consider the innermost non-1-pronged boundary component $d$ that is enclosed by $e_k$. If $d$ is $n$-pronged then it must enclose $n-1$ 1-pronged infinitesimal polygons. Let $f'$ be a filament connecting to one of these 1-pronged infinitesimal polygons. Then we must have $G_{f,f'} \geq 1$. See \Cref{fig:floraltttwoedgesexitpetal2}.
\end{proof}

\begin{figure}
    \centering
    \selectfont \fontsize{10pt}{10pt}
    %% Creator: Inkscape 1.3 (0e150ed6c4, 2023-07-21), www.inkscape.org
%% PDF/EPS/PS + LaTeX output extension by Johan Engelen, 2010
%% Accompanies image file 'floraltttwoedgesexitpetal2.pdf' (pdf, eps, ps)
%%
%% To include the image in your LaTeX document, write
%%   \input{<filename>.pdf_tex}
%%  instead of
%%   \includegraphics{<filename>.pdf}
%% To scale the image, write
%%   \def\svgwidth{<desired width>}
%%   \input{<filename>.pdf_tex}
%%  instead of
%%   \includegraphics[width=<desired width>]{<filename>.pdf}
%%
%% Images with a different path to the parent latex file can
%% be accessed with the `import' package (which may need to be
%% installed) using
%%   \usepackage{import}
%% in the preamble, and then including the image with
%%   \import{<path to file>}{<filename>.pdf_tex}
%% Alternatively, one can specify
%%   \graphicspath{{<path to file>/}}
%% 
%% For more information, please see info/svg-inkscape on CTAN:
%%   http://tug.ctan.org/tex-archive/info/svg-inkscape
%%
\begingroup%
  \makeatletter%
  \providecommand\color[2][]{%
    \errmessage{(Inkscape) Color is used for the text in Inkscape, but the package 'color.sty' is not loaded}%
    \renewcommand\color[2][]{}%
  }%
  \providecommand\transparent[1]{%
    \errmessage{(Inkscape) Transparency is used (non-zero) for the text in Inkscape, but the package 'transparent.sty' is not loaded}%
    \renewcommand\transparent[1]{}%
  }%
  \providecommand\rotatebox[2]{#2}%
  \newcommand*\fsize{\dimexpr\f@size pt\relax}%
  \newcommand*\lineheight[1]{\fontsize{\fsize}{#1\fsize}\selectfont}%
  \ifx\svgwidth\undefined%
    \setlength{\unitlength}{265.02637247bp}%
    \ifx\svgscale\undefined%
      \relax%
    \else%
      \setlength{\unitlength}{\unitlength * \real{\svgscale}}%
    \fi%
  \else%
    \setlength{\unitlength}{\svgwidth}%
  \fi%
  \global\let\svgwidth\undefined%
  \global\let\svgscale\undefined%
  \makeatother%
  \begin{picture}(1,0.36350487)%
    \lineheight{1}%
    \setlength\tabcolsep{0pt}%
    \put(0,0){\includegraphics[width=\unitlength,page=1]{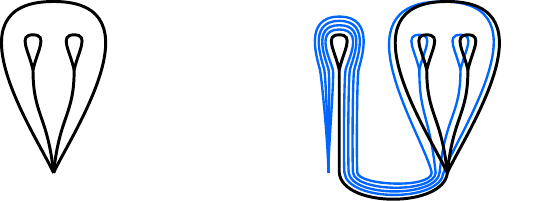}}%
    \put(0.10900515,0.02558668){\color[rgb]{0,0,0}\makebox(0,0)[lt]{\lineheight{1.25}\smash{\begin{tabular}[t]{l}$v$\end{tabular}}}}%
    \put(0.20009683,0.32782063){\color[rgb]{0,0,0}\makebox(0,0)[lt]{\lineheight{1.25}\smash{\begin{tabular}[t]{l}$e_k$\end{tabular}}}}%
    \put(0.91323107,0.32782063){\color[rgb]{0,0,0}\makebox(0,0)[lt]{\lineheight{1.25}\smash{\begin{tabular}[t]{l}$e_{k+1}$\end{tabular}}}}%
    \put(0.0888038,0.19829094){\color[rgb]{0,0,0}\makebox(0,0)[lt]{\lineheight{1.25}\smash{\begin{tabular}[t]{l}$f'$\end{tabular}}}}%
    \put(0,0){\includegraphics[width=\unitlength,page=2]{floraltttwoedgesexitpetal2.pdf}}%
  \end{picture}%
\endgroup%

    \caption{Let $f'$ be a filament enclosed by $e_k$. Then $G_{f,f'} \geq 1$.}
    \label{fig:floraltttwoedgesexitpetal2}
\end{figure}

\subsection{Extra setup} \label{subsec:extrasetup}

Recall the setup established in \Cref{sec:prelimobs}.
In particular:
\begin{itemize}
    \item $f$ is a pseudo-Anosov braid with $n$ strands (or $n+1$ strands, if we have to puncture at $c$),
    \item $b$ and $c$ are fixed points,
    \item $p_0$ and $q_0$ are the number of prongs at $b$ and $c$, respectively,
    \item $r_0$ is the number of secondary singular points, and
    \item by \Cref{lemma:accounting}, we have $p_0+q_0+r_0 \leq n$.
\end{itemize}

We now denote by $p$ and $q$ the rotationless periods of $b$, and $c$ respectively. 
We have $p \mid p_0$ and $q \mid q_0$.
In particular from \Cref{lemma:accounting}, we have $p+q+r_0 \leq n$.

Also, up to switching $b$ and $c$, we can assume that either
\begin{equation} \label{eq:outernotpetalcurve}
\begin{cases}
q \neq r_0 \\
\text{or} \\
p=q=r_0.
\end{cases}
\end{equation}

The following lemma states that $b$ and $c$ induce another pair of periodic points.

\begin{lemma} \label{lemma:secondarycycles}
There exists periodic points $b'$ and $c'$ such that
\begin{itemize}
    \item $b'$ is a rotated fixed point of $f^p$ and $c'$ is a rotated fixed point of $f^q$, and
    \item the orbits of $b$, $b'$, $c$, $c'$ are distinct.
\end{itemize}

\end{lemma}
\begin{proof}
From \Cref{prop:lefschetzbraid} we know that $\sum_x \indL(\overline{f}^{p},x) = 2$, where the sum is taken over all fixed points of $\overline{f}^{p}$. From \Cref{eq:indL} we observe that $b$ contributes $1-p_0$, $c$ contributes at most $1$, and elements in $\mathbf{a}$ (possibly) contributes $0$.

Thus there are at least $p_0$ points other than $b$ and $c$ that are rotated fixed points of $\overline{f}^{p}$. We pick one such point $b'$. Let $p'_0$ be the period of $b'$.

We follow the same strategy for $c'$. 
From \Cref{prop:lefschetzbraid} we know that $\sum_x \indL(\overline{f}^{q},x) = 2$, where the sum is taken over all fixed points of $\overline{f}^{q}$. From \Cref{eq:indL} we observe that $b$ contributes at most $1$, $c$ contributes $1-q_0$, points in the orbit of $b'$ (possibly) contribute at most $p'_0$, and elements in $\mathbf{a}$ (possibly) contributes $0$.

If $p'_0$ does not divide $q$, i.e. points in the orbit of $b'$ do not contribute to the sum, then there are at least $q_0$ points other than $b$ and $c$ that are rotated fixed points of $\overline{f}^{q}$, and we can pick $c'$ to be one such point. 

If $p'_0$ does divide $q$, i.e. points in the orbit of $b'$ do contribute to the sum, then unless $p'_0=q$, there are at least $q_0-p'_0 \geq q - p'_0 > 0$ points other than $b$, $c$, and (possibly) the points in the orbit of $b'$ that is a rotated fixed point of $\overline{f}^{q}$, and we can pick $c'$ to be one such point.

The remaining case is if $p'_0=q$. This can only happen if $p \geq q$. Similarly, the symmetric strategy of first picking $c'$ then $b'$ works unless $q \geq p$.

Thus it remains to tackle the case when $p=q$.
In this case, we have to use a slightly different argument.
From \Cref{prop:lefschetzbraid} we know that $\sum_x \indL(\overline{f}^{p},x) = 2$, where the sum is taken over all fixed points of $\overline{f}^{p}$. From \Cref{eq:indL} we observe that $b$ contributes $1-p_0$, $c$ contributes $1-q_0$, and elements in $\mathbf{a}$ (possibly) contribute $0$. 

Thus there are at least $p_0+q_0 \geq 2p$ points other than $b$ and $c$ that are rotated fixed points of $\overline{f}^{p}$. There must be at least two orbits of such points. We pick $b'$ to be one such point, and pick $c'$ to be a point in some other orbit.
\end{proof}

Taking the periodic points $b'$ and $c'$ into account, we have the following upgrade of \Cref{lemma:accounting}.

\begin{lemma} \label{lemma:accounting+}
\leavevmode
\begin{enumerate}
    \item Suppose $b'$ is singular. We denote by $p'$ the rotationless period of $b'$. Then
    \begin{enumerate}
        \item $p_0+q_0+\frac{1}{3}p' \leq n$, and
        \item for every $\kappa \in [-1,2]$, we have $(\kappa-1)p_0+q_0+r_0+\frac{2-\kappa}{3}p' \leq n$
    \end{enumerate}
    \item Suppose $b'$ and $c'$ are singular. We denote by $p'$ and $q'$ the rotationless period of $b'$ and $c'$ respectively. Then
    \begin{enumerate}
        \item $p_0+q_0+\frac{1}{3}p'+\frac{1}{3}q' \leq n$, and
        \item for every $\kappa_1,\kappa_2 \in [-1,2]$, we have $(\kappa_1-1)p_0+(\kappa_2-1)q_0+r_0+\frac{2-\kappa_2}{3}p'+\frac{2-\kappa_2}{3}q' \leq n$
    \end{enumerate}
\end{enumerate}
\end{lemma}
\begin{proof}
Let us first suppose that $b'$ is singular and unpunctured. 
Let $p'_0$ be the period of $b'$ and $k \geq 3$ be the number of prongs at $b'$.
By splitting off the orbit of $b'$ from the last term of \Cref{eq:accounting}, we have
$$1-n \leq -\frac{m}{2} -\frac{p_0}{2} + (1-\frac{q_0}{2}) - (n-m) - \frac{r_0-(n-m)-p'_0}{2} + p'_0 (1-\frac{k}{2}).$$
Here we are assuming that $c$ is unpunctured. If this is not the case then one modifies the terms to arrive at the same inequality as in the proof of \Cref{lemma:accounting}.

To show (1a), we bound the fourth term by $-\frac{1}{2}(n-m)$ and bound the fifth term by $0$. 
This gives us $p_0 + q_0 + (k-2)p'_0 \leq n$. 
Since $k \geq 3$ and $p' \leq kp_0$, we have 
\begin{align*}
p_0+q_0+\frac{1}{3}p' &\leq p_0+q_0+\frac{k-2}{k}p' \\
&\leq p_0+q_0+(k-2)p'_0 \\
&\leq n.
\end{align*}

To show (1b), we retain all the terms, which gives us $p_0 + q_0 + r_0 + (k-3) p'_0 \leq n$.
Since $p'_0 \leq p \leq p_0$ and $p' \leq kp_0$, for every $\kappa \in [-1,2]$ we have
\begin{align*}
(\kappa-1)p_0 + q_0 + r_0 + \frac{2-\kappa}{3}p' &\leq (\kappa-1)p_0 + q_0 + r_0 + \frac{(2-\kappa)k}{3}p'_0 \\
&= (\kappa-1)p_0 + (2-\kappa)p'_0 + q_0 + r_0 + \frac{(2-\kappa)(k-3)}{3}p'_0 \\
&\leq (\kappa-1)p_0 + (2-\kappa)p_0 + q_0 + r_0 + (k-3)p'_0 \\
&= p_0 + q_0 + r_0 + (k-3)p'_0 \\
&\leq n.
\end{align*}

Now suppose $b'$ is singular and punctured.
Let $p'_0$ be the period of $b'$ and $k \geq 2$ be the number of prongs at $b'$.
By splitting off the orbit of $b'$ from the second-to-last term of \Cref{eq:accounting}, we have
$$1-n \leq -\frac{m}{2} -\frac{p_0}{2} + (1-\frac{q_0}{2}) - (n-m-p'_0) - \frac{r-(n-m)}{2} + p'_0 (-\frac{k}{2}).$$
Here, as above, we are assuming that $c$ is unpunctured.

To show (1a), we bound the fourth term by $-\frac{1}{2}(n-m-p'_0)$ and bound the fifth term by $0$. 
This gives us $p_0 + q_0 + (k-1)p'_0 \leq n$. 
Since $k \geq 2$ and $p' \leq kp_0$, we have 
\begin{align*}
p_0+q_0+\frac{1}{3}p' &\leq p_0+q_0+\frac{k-1}{k}p' \\
&\leq p_0+q_0+(k-1)p'_0 \\
&\leq n.
\end{align*}

For (1b), if $k \geq 3$, then the computation in the previous case carries through: For every $\kappa \in [-1,2]$ we have
\begin{align*}
(\kappa-1)p_0 + q_0 + r_0 + \frac{2-\kappa}{3}p' &\leq (\kappa-1)p_0 + (2-\kappa)p'_0 + q_0 + r_0 + \frac{(2-\kappa)(k-3)}{3}p'_0 \\
&\leq (\kappa-1)p_0 + (2-\kappa)p_0 + q_0 + r_0 + (k-3)p'_0 \\
&= p_0 + q_0 + r_0 + (k-3)p'_0 \\
&\leq p_0 + q_0 + r_0 + (k-2)p'_0 \\
&\leq n.
\end{align*}
If $k=2$, then the second inequality above is problematic, since $k-3 < 0$, so we have to use a different chain of inequalities:
\begin{align*}
(\kappa-1)p_0 + q_0 + r_0 + \frac{2-\kappa}{3}p' &\leq (\kappa-1)p_0 + (2-\kappa)p'_0 + q_0 + r_0 - \frac{2-\kappa}{3}p'_0 \\
&\leq (\kappa-1)p_0 + (2-\kappa)p_0 + q_0 + r_0 \\
&= p_0 + q_0 + r_0 \\
&\leq n.
\end{align*}

The inequalities in (2) can be proved using similar computations.
\end{proof}

\subsection{Case II: One of $\beta$ and $\gamma$ passes through both filaments and petals while the other passes through petals only} \label{sec:caseII}

The goal of this subsection is to prove the following proposition. 
The assumptions in the proposition are more general than the hypothesis for this case. The added generality will come into play in the later cases.

\begin{prop} \label{prop:caseII}
Suppose there are two cycles $\beta$ and $\gamma$ in the directed graph $\Gamma$ satisfying:
\begin{itemize}
    \item $\len(\beta) + \len(\gamma) + r_0 \leq n$,
    \item $\beta$ passes through both petals and filaments, and
    \item $\gamma$ passes through petals only.
\end{itemize}
Then $\lambda^n \geq 14.5$.
\end{prop}

Up to replacing $\beta$ with a curve that it contains, we can assume that it is embedded. Similarly, up to replacing $\gamma$ with a curve that it contains, we can assume that it is embedded.
We denote the lengths of $\beta$ and $\gamma$ by $p$ and $q$ respectively.
Throughout this subsection, we fix the canonical system of petal curves.

First, we rule out $\beta$ and $\gamma$ intersecting each other.

\begin{lemma}
If $\beta$ and $\gamma$ intersect, then $\lambda^n \geq 14.5$
\end{lemma}
\begin{proof}
In this case, the growth rate of $G$ is bounded from below by that of the following graph $G_1$: 

\begin{center}
\scalebox{\graphscale}{
\begin{tikzpicture}
\draw[draw=black, fill=black, thin, solid] (0,0) circle (0.1);
\node[black, anchor=east] at (-0.2,0) {$p$};
\node[black, anchor=north] at (0,-0.2) {$\times 2$};
\draw[draw=black, fill=black, thin, solid] (1,1) circle (0.1);
\node[black, anchor=south] at (1,1.2) {$m$};
\draw[draw=black, fill=black, thin, solid] (2,0) circle (0.1);
\node[black, anchor=west] at (2.2,0) {$q$};
\draw[draw=black, thin, solid] (2,0) -- (1,1);
\end{tikzpicture}}
\end{center}

Here, since $\gamma$ is a curve that passes through petals only, we have $q \leq r_0$, thus $p+2q \leq p+q+r_0 \leq n$.

Subject to the conditions $m \leq n$ and $p+2q \leq n$, we compute by hand, as in \Cref{lemma:leq2curves}, that $\lambda^n \geq \lambda(G_1,w_1)^n \geq 14.5$.
\end{proof}

Thus we can assume from this point onwards that $\beta$ and $\gamma$ are disjoint.

\begin{lemma}
If $\gamma$ is not a petal curve, then $\lambda^n \geq 33$.
\end{lemma}
\begin{proof}
In this case, the growth rate of $G$ is bounded from below by that of the following graph $G_1$: 

\begin{center}
\scalebox{\graphscale}{
\begin{tikzpicture}
\draw[draw=black, fill=black, thin, solid] (0,0) circle (0.1);
\node[black, anchor=east] at (-0.2,0) {$p$};
\node[black, anchor=north] at (0,-0.2) {$\times 2$};
\draw[draw=black, fill=black, thin, solid] (1,1) circle (0.1);
\node[black, anchor=south] at (1,1.2) {$m$};
\draw[draw=black, fill=black, thin, solid] (2,0) circle (0.1);
\node[black, anchor=west] at (2.2,0) {$q$};
\draw[draw=black, thin, solid] (0,0) -- (2,0);
\draw[draw=black, thin, solid] (1,1) -- (2,0);
\draw[draw=black, fill=black, thin, solid] (1,-1) circle (0.1);
\node[black, anchor=north] at (1,-1.2) {$r_0$};
\draw[draw=black, thin, solid] (1,-1) -- (1,1);
\end{tikzpicture}}
\end{center}

Subject to the conditions $m \leq n$ and $p+q+r_0 \leq n$, we compute by hand that $\lambda^n \geq \lambda(G_1,w_1)^n \geq 33$.
\end{proof}

Thus we can assume from this point onwards that $\gamma$ is a petal curve.
Since we are assuming that $\beta$ and $\gamma$ are disjoint, $\beta$ must intersect some petal curve other than $\gamma$.
Since $\Gamma$ is strongly connected, there is some curve $\mu$ intersecting $\gamma$ and some filament curve or some petal curve other than $\gamma$.

\begin{lemma}
If $\mu$ does not pass through filaments, then $\lambda^n \geq 32$.
\end{lemma}
\begin{proof}
In this case, the growth rate of $G$ is bounded from below by that of the following graph $G_1$: 

\begin{center}
\scalebox{\graphscale}{
\begin{tikzpicture}
\draw[draw=black, fill=black, thin, solid] (0,0) circle (0.1);
\node[black, anchor=east] at (-0.2,0) {$p$};
\node[black, anchor=north] at (0,-0.2) {$\times 2$};
\draw[draw=black, fill=black, thin, solid] (2,0) circle (0.1);
\node[black, anchor=west] at (2.2,0) {$q$};
\draw[draw=black, thin, solid] (0,0) -- (2,0);
\draw[draw=black, fill=black, thin, solid] (0.4,-1) circle (0.1);
\node[black, anchor=east] at (0.2,-1) {$u$};
\draw[draw=black, fill=black, thin, solid] (1.6,-1) circle (0.1);
\node[black, anchor=west] at (1.8,-1) {$s$};
\draw[draw=black, thin, solid] (0.4,-1) -- (0,0);
\draw[draw=black, thin, solid] (1.6,-1) -- (2,0);
\end{tikzpicture}}
\end{center}
Here the bottom right vertex is the sum of the petal curves other than $\gamma$. In particular $q+s = r_0$, thus $p+2q+s = p+q+r_0 \leq n$.

Subject to the conditions $p+2q+s \leq n$ and $u \leq q+s$, we compute by hand that $\lambda^n \geq \lambda(G_1,w_1)^n \geq 32$.
\end{proof}

\begin{lemma}
If $\mu$ passes through filaments, then $\lambda^n \geq 14.5$.
\end{lemma}
\begin{proof}
We first suppose that $\mu$ meets $\beta$ and petal curves other than $\gamma$.
In this case, the growth rate of $G$ is bounded from below by that of the following graph $G_1$: 

\begin{center}
\scalebox{\graphscale}{
\begin{tikzpicture}
\draw[draw=black, fill=black, thin, solid] (0,0) circle (0.1);
\node[black, anchor=east] at (-0.2,0) {$p$};
\node[black, anchor=north] at (0,-0.2) {$\times 2$};
\draw[draw=black, fill=black, thin, solid] (1,1) circle (0.1);
\node[black, anchor=south] at (1,1.2) {$m$};
\draw[draw=black, fill=black, thin, solid] (2,0) circle (0.1);
\node[black, anchor=west] at (2.2,0) {$q$};
\draw[draw=black, thin, solid] (0,0) -- (2,0);
\draw[draw=black, thin, solid] (1,1) -- (2,0);
\draw[draw=black, fill=black, thin, solid] (0.4,-1) circle (0.1);
\node[black, anchor=east] at (0.2,-1) {$u$};
\node[black, anchor=north] at (0.4,-1.2) {$\times 2$};
\draw[draw=black, fill=black, thin, solid] (1.6,-1) circle (0.1);
\node[black, anchor=west] at (1.8,-1) {$s$};
\draw[draw=black, thin, solid] (1,1) -- (1.6,-1);
\draw[draw=black, thin, solid] (1.6,-1) -- (2,0);
\end{tikzpicture}}
\end{center}

Subject to the conditions $m \leq n$, $p+2q+s \leq n$, and $u \leq m+q+s$, we compute using our code that $\lambda^n \geq \lambda(G_1,w_1)^n \geq 18.9$.

Next, we suppose that $\mu$ meets petal curves other than $\gamma$, but not $\beta$.
In this case, the growth rate of $G$ is bounded from below by that of the following graph $G_1$: 

\begin{center}
\scalebox{\graphscale}{
\begin{tikzpicture}
\draw[draw=black, fill=black, thin, solid] (0,0) circle (0.1);
\node[black, anchor=east] at (-0.2,0) {$p$};
\node[black, anchor=north] at (0,-0.2) {$\times 2$};
\draw[draw=black, fill=black, thin, solid] (1,1) circle (0.1);
\node[black, anchor=south] at (1,1.2) {$m$};
\draw[draw=black, fill=black, thin, solid] (2,0) circle (0.1);
\node[black, anchor=west] at (2.2,0) {$q$};
\draw[draw=black, thin, solid] (0,0) -- (2,0);
\draw[draw=black, thin, solid] (1,1) -- (2,0);
\draw[draw=black, fill=black, thin, solid] (0.4,-1) circle (0.1);
\node[black, anchor=east] at (0.2,-1) {$u$};
\node[black, anchor=north] at (0.4,-1.2) {$\times 2$};
\draw[draw=black, fill=black, thin, solid] (1.6,-1) circle (0.1);
\node[black, anchor=west] at (1.8,-1) {$s$};
\draw[draw=black, thin, solid] (1,1) -- (1.6,-1);
\draw[draw=black, thin, solid] (0.4,-1) -- (0,0);
\draw[draw=black, thin, solid] (1.6,-1) -- (2,0);
\end{tikzpicture}}
\end{center}

Subject to the conditions $m \leq n$, $p+2q+s \leq n$, and $p+u \leq m+q+s$, we compute using our code that $\lambda^n \geq \lambda(G_1,w_1)^n \geq 21.2$.

The remaining case is if $\mu$ does not meet petal curves other than $\gamma$.
In this case, the growth rate of $G$ is bounded from below by that of the following graph $G_1$: 

\begin{center}
\scalebox{\graphscale}{
\begin{tikzpicture}
\draw[draw=black, fill=black, thin, solid] (0,0) circle (0.1);
\node[black, anchor=east] at (-0.2,0) {$p$};
\node[black, anchor=north] at (0,-0.2) {$\times 2$};
\draw[draw=black, fill=black, thin, solid] (1,1) circle (0.1);
\node[black, anchor=south] at (1,1.2) {$m$};
\draw[draw=black, fill=black, thin, solid] (2,0) circle (0.1);
\node[black, anchor=west] at (2.2,0) {$q$};
\draw[draw=black, thin, solid] (0,0) -- (2,0);
\draw[draw=black, thin, solid] (1,1) -- (2,0);
\draw[draw=black, fill=black, thin, solid] (0.4,-1) circle (0.1);
\node[black, anchor=east] at (0.2,-1) {$u$};
\node[black, anchor=north] at (0.4,-1.2) {$\times 2$};
\draw[draw=black, fill=black, thin, solid] (1.6,-1) circle (0.1);
\node[black, anchor=west] at (1.8,-1) {$s$};
\draw[draw=black, thin, solid] (1,1) -- (1.6,-1);
\draw[draw=black, thin, solid] (0.4,-1) -- (1.6,-1);
\draw[draw=black, thin, solid] (0.4,-1) -- (0,0);
\draw[draw=black, thin, solid] (1.6,-1) -- (2,0);
\end{tikzpicture}}
\end{center}

Subject to the conditions $m \leq n$, $p+2q+s \leq n$, and $s+u \leq m+q+s$, we compute using our code that $\lambda^n \geq \lambda(G_1,w_1)^n \geq 14.5$.

All computations in this lemma are done under the list \texttt{Case\_II}.
\end{proof}

\subsection{Case III: One of $\beta$ and $\gamma$ passes through filaments only while the other passes through petals only and is not persistent} \label{sec:caseIII}

The goal of this subsection is to prove the following proposition. As in case II, the assumptions in the proposition are more general than the hypothesis in this case.

\begin{prop} \label{prop:caseIII}
Let $g:\tau \to \tau$ be a floral train track that carries $f$. Suppose there are two cycles $\beta$ and $\gamma$ in the directed graph $\Gamma$ satisfying:
\begin{itemize}
    \item $\len(\beta) + \len(\gamma) + r_0 \leq n$,
    \item $\beta$ passes through filaments only and is not a filament curve, and
    \item $\gamma$ passes through petals only and is not persistent.
\end{itemize}
Then $\lambda^n \geq 14.8$.
\end{prop}

Up to replacing $\beta$ with a curve that it contains, we can assume that it is embedded. Similarly, up to replacing $\gamma$ with a curve that it contains, we can assume that it is embedded.
We denote the lengths of $\beta$ and $\gamma$ by $p$ and $q$ respectively.

We fix some system of petal curves that does not contain $\gamma$. Let $r$ be the sum of lengths of the petal curves.

\begin{lemma}
If $\gamma$ intersects at least two petal curves, then $\lambda^n \geq 33$.
\end{lemma}
\begin{proof}
In this case, the growth rate of $G$ is bounded from below by that of the following graph $G_1$: 

\begin{center}
\scalebox{\graphscale}{
\begin{tikzpicture}
\draw[draw=black, fill=black, thin, solid] (0,0) circle (0.1);
\node[black, anchor=east] at (-0.2,0) {$p$};
\node[black, anchor=north] at (0,-0.2) {$\times 2$};
\draw[draw=black, fill=black, thin, solid] (1,1) circle (0.1);
\node[black, anchor=south] at (1,1.2) {$m$};
\draw[draw=black, fill=black, thin, solid] (2,0) circle (0.1);
\node[black, anchor=west] at (2.2,0) {$q$};
\draw[draw=black, thin, solid] (0,0) -- (2,0);
\draw[draw=black, thin, solid] (1,1) -- (2,0);
\draw[draw=black, fill=black, thin, solid] (0.4,-1) circle (0.1);
\node[black, anchor=east] at (0.2,-1) {$s$};
\draw[draw=black, fill=black, thin, solid] (1.6,-1) circle (0.1);
\node[black, anchor=west] at (1.8,-1) {$t$};
\draw[draw=black, thin, solid] (1,1) -- (1.6,-1);
\draw[draw=black, thin, solid] (0,0) -- (1.6,-1);
\draw[draw=black, thin, solid] (0.4,-1) -- (1.6,-1);
\draw[draw=black, thin, solid] (0.4,-1) -- (0,0);
\draw[draw=black, thin, solid] (0.4,-1) -- (1,1);
\end{tikzpicture}}
\end{center}
where $s$ and $t$ are lengths of two petal curves which $\gamma$ intersect.

Subject to the conditions $m \leq n$ and $p+q+s+t \leq n$, we compute by hand that $\lambda^n \geq \lambda(G_1,w_1)^n \geq 33$.
\end{proof}

Thus we can assume from this point onwards that $\gamma$ only intersects one petal curve $\delta$.
Since $\gamma \neq \delta$, there is an edge $\epsilon$ of $\gamma$ at which $\gamma$ exits $\delta$. By \Cref{prop:floralttpetalexit}, there is another edge $\epsilon'$ exiting $\delta$ at the same vertex as $\epsilon$ and which is not a double of $\epsilon$. We let $\mu$ be a curve that passes through $\epsilon'$. 
Note that $\mu$ intersects $\gamma$ and $\delta$.

\begin{lemma}
If $\mu$ does not pass through filaments, then $\lambda^n \geq 17$.
\end{lemma}
\begin{proof}
In this case, the growth rate of $G$ is bounded from below by that of the following graph $G_1$: 

\begin{center}
\scalebox{\graphscale}{
\begin{tikzpicture}
\draw[draw=black, fill=black, thin, solid] (0,0) circle (0.1);
\node[black, anchor=east] at (-0.2,0) {$p$};
\node[black, anchor=north] at (0,-0.2) {$\times 2$};
\draw[draw=black, fill=black, thin, solid] (1,1) circle (0.1);
\node[black, anchor=south] at (1,1.2) {$m$};
\draw[draw=black, fill=black, thin, solid] (2,0) circle (0.1);
\node[black, anchor=west] at (2.2,0) {$q$};
\draw[draw=black, thin, solid] (0,0) -- (2,0);
\draw[draw=black, thin, solid] (1,1) -- (2,0);
\draw[draw=black, fill=black, thin, solid] (0.4,-1) circle (0.1);
\node[black, anchor=east] at (0.2,-1) {$u$};
\draw[draw=black, fill=black, thin, solid] (1.6,-1) circle (0.1);
\node[black, anchor=west] at (1.8,-1) {$r$};
\draw[draw=black, thin, solid] (1,1) -- (1.6,-1);
\draw[draw=black, thin, solid] (0,0) -- (1.6,-1);
\draw[draw=black, thin, solid] (0.4,-1) -- (0,0);
\draw[draw=black, thin, solid] (0.4,-1) -- (1,1);
\end{tikzpicture}}
\end{center}

Subject to the conditions $m \leq n$, $p+q+r \leq n$, and $u \leq r$, we compute by hand that $\lambda^n \geq \lambda(G_1,w_1)^n \geq 17$.
\end{proof}

\begin{lemma}
If $\mu$ passes through filaments, then $\lambda^n \geq 14.8$.
\end{lemma}
\begin{proof}
We first suppose that $\mu$ meets $\beta$.
In this case, the growth rate of $G$ is bounded from below by that of the following graph $G_1$: 

\begin{center}
\scalebox{\graphscale}{
\begin{tikzpicture}
\draw[draw=black, fill=black, thin, solid] (0,0) circle (0.1);
\node[black, anchor=east] at (-0.2,0) {$p$};
\node[black, anchor=north] at (0,-0.2) {$\times 2$};
\draw[draw=black, fill=black, thin, solid] (1,1) circle (0.1);
\node[black, anchor=south] at (1,1.2) {$m$};
\draw[draw=black, fill=black, thin, solid] (2,0) circle (0.1);
\node[black, anchor=west] at (2.2,0) {$q$};
\draw[draw=black, thin, solid] (0,0) -- (2,0);
\draw[draw=black, thin, solid] (1,1) -- (2,0);
\draw[draw=black, fill=black, thin, solid] (0.4,-1) circle (0.1);
\node[black, anchor=east] at (0.2,-1) {$u$};
\node[black, anchor=north] at (0.4,-1.2) {$\times 2$};
\draw[draw=black, fill=black, thin, solid] (1.6,-1) circle (0.1);
\node[black, anchor=west] at (1.8,-1) {$r$};
\draw[draw=black, thin, solid] (1,1) -- (1.6,-1);
\draw[draw=black, thin, solid] (0,0) -- (1.6,-1);
\end{tikzpicture}}
\end{center}

Subject to the conditions $m \leq n$, $p+q+r \leq n$, and $u \leq m+r$, we compute using our code that $\lambda^n \geq \lambda(G_1,w_1)^n \geq 16.9$.

Thus we can assume that $\mu$ does not meet $\beta$.
In this case, the growth rate of $G$ is bounded from below by that of the following graph $G_1$: 

\begin{center}
\scalebox{\graphscale}{
\begin{tikzpicture}
\draw[draw=black, fill=black, thin, solid] (0,0) circle (0.1);
\node[black, anchor=east] at (-0.2,0) {$p$};
\node[black, anchor=north] at (0,-0.2) {$\times 2$};
\draw[draw=black, fill=black, thin, solid] (1,1) circle (0.1);
\node[black, anchor=south] at (1,1.2) {$m$};
\draw[draw=black, fill=black, thin, solid] (2,0) circle (0.1);
\node[black, anchor=west] at (2.2,0) {$q$};
\draw[draw=black, thin, solid] (0,0) -- (2,0);
\draw[draw=black, thin, solid] (1,1) -- (2,0);
\draw[draw=black, fill=black, thin, solid] (0.4,-1) circle (0.1);
\node[black, anchor=east] at (0.2,-1) {$u$};
\node[black, anchor=north] at (0.4,-1.2) {$\times 2$};
\draw[draw=black, fill=black, thin, solid] (1.6,-1) circle (0.1);
\node[black, anchor=west] at (1.8,-1) {$r$};
\draw[draw=black, thin, solid] (1,1) -- (1.6,-1);
\draw[draw=black, thin, solid] (0,0) -- (1.6,-1);
\draw[draw=black, thin, solid] (0.4,-1) -- (0,0);
\end{tikzpicture}}
\end{center}

Subject to the conditions $m \leq n$, $p+q+r \leq n$, and $p+u \leq m+r$, we compute using our code that $\lambda^n \geq \lambda(G_1,w_1)^n \geq 14.8$.

All computations in this lemma are done under the list \texttt{Case\_III}.
\end{proof}

\subsection{Case IV: One of $\beta$ and $\gamma$ passes through filaments only while the other passes through petals only and is persistent} \label{sec:caseIV}

The goal of this subsection is to prove the following proposition.

\begin{prop} \label{prop:caseIV}
Define $\beta$ and $\gamma$ as in \Cref{subsec:casedivision}. Suppose 
\begin{itemize}
    \item one of $\beta$ and $\gamma$ passes through filaments only, while 
    \item the other passes through petals only and is persistent.
\end{itemize}
Then $\lambda^n \geq 14.8$.
\end{prop}

Note that unlike the previous 3 cases, \Cref{prop:caseIV} is only stated for the specific $\beta$ and $\gamma$ defined in \Cref{subsec:casedivision}. This is because we will need some specific properties of these curves towards the end of its proof.

For now, we first prove the following proposition, which restricts the possibilities for the petal curves.

\begin{prop} \label{prop:caseIVsolopetal}
Fix some system of petal curves. Suppose there are two cycles $\beta$ and $\gamma$ in the directed graph $\Gamma$ satisfying:
\begin{itemize}
    \item $\len(\beta) + \len(\gamma) + r_0 \leq n$,
    \item $\beta$ passes through filaments only and is not a filament curve, and
    \item $\gamma$ passes through petals only and is a petal curve.
\end{itemize}

Then either $\lambda^n \geq 14.8$, or 
\begin{itemize}
    \item $\gamma$ is the only petal curve,
    \item every curve other than $\gamma$ has to pass through filaments, and
    \item for every vertex $e$ on $\gamma$, if $\mu_1$ and $\mu_2$ are two curves exiting $\gamma$ through $e$, then $\mu_1$ and $\mu_2$ intersect a common filament curve.
\end{itemize}
\end{prop}

Up to replacing $\beta$ with a curve that it contains, we can assume that it is embedded.
We denote the lengths of $\beta$ and $\gamma$ by $p$ and $q$ respectively.
We will also denote the sum of the lengths of the petal curves other than $\gamma$ by $s$.

All computations towards proving \Cref{prop:caseIVsolopetal} will be done under the list \texttt{Prop\_A\_6\_2}.

By \Cref{prop:caseIII}, if there is a curve that passes through petals only and intersects $\gamma$, then we are done.
Thus we can assume that such a curve does not exist.

On the other hand, since $\Gamma$ is strongly connected, there must be an edge $\epsilon$ exiting $\gamma$. By \Cref{prop:floralttpetalexit}, there is another edge exiting $\gamma$ at the same vertex as $\epsilon$ and which is not a double of $\epsilon$. Let $\mu$ and $\nu$ be two curves passing through each of the two edges.
By our assumption in the previous paragraph, $\mu$ and $\nu$ must pass through filaments. Thus they have doubles $\mu'$ and $\nu'$. By construction, $\mu$, $\mu'$, $\nu$, $\nu'$ meet each other.

\begin{lemma} \label{lemma:caseIVnumeetgamma}
If $\mu$ and $\nu$ meet some petal curve other than $\gamma$, then $\lambda^n \geq 15.8$.
\end{lemma}
\begin{proof}
If $\mu$ and $\nu$ both meet $\beta$, then using the inequalities $m \leq n$, $p+2q+s \leq n$, and $u,v \leq n+q+s$, we compute that $\lambda^n \geq 17.4$.

If $\mu$ meets $\beta$ but $\nu$ does not, then using the inequalities $m \leq n$, $p+2q+s \leq n$, $u \leq n+q+s$, and $p+v \leq n+q+s$, we compute that $\lambda^n \geq 17.5$.

If $\mu$ and $\nu$ both do not meet $\beta$, then using the inequalities $m \leq n$, $p+2q+s \leq n$, and $p+u, p+v \leq n+q+s$, we compute that $\lambda^n \geq 15.8$.
\end{proof}

\begin{lemma} \label{lemma:caseIVnumissgamma}
If $\mu$ meets some petal curve other than $\gamma$, but $\nu$ does not, then $\lambda^n \geq 16.8$.
\end{lemma}
\begin{proof}
If $\mu$ and $\nu$ both meet $\beta$, then using the inequalities $m \leq n$, $p+2q+s \leq n$, $u \leq n+q+s$, and $v \leq n+q$, we compute that $\lambda^n \geq 17.6$.

If $\mu$ meets $\beta$ but $\nu$ does not, then using the inequalities $m \leq n$, $p+2q+s \leq n$, $u \leq n+q+s$, and $p+v \leq n+q$, we compute that $\lambda^n \geq 18.6$.

If $\nu$ meets $\beta$ but $\mu$ does not, then using the inequalities $m \leq n$, $p+2q+s \leq n$, $p+u \leq n+q+s$, and $v \leq n+q$, we compute that $\lambda^n \geq 18.5$.

If $\mu$ and $\nu$ both do not meet $\gamma$, then using the inequalities $m \leq n$, $p+2q+s \leq n$, $p+u \leq n+q+s$, and $p+v \leq n+q$, we compute that $\lambda^n \geq 16.8$.
\end{proof}

Similarly, if $\nu$ intersects some petal curve other than $\gamma$, but $\mu$ does not, then $\lambda^n \geq 16.8$.
Thus we can assume from this point that both $\mu$ and $\nu$ do not intersect petal curves other than $\gamma$.

Next, we rule out there being a petal curve other than $\gamma$.
Suppose otherwise, then since $\Gamma$ is strongly connected, there must be a curve $\omega$ intersecting such a petal curve $\delta$ and some filament curve or $\gamma$.
Let $\epsilon$ be an edge where $\omega$ exits $\delta$. There must be another edge exiting $\delta$ at the same vertex as $\epsilon$, and which is not a double of $\epsilon$. Let $\chi$ be a curve passing through such an edge.

We can assume that $\omega$ and $\chi$ do not meet $\gamma$, since otherwise we could have chosen $\mu$ to be such a curve, in which case we are done by \Cref{lemma:caseIVnumeetgamma} and \Cref{lemma:caseIVnumissgamma}.
In particular we can assume that $\omega$ passes through filaments.

\begin{lemma}
If $\chi$ does not pass through filaments, then $\lambda^n \geq 15.7$.
\end{lemma}
\begin{proof}
If $\omega$ meets $\beta$, then using the inequalities $m \leq n$, $p+2q+s \leq n$, $u \leq m+q$, $w \leq n+s$, and $x \leq s$, we compute that $\lambda^n \geq 16$.

If $\omega$ does not meet $\beta$, then using the inequalities $m \leq n$, $p+2q+s \leq n$, $u \leq m+q$, $p+w \leq n+s$, and $x \leq s$, we compute that $\lambda^n \geq 15.7$.
\end{proof}

\begin{lemma}
If $\chi$ passes through filaments, then $\lambda^n \geq 14.8$.
\end{lemma}
\begin{proof}
If $\omega$ and $\chi$ both meet $\beta$, then using the inequalities $m \leq n$, $p+2q+s \leq n$, $u \leq m+q$, and $w \leq m+s$, we compute that $\lambda^n \geq 14.8$.

If $\mu$ and $\nu$ both meet $\beta$, then using the inequalities $m \leq n$, $p+2q+s \leq n$, $u \leq m+q$, and $w \leq m+s$, we compute that $\lambda^n \geq 16.2$.

If at least one of $\mu$ and $\nu$, say $\mu$, does not meet $\beta$, and at least one of $\omega$ and $\chi$, say $\omega$, does not intersect $\beta$ then using the inequalities $m \leq n$, $p+2q+s \leq n$, $p+u \leq m+q$, $v \leq m+q$, $p+w \leq m+s$, and $x \leq m+s$, we compute that $\lambda^n \geq 16.6$.
\end{proof}

Thus we can assume from this point that $\gamma$ is the only petal curve.
Next, we argue that we can assume $\mu$ and $\nu$ passes through a common filament curve.

\begin{lemma}
If the collection of filament curves passed through by $\mu$ is disjoint from that of $\nu$, then $\lambda^n \geq 19.5$.
\end{lemma}
\begin{proof}
Without loss of generality, we can assume that $\beta$ intersects a filament curve that is disjoint from $\nu$. Let $m_1$ be the sum of lengths of filament curves that are disjoint from $\nu$, and let $m_2$ be the sum of the remaining filament curves. Using the inequalities $m_1+m_2 \leq n$, $p+2q \leq n$, $u \leq m_1+q$, and $v \leq m_2+q$, we compute that $\lambda^n \geq 19.5$.
\end{proof}

This concludes the proof of \Cref{prop:caseIVsolopetal}. 
For the rest of this subsection, we specialize to the canonical system of petal curves. In particular, the sum of the lengths of the petal curves is $r_0$.
To show \Cref{prop:caseIV}, we divide into two cases depending on whether $\beta$ or $\gamma$ is the curve that passes through filaments only.

\subsubsection{Case IVa: $\gamma$ is the curve that passes through filaments only} \label{subsubsec:caseIVinnerpetal}

By \Cref{prop:curvesenterexit}, there are at least two edges $\epsilon_1$ and $\epsilon_2$ entering $\gamma$ in this case. We claim that we can extract curves $\mu, \nu, \omega, \chi$ such that:
\begin{itemize}
    \item each of these curves pass through both filaments and petals, in particular they each have doubles,
    \item these curves and their doubles are distinct, and
    \item $\mu$ intersects $\nu$ and $\omega$ intersects $\chi$.
\end{itemize}
Indeed, we first pick a curve $\mu$ passing through $\epsilon_1$ and a curve $\omega$ passing through $\epsilon_2$. If $\mu$ and $\omega$ exit $\gamma$ from different vertices, then we can construct $\nu$ and $\chi$ by following the other edge exiting $\gamma$ at the respective vertex. Otherwise by \Cref{prop:caseIVsolopetal}, $\mu$ and $\omega$ pass through a common filament curve $\alpha$. We can construct $\nu$ by traversing $\epsilon_2$, then exit $\gamma$ by following $\mu$ until we reach $\alpha$, then start following $\omega$ until we close up at $\epsilon_2$. We can construct $\chi$ symmetrically.

\begin{proof}[Proof of \Cref{prop:caseIV} in case IVa]
If at least one of $\mu$ and $\nu$ is disjoint from at least one of $\omega$ and $\chi$, say $\mu$ and $\omega$ are disjoint, then using the inequalities $m \leq n$, $p+2q \leq n$, $u+w \leq m+q$, and $v,x \leq m+q$, we compute that $\lambda^n \geq 20.2$.
Thus we can assume that $\mu, \nu, \omega, \chi$ meet each other.

If all four curves meet $\beta$, then using the inequalities $m \leq n$, $p+2q \leq n$, and $u,v,w,x \leq n+q$, we compute that $\lambda^n \geq 16.9$.

By similar computations:
\begin{itemize}
    \item If exactly three of the curves meet $\beta$, then $\lambda^n \geq 18.4$.
    \item If exactly two of the curves meet $\beta$, then $\lambda^n \geq 18.8$.
    \item If exactly one of the curves meets $\beta$, then $\lambda^n \geq 18.5$.
    \item If all of the curves do not meet $\beta$, then $\lambda^n \geq 17.2$.
\end{itemize}

All computations in this lemma are done under the list \texttt{Case\_IV\_a}.
\end{proof}

\subsubsection{Case IVb: $\beta$ is the curve that passes through filaments only} \label{subsubsec:caseIVouterpetal}

By \Cref{eq:outernotpetalcurve}, we have $p=q=r_0$ in this case.
By \Cref{prop:floraltttwoedgesexitouterpetal}, either 
\begin{enumerate}
    \item there are two other edges that exit $\gamma$, or
    \item the edges from which $\mu$ and $\nu$ exit $\gamma$ are the only edges that exit $\gamma$, in which case they have the same terminal vertex $f$, and there is a third edge $\epsilon$ entering $f$.
\end{enumerate} 

In case (1), we can construct four curves and reason as in \Cref{subsubsec:caseIVinnerpetal}. Thus we assume that we are in case (2). 
In this case, we can choose $\nu$ to pass through the same set of vertices as $\mu$.
We let $\omega$ be a curve passing through $\epsilon$. In particular, $\omega$ intersects $\mu$ and $\nu$.

\begin{proof}[Proof of \Cref{prop:caseIV} in case IVb]
If $\mu$ and $\nu$ meet $\beta$, then subject to the conditions $m \leq n$, $3p \leq n$, and $u \leq n+p$, we compute that $\lambda^n \geq 16.5$.
Thus we can assume that $\mu$ and $\nu$ do not meet $\beta$.

If $\omega$ meets $\beta$ and $\gamma$, then subject to the conditions $m \leq n$, $3p \leq n$, $p+u \leq n+p$, and $w \leq n+p$, we compute that $\lambda^n \geq 15.7$.

By similar computations,
\begin{itemize}
    \item if $\omega$ meets $\beta$ but not $\gamma$, then $\lambda^n \geq 18.6$,
    \item if $\omega$ meets $\gamma$ but not $\beta$, then $\lambda^n \geq 15.2$,
    \item if $\omega$ does not meet $\beta$ and $\gamma$, then $\lambda^n \geq 16.3$.
\end{itemize}

All computations in this lemma are done under the list \texttt{Case\_IV\_b}.
\end{proof}

\subsection{Case V: Both $\beta$ and $\gamma$ pass through petals only, and at least one of them is not persistent} \label{sec:caseV}

The goal of this subsection is to prove the following proposition.

\begin{prop} \label{prop:caseV}
Define $\beta$ and $\gamma$ as in \Cref{subsec:casedivision}. Suppose 
\begin{itemize}
    \item both $\beta$ and $\gamma$ pass through petals only, and
    \item at least one of $\beta$ and $\gamma$ is not persistent.
\end{itemize}
Then $\lambda^n \geq 14.5$.
\end{prop}

We first show that we are done if we can find a system of petal curves that does not contain both $\beta$ and $\gamma$. 

\begin{prop} \label{prop:caseVbothnonpetalcurve}
Fix some system of petal curves. Suppose there are two cycles $\beta$ and $\gamma$ in the directed graph $\Gamma$ satisfying:
\begin{itemize}
    \item $\len(\beta) + \len(\gamma) + r_0 \leq n$, and
    \item $\beta$ and $\gamma$ pass through petals only and are not petal curves.
\end{itemize}
Then $\lambda^n \geq 16$.
\end{prop}
\begin{proof}
By \Cref{lemma:leq2curves}, we can assume that each of $\beta$ and $\gamma$ contains at most two curves. 
Under this assumption, there are three possibilities for $\beta$:
\begin{enumerate}
    \item $\beta$ contains one non-petal curve,
    \item $\beta$ contains two non-petal curves, or
    \item $\beta$ contains one non-petal curve and one petal curve,
\end{enumerate}
and similarly for $\gamma$.

If both $\beta$ and $\gamma$ satisfy (1), then they are embedded. If at least one of $\beta$ and $\gamma$ satisfy (2), then up to replacing $\beta$ and $\gamma$ by non-petal curves that they contain, we can assume that $\beta$ and $\gamma$ are embedded.

If $\beta$ satisfies (3), then the same statement holds unless the non-petal curve that $\beta$ contains is also contained in $\gamma$. But in that case we have $\len(\beta) \leq \len(\gamma)+r_0$, which implies that $\len(\beta) \leq \frac{n}{2}$. By \Cref{lemma:halfnemb}, we would then have $\lambda^n \geq 16$.
Similarly, if $\gamma$ satisfies (3), then we can assume that $\beta$ and $\gamma$ are embedded.

Under this assumption, the growth rate of $G$ is bounded from below by that of the following graph $G_1$: 

\begin{center}
\scalebox{\graphscale}{
\begin{tikzpicture}
\draw[draw=black, fill=black, thin, solid] (0,0) circle (0.1);
\node[black, anchor=east] at (-0.2,0) {$p$};
\draw[draw=black, fill=black, thin, solid] (2,0) circle (0.1);
\node[black, anchor=west] at (2.2,0) {$q$};
\draw[draw=black, thin, solid] (0,0) -- (2,0);
\draw[draw=black, fill=black, thin, solid] (1,-1) circle (0.1);
\node[black, anchor=north] at (1,-1.2) {$r$};
\end{tikzpicture}}
\end{center}

Subject to the condition $p+q+r \leq n$, we compute by hand that $\lambda^n \geq \lambda(G_1,w_1)^n \geq 16$.
\end{proof}

Unlike case IV, $\beta$ and $\gamma$ will play a symmetric role in this case. Hence without loss of generality let us assume that $\gamma$ is not persistent. Pick some system of petal curves that does not contain $\gamma$. 

We first argue that we can assume $\gamma$ only meets one petal curve. Suppose otherwise that $\gamma$ meets petal curves $\delta_1$ and $\delta_2$, then the cycle obtained by composing $\gamma, \delta_1, \delta_2$ has length $q+r \leq n$, so by \Cref{lemma:leq2curves} we have $\lambda^n \geq 16$.

If the only petal curve that $\gamma$ meets is $\beta$, we argue that we can further assume $\gamma$ passes through the vertices of $\beta$ only. This would follow from the previous paragraph if every petal lies in one petal curve in our chosen system, for example if our system is the canonical one. 
So let us suppose that $\gamma$ is a canonical petal curve.

On the other hand, by \Cref{lemma:petalpersistentchar}, $\beta$ is also non-persistent in this case. The same reasoning as above shows that we can assume $\beta$ is a canonical petal curve. But since $\beta$ and $\gamma$ meet, we run into a contradiction.

Accordingly, we divide the rest of this subsection into two cases:
\begin{itemize}
    \item[(Va)] $\gamma$ passes through the vertices of $\beta$ only.
    \item[(Vb)] $\gamma$ meets a petal curve $\delta \neq \beta$ and no other petal curve.
\end{itemize}

\subsubsection{Case Va: $\gamma$ passes through the vertices of $\beta$ only}

We first prove the following proposition, which restricts the possibilities for the remaining petal curves.

\begin{prop} \label{prop:caseVasolopetal}
Fix some system of petal curves. Suppose there are two cycles $\beta$ and $\gamma$ in the directed graph $\Gamma$ satisfying:
\begin{itemize}
    \item $\len(\beta) + \len(\gamma) + r_0 \leq n$,
    \item $\beta$ and $\gamma$ pass through petals only,
    \item $\beta$ is a petal curve,
    \item $\gamma$ is not a petal curve, and
    \item $\gamma$ passes through the vertices of $\beta$ only.
\end{itemize}
Then either $\lambda^n \geq 15.1$ or
\begin{itemize}
    \item $\beta$ is the only petal curve, and
    \item every curve other than $\beta$ and $\gamma$ passes through filaments.
\end{itemize}
\end{prop}

Up to replacing $\gamma$ with a curve that it contains, we can assume that it is embedded.
We denote the lengths of $\beta$ and $\gamma$ by $p$ and $q$ respectively.
Throughout this subsubsection, we will denote the sum of the length of the petal curves other than $\beta$ by $s$.

All computations in this subsubsection will be done under the list \texttt{Case\_V\_a}.

If there is a curve $\mu \neq \gamma$ that passes through petals only and meets $\beta$, then the curve obtained by composing $\beta, \gamma, \mu$ has length $p+q+r \leq n$, so by \Cref{lemma:leq2curves} we have $\lambda^n \geq 16$.
Thus we can assume from this point that any curve other than $\gamma$ that meets $\beta$ must pass through filaments.

Since $\Gamma$ is strongly connected, there must be a curve $\mu$ meeting $\beta$ and some other petal or filament curve. 
In fact, by our assumption in the previous paragraph, we know that $\mu$ must meet some filament curve.

\begin{lemma} \label{lemma:caseVImumeetbeta}
If $\mu$ meets some petal curve other than $\beta$, then $\lambda^n \geq 17$.
\end{lemma}
\begin{proof}
If $\mu$ meets $\gamma$, then we compute that $\lambda^n \geq 16$.
If $\mu$ does not meet $\gamma$, then we compute that $\lambda^n \geq 18$.
\end{proof}

Thus we can assume from this point that $\mu$ does not meet any petal curves other than $\beta$.

\begin{lemma}
If there is a petal curve other than $\beta$, then $\lambda^n \geq 15.1$.
\end{lemma}
\begin{proof}
Since $\Gamma$ is strongly connected, there is a curve $\nu$ meeting a petal curve other than $\beta$, and ($\beta$ or some filament curve).
We can assume that $\nu$ does not pass through $\beta$, since otherwise we could have taken $\mu$ to be $\nu$ and be done by \Cref{lemma:caseVImumeetbeta}.
Thus $\nu$ passes through the filaments.

We then separate into multiple cases depending on whether $\mu$, $\nu$, and $\gamma$ meet, and compute that $\lambda^n \geq 15.1$.
In some of the computations, we have to use the inequality $q \leq p$. This is where we use the fact that $\gamma$ passes through the vertices of $\beta$ only.
See \texttt{Case\_V\_a} for details.
\end{proof}

This concludes the proof of \Cref{prop:caseVasolopetal}. With it, we can now prove \Cref{prop:caseV} in this subcase.

\begin{proof}[Proof of \Cref{prop:caseV} in case Va]
By \Cref{prop:caseVasolopetal}, we can assume that $\beta$ is the only petal curve and every curve other than $\beta$ and $\gamma$ passes through filaments.

Recall \Cref{lemma:secondarycycles}. First suppose that $b'$ is nonsingular. Let $\beta'$ be a cycle determined by $b'$. 
The length of $\beta'$ is $\leq p$, so $\beta'$ cannot contain $\beta$. In particular $\beta'$ must contain some curve other than $\beta$ and $\gamma$. We replace $\beta'$ with this curve. 

By assumption, $\beta'$ has to pass through filaments.
Hence we can conclude in this setting by applying \Cref{prop:caseII} or \Cref{prop:caseIII} to $\beta'$ and $\gamma$.

A similar reasoning, now using the fact that $q \leq p$ to show that $\gamma'$ cannot contain $\beta$, and applying \Cref{prop:caseIVsolopetal} to $\beta$ and $\gamma'$, shows that if $c'$ is nonsingular then we are done. 
Hence we can assume that both $b'$ and $c'$ are singular. 

Let $\beta'$ and $\gamma'$ be cycles determined by $b'$ and $c'$ which pass through petals. These exist by definition of petals. Let $p'$ and $q'$ be the lengths of $\beta'$ and $\gamma'$ respectively. 
\Cref{lemma:accounting+}(2a) states that $p+q+\frac{1}{3}p'+\frac{1}{3}q' \leq n$.

Each of $\beta'$ and $\gamma'$ either contains a curve other than $\beta$ and $\gamma$, or is a concatenation of nontrivial powers of $\beta$ and $\gamma$. But the latter case cannot occur by \Cref{cor:inneroutercurveemb}, since either $b$ or $c$ lies in $\mathcal{X}_O$.
Thus $\beta'$ and $\gamma'$ must each contain a curve that passes through both filaments and petals.
By replacing $\beta'$ and $\gamma'$ by these curves, we can assume that they are embedded.

At this point, we can use the inequalities $m \leq n$ and $p+q+\frac{1}{3}p'+\frac{1}{3}q' \leq n$ to compute that $\lambda^n \geq 29.2$.
\end{proof}

\subsubsection{Case Vb: $\gamma$ meets a petal curve $\delta \neq \beta$ and no other petal curve}

We first prove the following proposition, which restricts the possibilities for the remaining petal curves.

\begin{prop} \label{prop:caseVbsolopetal}
Fix some system of petal curves. Suppose there are two cycles $\beta$ and $\gamma$ in the directed graph $\Gamma$ satisfying:
\begin{itemize}
    \item $\len(\beta) + \len(\gamma) + r_0 \leq n$,
    \item $\beta$ and $\gamma$ passes through petals only,
    \item $\beta$ is a petal curve,
    \item $\gamma$ is not a petal curve, and
    \item $\gamma$ meets a petal curve $\delta \neq \beta$ and no other petal curve.
\end{itemize}
Then either $\lambda^n \geq 14.5$ or 
\begin{itemize}
    \item $\beta$ and $\delta$ are the only petal curves, and
    \item every curve other than $\beta$, $\gamma$, and $\delta$ passes through filaments.
\end{itemize}
\end{prop}

Up to replacing $\gamma$ with a curve that it contains, we can assume that it is embedded.
We denote the lengths of $\beta$ and $\gamma$ by $p$ and $q$ respectively.
Throughout this subsubsection, we will denote the length of $\delta$ by $s$.

We first rule out certain configurations of curves that pass through petals only.

If there is a curve $\mu$ that passes through petals only and which meets $\beta$, then $\beta$ is non-persistent. By applying the reasoning below \Cref{prop:caseVbothnonpetalcurve} (with the roles of $\beta$ and $\gamma$ switched), we reduce to \Cref{prop:caseVasolopetal}. 
Thus we can assume from this point that every curve that passes through petals only does not meet $\beta$.

\begin{lemma} \label{lemma:caseVbmupetalgamma}
If there is a curve $\mu$ that passes through petals only and which meets $\gamma$, then $\lambda^n \geq 14.5$.
\end{lemma}
\begin{proof}
By our assumption in the paragraph before the lemma, $\mu$ does not meet $\beta$.
Since $\Gamma$ is strongly connected, there must be a curve $\nu$ meeting some petal curve other than $\beta$, and ($\beta$ or some filament curve).
Again by our assumption, $\nu$ must pass through filaments.

We first tackle the case where $\nu$ meets $\beta$. 
We separate into multiple cases depending on whether $\nu$ meets $\gamma$ and $\mu$, and compute that $\lambda^n \geq 14.5$.

We then tackle the case where $\nu$ does not meet $\beta$. Since $\Gamma$ is strongly connected, there is a curve $\omega$ meeting $\beta$ and (some petal curve other than $\beta$, or some filament curve).
Using our assumption as in the first paragraph, $\omega$ must pass through filaments.
Furthermore, we can assume that $\omega$ does not meet petal curves other than $\beta$, since otherwise we could have taken $\nu$ to be $\omega$ and run the computations in the previous paragraph. 
We can even assume that $\nu$ and $\omega$ meet, since otherwise we can sum up their corresponding vertices in the curve complex to reduce to computations in the previous paragraph as well.

We now separate into multiple cases depending on whether $\nu$ meets $\gamma$ and $\mu$. In all cases except the one where $\nu$ meets $\mu$ but not $\gamma$, we have $\lambda^n \geq 14.5$.

In this remaining case, we have to work harder.
By \Cref{prop:floralttpetalexit}, there is a curve $\chi$ that exits $\beta$ from the same vertex as $\omega$.
By assumption, $\chi$ passes through filaments, and by the reasoning above for $\omega$, we can assume that $\chi$ does not meet petal curves other than $\beta$. In this case, we compute that $\lambda^n \geq 17.4$.

All computations in this lemma are done under the list \texttt{Lemma\_A\_7\_7}.
\end{proof}

Thus we can assume from this point that every curve that passes through petals only does not meet $\gamma$.

We now focus on showing that there are no petal curves other than $\beta$ and $\delta$.

\begin{lemma} \label{lemma:caseVbmupetaldeltaother}
If there is a curve $\mu$ that passes through petals only and which meets $\delta$ and a petal curve other than $\beta$ and $\delta$, then $\lambda^n \geq 18.4$.
\end{lemma}
\begin{proof}
Since $\Gamma$ is strongly connected, there must be a curve $\nu$ meeting some petal curve other than $\beta$, and ($\beta$ or some filament curve).
By assumption, $\nu$ pass through filaments.

We first tackle the case where $\nu$ meets $\beta$. 
We separate into multiple cases depending on whether $\nu$ meets $\gamma$, $\delta$, and other petal curves, and compute that $\lambda^n \geq 18.4$.

We then tackle the case where $\nu$ does not meet $\beta$. Since $\Gamma$ is strongly connected, there is a curve $\omega$ meeting $\beta$ and (some petal curve other than $\beta$, or some filament curve).
By assumption, $\nu$ passes through filaments.
By the case in the previous paragraph, we can assume that $\omega$ does not meet petal curves other than $\beta$. 
We can also assume that $\nu$ and $\omega$ meet, since otherwise we can sum up their vertices to reduce to the previous case as well.

We now separate into multiple cases depending on whether $\nu$ meets $\gamma$, $\delta$, and other petal curves, and compute that $\lambda^n \geq 18.5$.

All computations in this lemma are done under the list \texttt{Lemma\_A\_7\_8}.
\end{proof}

Thus we can assume from this point that every curve that passes through petals only and which meets a petal curve other than $\beta$ and $\delta$ does not meet $\delta$ (nor, as previously assumed, $\beta$).

By \Cref{prop:floralttpetalexit}, there is a curve $\mu$ that exits $\delta$ from the same vertex as $\gamma$.
By assumption, $\mu$ passes through filaments.

\begin{lemma} \label{lemma:caseVbmuother}
If $\mu$ passes through a petal curve other than $\beta$ and $\delta$, then $\lambda^n \geq 14.9$.
\end{lemma}
\begin{proof}
If $\mu$ passes through $\beta$, then we compute that $\lambda^n \geq 17.8$.

Thus we can assume that $\mu$ does not pass through $\beta$. Since $\Gamma$ is strongly connected, there is a curve $\nu$ meeting $\beta$ and (some petal curve other than $\beta$, or some filament curve).
By assumption, $\nu$ passes through filaments.
We may also assume that $\mu$ and $\nu$ meet, since otherwise in our computations, we can sum up their vertices to reduce to the previous case.

We separate into multiple cases depending on whether $\nu$ meets $\gamma$, $\delta$, and other petal curves, and compute that $\lambda^n \geq 14.9$.

All computations in this lemma are done under the list \texttt{Lemma\_A\_7\_9}.
\end{proof}

\begin{lemma} \label{lemma:caseVbmubetaother}
If $\mu$ passes through $\beta$ and there is a petal curve other than $\beta$ and $\delta$, then $\lambda^n \geq 15.4$.
\end{lemma}
\begin{proof} 
By \Cref{lemma:caseVbmuother}, we can assume that $\mu$ does not pass through any petal curves other than $\beta$ and $\delta$.
Since $\Gamma$ is strongly connected, there is a curve $\nu$ meeting a petal curve other than ($\beta$ and $\delta$), and ($\beta$ or $\delta$ or some filament curve).
By assumption, $\nu$ passes through filaments. 

We separate into multiple cases depending on whether $\nu$ meets $\beta$, $\gamma$, and $\delta$, and compute that $\lambda^n \geq 15.4$.

All computations in this lemma are done under the list \texttt{Lemma\_A\_7\_10}.
\end{proof}

\begin{lemma}
If $\mu$ does not pass through a petal curve other than $\delta$ and there is a petal curve other than $\beta$ and $\delta$, then $\lambda^n \geq 15.4$.
\end{lemma}
\begin{proof}
Since $\Gamma$ is strongly connected, there is a curve meeting $\beta$ and (some other petal curve or some filament curve), and there is a curve meeting a petal curve other than ($\beta$ and $\delta$), and ($\beta$ or $\delta$ or some filament curve).

We first suppose that there is a single curve $\nu$ meeting $\beta$, a petal curve other than ($\beta$ and $\delta$), and ($\delta$ or some filament curve).
By assumption, $\nu$ passes through the filaments. 
We separate into multiple cases depending on whether $\nu$ meets $\gamma$ and $\delta$, and compute that $\lambda^n \geq 15.7$.

We then suppose that there a curve $\nu$ meeting $\beta$, and ($\delta$ or some filament curve), and there is a curve $\omega$ meeting a petal curve other than ($\beta$ and $\delta$), and ($\delta$ or some filament curve).
By assumption, $\nu$ and $\omega$ pass through filaments. 
We separate into multiple cases depending on whether $\nu$ and $\omega$ meets $\gamma$ and $\delta$. In all cases except the one where both $\nu$ and $\omega$ meet $\delta$ but not $\gamma$, we have $\lambda^n \geq 15.4$.

In the remaining case, we have to work harder. 
By \Cref{prop:floralttpetalexit}, there is a curve $\chi$ exiting $\beta$ from the same vertex as $\nu$ and a curve $\upsilon$ exiting a petal curve other than $\beta$ and $\delta$ from the same vertex as $\omega$. By previous computations, we can assume that both $\chi$ and $\upsilon$ meet $\delta$ but not $\gamma$.
In this case, we have $\lambda^n \geq 17.4$.

All computations in this lemma are done under the list \texttt{Lemma\_A\_7\_11}.
\end{proof}

Thus we can assume from this point that there are no petal curves other than $\beta$ and $\delta$.
To show \Cref{prop:caseVbsolopetal}, it remains to show that any curve other than $\beta$, $\gamma$, and $\delta$ passes through filaments.

\begin{lemma}
If there is a curve $\nu$ other than $\beta$, $\gamma$, and $\delta$ that passes through petals only, then $\lambda^n \geq 14.8$.
\end{lemma}
\begin{proof}
By assumption, $\nu$ meets $\delta$ but not $\beta$ and $\gamma$.

Recall the curve $\mu$. If $\mu$ meets $\beta$ but not $\nu$, then $\lambda^n \geq 17.8$.

If $\mu$ meets $\beta$ and $\nu$, let $\omega$ be a curve exiting $\beta$ at the same vertex as $\mu$. Then $\lambda^n \geq 16$.

If $\mu$ does not meet $\beta$ and does not meet $\nu$, let $\omega$ be a curve meeting $\beta$ and some filament curve. Then $\lambda^n \geq 15$.

If $\mu$ does not meet $\beta$ but meets $\nu$, let $\omega$ and $\chi$ be curves meeting $\beta$ and some filaments curves, which exit $\beta$ at the same vertex. Then $\lambda^n \geq 14.6$.

All computations in this lemma are done under the list \texttt{Lemma\_A\_7\_12}.
\end{proof}

This concludes the proof of \Cref{prop:caseVbsolopetal}. With it, we can now prove \Cref{prop:caseV} in this subcase.

\begin{proof}[Proof of \Cref{prop:caseV} in case Vb]
By \Cref{prop:caseVbsolopetal}, we can assume that $\beta$ and $\delta$ are the only petal curves and every curve other than $\beta$, $\gamma$, and $\delta$ passes through filaments.

Recall \Cref{lemma:secondarycycles}. First suppose that $b'$ is nonsingular. Let $\beta'$ be a cycle determined by $b'$. 
The length of $\beta'$ is $\leq p$, so $\beta'$ cannot contain $\beta$.
If $\beta'$ contains $\delta$, then by \Cref{prop:caseVasolopetal}, we have $\lambda^n \geq 15.1$.
Hence we can assume that $\beta'$ contains some curve other than $\beta$, $\gamma$, and $\delta$. We replace $\beta'$ with this curve. 
By assumption, $\beta'$ passes through filaments, so we can conclude by applying \Cref{prop:caseII} to $\beta'$ and $\gamma$.

Hence we can assume that $b'$ is singular. Let $\beta'$ be a cycle determined by $b'$ which passes through petals. Let $p'$ be the length of $\beta'$. 

We first suppose that $\beta'$ passes through filaments.
By picking $\kappa = 1$ in \Cref{lemma:accounting+}(1b), we get
\begin{equation} \label{eq:caseVbbeta'betadelta}
p+q+s+\frac{1}{3}p' = q + (p+s) + \frac{1}{3}p' \leq n.
\end{equation}
By picking $\kappa = 0$, we get
\begin{equation} \label{eq:caseVbbeta'delta}
q+s+\frac{2}{3}p' = -p + q + (p+s) + \frac{2}{3}p' \leq n.
\end{equation}

Up to replacing $\beta'$ by a curve that it contains, we can assume that it is embedded. 
If $\beta'$ meets both $\beta$ and $\delta$, then using \Cref{eq:caseVbbeta'betadelta}, we compute that $\lambda^n \geq 22$.

If $\beta'$ meets $\delta$ but not $\beta$, then using \Cref{eq:caseVbbeta'delta}, we compute that $\lambda^n \geq 20$.

If $\beta'$ meets $\beta$ but not $\delta$, we let $\nu$ be a curve exiting $\delta$ at the same vertex as $\gamma$. 
If $\nu$ meets $\beta$, then using \Cref{eq:caseVbbeta'betadelta}, we compute that $\lambda^n \geq 19.2$.

If $\nu$ does not meet $\beta$, then using \Cref{eq:caseVbbeta'betadelta}, we compute that $\lambda^n \geq 16.9$.

Hence we can assume that $\beta'$ does not pass through filaments. In this case, $\beta'$ must be a concatenation of nontrivial powers of $\gamma$ and $\delta$. In fact, by \Cref{prop:innerouteremb} it could only be the concatenation of one copy of $\gamma$ and one copy of $\delta$.
Thus $p'= q+s$.

We turn our attention to $c'$.  
First suppose that $c'$ is nonsingular. Let $\gamma'$ be a cycle determined by $c'$.
If $\gamma'$ passes through filaments, we can conclude by applying \Cref{prop:caseII} to $\beta$ and $\gamma'$.
Hence we can assume that $\gamma'$ passes through petals only. Since $\len(\gamma') \leq q$, we must have $\gamma' = \delta$.

By picking $\kappa = \frac{4}{5}$ in \Cref{lemma:accounting+}(1b), we get
\begin{equation} \label{eq:caseVbgamma'=delta}
\frac{4}{5}p+\frac{7}{5}q+\frac{7}{5}s = -\frac{1}{5}p + q + (p+s) + \frac{2}{5}(q+s) \leq n.
\end{equation}

Let $\mu$ be a curve that exits $\delta$ at the same vertex as $\gamma$. 
If $\mu$ meets $\beta$, let $\nu$ be a curve that exits $\beta$ at the same vertex as $\mu$. Using \Cref{eq:caseVbgamma'=delta}, we compute that $\lambda^n \geq 14.7$.

If $\mu$ does not meet $\beta$, there exists curves $\nu$ and $\omega$ that meet $\beta$, and (other petal curves or filament curves). Using \Cref{eq:caseVbgamma'=delta}, we compute that $\lambda^n \geq 14.5$.

Thus we can assume that $c'$ is singular. 
Let $\gamma'$ be a cycle determined by $c'$ which passes through petals. Let $q'$ be the length of $\gamma'$.
If $\gamma'$ does not pass through filaments, then by the same reasoning we used for $\beta'$, $\gamma'$ must be the concatenation of one copy of $\gamma$ and one copy of $\delta$. But then this contradicts \Cref{prop:innerouteremb}.
Hence we can assume that $\gamma'$ passes through filaments.

Recall that $p'=q+s$.
By picking $\kappa_1 = \frac{1}{2}$ and $\kappa_2 = \frac{5}{4}$ in \Cref{lemma:accounting+}(2b), we get
\begin{equation} \label{eq:caseVbgamma'betadelta}
\frac{1}{2}p+\frac{3}{4}q+\frac{3}{2}s+\frac{1}{4}q' = -\frac{1}{2}p + \frac{1}{4}q + (p+s) + \frac{1}{2}(q+s) + \frac{1}{4}q' \leq n.
\end{equation}
By picking $\kappa = 0$ in \Cref{lemma:accounting+}(1b), we get
\begin{equation} \label{eq:caseVbgamma'delta}
q+s+\frac{2}{3}p' = -p + q + (p+s) + \frac{2}{3}p' \leq n.
\end{equation}

Moreover, we claim that we can assume $q \leq s$. This is true if $\gamma$ passes through the vertices in $\delta$ only, which is in turn automatically true if every petal lies in one petal curve in our chosen system. 
So let us suppose that $\gamma$ is contained in every system for which every petal lies in one petal curve.

Meanwhile, if $\delta$ passes through the vertices in $\gamma$ only, then we can swap the roles of $\gamma$ and $\delta$. So let us suppose that $\delta$ contains vertices that lie outside of $\gamma$ as well.
Now let us fix a system for which every petal lies in one petal curve. Let $\mu$ be a petal curve that contains a vertex in $\delta$ but not in $\gamma$. This curve $\mu$ will contradict our assumption that every curve other than $\beta$, $\gamma$, and $\delta$ passes through filaments.

Armed with these inequalities, we move on to the final computations in this proof. 

Let $\mu$ be a curve in $\gamma'$ that passes through both filaments and petals. 
If $\mu$ meets both $\beta$ and $\delta$, then using \Cref{eq:caseVbgamma'betadelta}, we compute that $\lambda^n \geq 16.7$.

If $\mu$ meets $\delta$ but not $\beta$, then using \Cref{eq:caseVbgamma'delta}, we compute that $\lambda^n \geq 20.1$.

If $\mu$ meets $\beta$ but not $\delta$, we let $\nu$ be a curve exiting $\delta$ at the same vertex as $\gamma$. 
If $\nu$ meets $\beta$, then using \Cref{eq:caseVbgamma'betadelta} and $q \leq s$, we compute that $\lambda^n \geq 15$.

If $\nu$ does not meet $\beta$, then using \Cref{eq:caseVbgamma'betadelta} and $q \leq s$, we compute that $\lambda^n \geq 14.9$.

All computations in this proof are done under the list \texttt{Case\_V\_b}.
\end{proof}

\subsection{Case VI: Both $\beta$ and $\gamma$ pass through petals only, and both are persistent} \label{sec:caseVI}

The goal of this subsection is to prove the following proposition.

\begin{prop} \label{prop:caseVI}
Define $\beta$ and $\gamma$ as in \Cref{subsec:casedivision}. Suppose 
\begin{itemize}
    \item both $\beta$ and $\gamma$ pass through petals only, and
    \item both $\beta$ and $\gamma$ are persistent.
\end{itemize}
Then $\lambda^n \geq 14.5$.
\end{prop}

We first prove the following proposition, which restricts the possibilities for the remaining petal curves.

\begin{prop} \label{prop:caseVIsolopetal}
Fix some system of petal curves. Suppose there are two cycles $\beta$ and $\gamma$ in the directed graph $\Gamma$ satisfying:
\begin{itemize}
    \item $\len(\beta) + \len(\gamma) + r_0 \leq n$,
    \item $\beta$ and $\gamma$ passes through petals only, and
    \item $\beta$ and $\gamma$ are petal curves.
\end{itemize}
Then either $\lambda^n \geq 14.5$, or:
\begin{itemize}
    \item There is at most one other petal curve $\delta$.
    \item If there is exactly one other petal curve $\delta$, then every curve other than $\beta$, $\gamma$, and $\delta$ either passes through filaments or only meets $\delta$ and not $\beta$ and $\gamma$.
    \item If there is exactly one other petal curve $\delta$ and there is a curve $\eta$ other than $\beta$, $\gamma$, and $\delta$ that passes through petals only, then 
    \begin{itemize}
        \item there is a curve meeting $\beta$ and $\eta$, but not $\gamma$, and
        \item there is a curve meeting $\gamma$ and $\eta$, but not $\beta$. 
    \end{itemize}
\end{itemize}
\end{prop}

If there are no petal curves other than $\beta$ and $\gamma$, then we are done. So we assume that there is some petal curve $\delta$ other than $\beta$ and $\gamma$.
In this case, by \Cref{prop:caseVbothnonpetalcurve} and \Cref{prop:caseVbsolopetal}, we can assume that every curve that meets $\beta$ or $\gamma$ passes through filaments.

\begin{lemma} \label{lemma:caseVInonpetalmultiplepetal}
If there is some curve $\delta$ that passes through petals only and meets multiple petal curves, then $\lambda^n \geq 15$.
\end{lemma}
\begin{proof}
Since $\Gamma$ is strongly connected, there is a curve $\mu$ that meets $\beta$. Let $\nu$ be a curve that exits $\beta$ at the same vertex as $\mu$. By our assumption, $\mu$ and $\nu$ pass through filaments. 

If $\mu$ and $\nu$ meet $\gamma$, then we have $\lambda^n \geq 15.1$.

If $\mu$ meets $\gamma$ while $\nu$ does not meet $\gamma$, then we have $\lambda^n \geq 15$.

It remains to analyze the case when $\mu$ and $\nu$ do not meet $\gamma$. In this case, symmetrically, we may assume that there are two curves $\omega$ and $\chi$ that meet $\gamma$ but not $\beta$. In this case, we have $\lambda^n \geq 15.2$.

All computations in this lemma are done under the list \texttt{Lemma\_A\_8\_3}.
\end{proof}

Thus we can assume from this point that any curve that passes through petals only but is not a petal curve only meets one petal curve, which must be one other than $\beta$ and $\gamma$.

The first item of \Cref{prop:caseVIsolopetal} follows from the lemma below.

\begin{lemma}
If there are multiple petal curves other than $\beta$ and $\gamma$, then $\lambda^n \geq 14.5$.
\end{lemma}
\begin{proof}
Suppose there are two petal curves $\delta$ and $\eta$ other than $\beta$ and $\gamma$. 

If both $\delta$ and $\eta$ meet curves that pass through petals only, then by \Cref{lemma:caseVInonpetalmultiplepetal}, we can assume that $\delta$ meets a curve $\epsilon$ that passes through petals only and does not meet petal curves other than $\delta$, and $\eta$ meets a curve $\theta$ that passes through petals only and does not meet petal curves other than $\eta$.
In this case, we compute that $\lambda^n \geq 14.5$. 
These computations are done under the list \texttt{Lemma\_A\_8\_4\_1}.

If $\eta$ meets curves that pass through petals only while $\delta$ does not, then by \Cref{lemma:caseVInonpetalmultiplepetal}, we can assume that $\eta$ meets a curve $\theta$ that passes through petals only and does not meet petal curves other than $\eta$.
In this case, we compute that $\lambda^n \geq 14.5$.
These computations are done under the lists \texttt{Lemma\_A\_8\_4\_2}, \texttt{Lemma\_A\_8\_4\_3}, and \texttt{Lemma\_A\_8\_4\_4}.

If both $\delta$ and $\eta$ do not meet curves that pass through petals only, then $\lambda^n \geq 14.5$. These computations are done under the lists \texttt{Lemma\_A\_8\_4\_5}, \texttt{Lemma\_A\_8\_4\_6}, \texttt{Lemma\_A\_8\_4\_7}, and \texttt{Lemma\_A\_8\_4\_8}.
\end{proof}

The second item of \Cref{prop:caseVIsolopetal} follows from our assumption on curves that pass through petals only. We show the third item of \Cref{prop:caseVIsolopetal} below.

\begin{lemma}
If there is exactly one other petal curve $\delta$ and there is a curve $\eta$ which passes through petals only, then either $\lambda^n \geq 14.6$, or 
\begin{itemize}
    \item there is a curve meeting $\beta$ and $\eta$, but not $\gamma$, and
    \item there is a curve meeting $\gamma$ and $\eta$, but not $\beta$.
\end{itemize}
\end{lemma}
\begin{proof}
By considering curves that meet $\beta$ and $\gamma$ as above, we compute that $\lambda^n \geq 14.6$. These computations are done under the lists \texttt{Lemma\_A\_8\_5\_1} and \texttt{Lemma\_A\_8\_5\_2}.
\end{proof}

This concludes the proof of \Cref{prop:caseVIsolopetal}.

For the rest of this subsection, we specialize to the canonical system of petal curves. To show \Cref{prop:caseVI}, we will divide into three subcases according to \Cref{lemma:secondarycycles}:
\begin{itemize}
    \item[(VIa)] $b'$ and $c'$ are nonsingular.
    \item[(VIb)] One of $b'$ and $c'$ is nonsingular and the other is singular.
    \item[(VIc)] $b'$ and $c'$ are singular.
\end{itemize}

\subsubsection{Case VIa: $b'$ and $c'$ are nonsingular}

Let $\beta'$ be a cycle determined by $b'$. If $\beta'$ passes through filaments, then we can conclude by case II (\Cref{prop:caseII}) or \Cref{prop:caseIVsolopetal}. 
If $\beta'$ is not a petal curve, then we can conclude by \Cref{prop:caseVbothnonpetalcurve}, \Cref{prop:caseVasolopetal}, or \Cref{prop:caseVbsolopetal}.
So we can assume that $\beta'$ is a petal curve.

Similarly, let $\gamma'$ be a cycle determined by $c'$. We can assume that $\gamma'$ is a petal curve.
In other words, we have 4 distinct petal curves $\beta$, $\gamma$, $\beta'$, and $\gamma'$, so we are done by \Cref{prop:caseVIsolopetal}.

\subsubsection{Case VIb: One of $b'$ and $c'$ is nonsingular and the other is singular}

Here $\beta$ and $\gamma$ will play a symmetric role, so we suppose that $b'$ is nonsingular and $c'$ is singular.

Let $\beta'$ be a cycle determined by $b'$. Arguing as in case VIa, we can assume that $\beta'$ is a petal curve. 
By \Cref{prop:caseVIsolopetal}, we can assume that $\beta'$ is the only petal curve other than $\beta$ and $\gamma$.
Let $p'$ be the length of $\beta'$.

Let $\gamma'$ be a cycle determined by $c'$ which passes through petals. Up to replacing $\gamma'$ by a curve that it contains, we can assume that $\gamma'$ is embedded. 
Let $q'$ be the length of $\gamma'$.

\begin{lemma}
If $\gamma'$ passes through petals only, then $\lambda^n \geq 14.5$.
\end{lemma}
\begin{proof}
By \Cref{prop:caseVIsolopetal}, we can assume that $\gamma'$ meets $\beta'$ but not $\beta$ and $\gamma$, and that there is a curve $\mu$ meeting $\beta$, $\beta'$, but not $\gamma$, and a curve $\nu$ meeting $\gamma$, $\beta'$, but not $\beta$.

We consider curves that exit $\beta$ at the same vertex as $\mu$, or exit $\gamma$ at the same vertex as $\nu$. Using the fact that $p' \leq p$, we compute that $\lambda^n \geq 15.6$.

All computations in this lemma are done under the list \texttt{Lemma\_A\_8\_6}.
\end{proof}

\begin{lemma} \label{lemma:caseVIbgamma'filament}
If $\gamma'$ passes through filaments, then $\lambda^n \geq 15.7$.
\end{lemma}
\begin{proof}
By choosing $\kappa = \frac{1}{2}$ in \Cref{lemma:accounting+}(1b) (with the roles of $b$ and $c$ reversed), we have
\begin{equation} \label{eq:caseVIbonesing}
2p+\frac{1}{2}q+p'+\frac{1}{2}q' \leq p-\frac{1}{2}q+(p+q+p')+\frac{1}{2}q' \leq n.
\end{equation}

If $\gamma'$ meets at least two of the three petal curves $\beta$, $\gamma$, and $\beta'$, then, using \Cref{eq:caseVIbonesing} and $p' \leq p$, we compute that $\lambda^n \geq 15.9$.

If $\gamma'$ only meets one of the three petal curves, then we compute that $\lambda^n \geq 15.1$.

All computations in this lemma are done under the list \texttt{Lemma\_A\_8\_7}.
\end{proof}

\subsubsection{Case VIc: $b'$ and $c'$ are singular}

Let $\beta'$ and $\gamma'$ be cycles determined by $b'$ and $c'$ which pass through petals. By \Cref{cor:inneroutercurveemb}, up to replacing these by curves that they contain, we can assume that they are embedded. Let $p'$ and $q'$ be the lengths of $\beta'$ and $\gamma'$ respectively. 

\begin{lemma}
If $\beta'$ and $\gamma'$ do not meet filaments, then $\lambda^n \geq 14.5$.
\end{lemma}
\begin{proof}
By \Cref{prop:caseVIsolopetal}, we can assume that
\begin{itemize}
    \item there is at most one other petal curve $\delta$, and
    \item if there is one other petal curve $\delta$, then every curve other than $\beta$, $\gamma$, and $\delta$ either passes through filaments or only meets $\delta$ and not $\beta$ and $\gamma$.
\end{itemize}

Thus, up to switching $\beta'$ and $\gamma'$, there are two cases:
\begin{enumerate}
    \item $\beta'$ and $\gamma'$ are both distinct from $\delta$, or
    \item $\beta'$ equals $\delta$ while $\gamma'$ is distinct from $\delta$.
\end{enumerate}

In case (1), we let $s$ be the length of $\delta$. We take $\kappa_1=\kappa_2=1$ in \Cref{lemma:accounting+}(2b) to obtain
$$p+q+s+\frac{1}{3}p'+\frac{1}{3}q' \leq r + \frac{1}{3}p' + \frac{1}{3}q' \leq n.$$
Meanwhile, by the second statement in \Cref{prop:caseVIsolopetal}, we can assume that
\begin{itemize}
    \item there is a curve $\mu$ meeting $\beta$, $\beta'$, but not $\gamma$, and
    \item there is a curve $\nu$ meeting $\gamma$, $\beta'$, but not $\beta$.
\end{itemize}
In this case, we compute that $\lambda^n \geq 15.7$.

In case (2), we take $\kappa_1=\kappa_2=1$ in \Cref{lemma:accounting+}(2b) to obtain
$$p+q+\frac{4}{3}p'+\frac{1}{3}q' \leq r + \frac{1}{3}p' + \frac{1}{3}q' \leq n.$$
We apply the second statement in \Cref{prop:caseVIsolopetal} as in case (1). 
By considering curves that exit $\beta$ at the same vertex as $\mu$, or exit $\gamma$ at the same vertex as $\nu$, and using the fact that $q' \leq p'$, we compute that $\lambda^n \geq 14.5$.

All computations in this lemma are done under the list \texttt{Lemma\_A\_8\_8}.
\end{proof}

\begin{lemma}
If $\beta'$ does not meet filaments while $\gamma'$ does, then $\lambda^n \geq 15.7$.
\end{lemma}
\begin{proof}
The computations in this case are exactly the same as \Cref{lemma:caseVIbgamma'filament}, since in that lemma we did not appeal to any relation between $p'$ and $p$.
\end{proof}

\begin{lemma}
If $\beta'$ and $\gamma'$ meet filaments, then $\lambda^n \geq 14.5$.
\end{lemma}
\begin{proof}
By \Cref{prop:caseVIsolopetal}, we can assume that there is at most one other petal curve $\delta$. If $\delta$ exists, we let $s$ be its length, otherwise we set $s=0$.
In \Cref{lemma:accounting+}(2b): 
\begin{itemize}
    \item If we take $\kappa_1 = \kappa_2 = \frac{1}{2}$, we have
    \begin{equation} \label{eq:caseVIcbeta'gamma'filament1}
    \frac{1}{2}p+\frac{1}{2}q+s+\frac{1}{2}p'+\frac{1}{2}q' \leq -\frac{1}{2}p-\frac{1}{2}q+(p+q+s)+\frac{1}{2}p'+\frac{1}{2}q' \leq n.
    \end{equation}
    \item If we take $\kappa_1 = \frac{1}{2}$ and $\kappa_2 = 1$, we have
    \begin{equation} \label{eq:caseVIcbeta'gamma'filament2}
    \frac{1}{2}p+q+s+\frac{1}{2}p'+\frac{1}{3}q' \leq -\frac{1}{2}p+(p+q+s)+\frac{1}{2}p'+\frac{1}{3}q' \leq n.
    \end{equation}
    \item If we take $\kappa_1 = \kappa_2 = 1$, we have
    \begin{equation} \label{eq:caseVIcbeta'gamma'filament3}
    p+q+s+\frac{1}{3}p'+\frac{1}{3}q' \leq n.
    \end{equation}
\end{itemize}

We first suppose that $\beta'$ meets both $\beta$ and $\gamma$.
In this case, using \Cref{eq:caseVIcbeta'gamma'filament1}, we compute that $\lambda^n \geq 17.4$.

Then we suppose that $\delta$ exists and $\beta'$ meets $\delta$. 
In this case, using \Cref{eq:caseVIcbeta'gamma'filament1}, we compute that $\lambda^n \geq 15.7$.

Thus we can assume from this point that $\beta'$ meets exactly one of $\beta$ and $\gamma$, and does not meet $\delta$.
Similarly, we can assume from this point that $\gamma'$ meets exactly one of $\beta$ and $\gamma$, and does not meet $\delta$.

Suppose $\beta'$ and $\gamma'$ both meet $\gamma$.
By assumption, $\beta'$ and $\gamma'$ do not meet $\beta$. 
Since $\Gamma$ is strongly connected, there is some curve $\mu$ that meets $\beta$.

If $\mu$ meets $\gamma$, then using \Cref{eq:caseVIcbeta'gamma'filament2}, we compute that $\lambda^n \geq 16.1$.

If $\mu$ does not meet $\gamma$, then using \Cref{eq:caseVIcbeta'gamma'filament3}, we compute that $\lambda^n \geq 14.5$.

Thus we can assume that at most one of $\beta'$ and $\gamma'$ meet $\gamma$. Similarly, we can assume that at most one of $\beta'$ and $\gamma'$ meet $\beta$. Up to switching $\beta'$ and $\gamma'$, suppose that the only petal curve which $\beta'$ meets is $\beta$, and the only petal curve which $\gamma'$ meets is $\gamma$.

If $\beta'$ and $\gamma'$ meet, then using \Cref{eq:caseVIcbeta'gamma'filament1}, we compute that $\lambda^n \geq 17$. Thus we can assume that $\beta'$ and $\gamma'$ are disjoint.

We let $\mu$ be a curve that exits $\beta$ at the same vertex as $\beta'$.
If $\mu$ meets $\gamma$, then using \Cref{eq:caseVIcbeta'gamma'filament3}, we compute that $\lambda^n \geq 16$.

If $\mu$ does not meet $\gamma$, then we let $\nu$ be a curve that exits $\gamma$ at the same vertex as $\gamma'$.
Using \Cref{eq:caseVIcbeta'gamma'filament3}, we compute that $\lambda^n \geq 17.3$.

All computations in this lemma are done under the lists \texttt{Lemma\_A\_8\_10\_1} and \texttt{Lemma\_A\_8\_10\_2}.
\end{proof}

\bibliographystyle{alphaurl}

\bibliography{bib.bib}

\end{document}